\documentclass[11pt,reqno,twoside]{article}

\usepackage[hang]{footmisc}
\usepackage{lipsum}

\usepackage{cmap} 

\usepackage[T1]{fontenc}
\usepackage[utf8]{inputenc}
\usepackage{graphicx}
\usepackage{placeins}
\usepackage{enumerate}
\usepackage{algcompatible}
\usepackage{subcaption}
\usepackage{booktabs}
\usepackage{dblfloatfix}  
\usepackage{verbatim}
\newcommand{\comments}[1]{}

\usepackage{comment}
\excludecomment{supplementaryexperiments}

\usepackage{soul}

\usepackage{setspace}
\usepackage{dblfloatfix}
\usepackage{subcaption}
\usepackage{graphicx}
\usepackage{caption}   
\usepackage{cuted}    
\usepackage{placeins}

\usepackage{lmodern} 
\usepackage[scale=0.88]{tgheros} 

\usepackage{bm} 

\usepackage{bbold}

\usepackage{pgfplots}

\usepackage{amsmath,amsbsy,amsgen,amscd,amsthm,amsfonts,amssymb} 
\allowdisplaybreaks[4]

\usepackage[centering,top=0.9in,bottom=0.9in,left=0.9in,right=0.9in]{geometry}

\usepackage{titling}
\usepackage{musicography}
\graphicspath{ {./images/} }

\usepackage[sf,bf,compact]{titlesec}

\titleformat{\subsubsection}[runin]{\bfseries\sffamily}{\thesubsubsection}{1em}{}

\usepackage[font=small,margin=20pt,labelfont={sf,bf},labelsep={space}]{caption}

\definecolor{dark-gray}{gray}{0.3}
\definecolor{dkgray}{rgb}{.4,.4,.4}
\definecolor{dkblue}{rgb}{0,0,.5}
\definecolor{medblue}{rgb}{0,0,.75}
\definecolor{rust}{rgb}{0.5,0.1,0.1}

\usepackage{url}
\usepackage[colorlinks=true]{hyperref}
\hypersetup{linkcolor=dkblue}    
\hypersetup{citecolor=rust}      
\hypersetup{urlcolor=rust}     

\usepackage[final]{microtype} 

\newtheoremstyle{myThm} 
    {\topsep}                    
    {\topsep}                    
    {\itshape}                   
    {}                           
    {\sffamily\bfseries}                   
    {.}                          
    {.5em}                       
    {}  

\newtheoremstyle{myRem} 
    {\topsep}                    
    {\topsep}                    
    {}                   
    {}                           
    {\sffamily}                   
    {.}                          
    {.5em}                       
    {}  

\newtheoremstyle{myDef} 
    {\topsep}                    
    {\topsep}                    
    {}                   
    {}                           
    {\sffamily\bfseries}                   
    {.}                          
    {.5em}                       
    {}  

\theoremstyle{myThm}
\newtheorem{theorem}{Theorem}[section]
\newtheorem{lemma}[theorem]{Lemma}
\newtheorem{proposition}[theorem]{Proposition}
\newtheorem{corollary}[theorem]{Corollary}

\newtheorem{assumption}[theorem]{Assumption}

\theoremstyle{myRem}
\newtheorem{remarkx}[theorem]{Remark}
 \newenvironment{remark}
  {\pushQED{\qed}\remarkx}
  {\popQED\endremarkx}

\theoremstyle{myDef}
\newtheorem{definition}[theorem]{Definition}
\theoremstyle{plain}
 \newtheorem{example}[theorem]{Example}

\usepackage{fancyhdr}
\let\originalleft\left
\let\originalright\right
\renewcommand{\left}{\mathopen{}\mathclose\bgroup\originalleft}
\renewcommand{\right}{\aftergroup\egroup\originalright}

\usepackage{mathtools}
\mathtoolsset{centercolon}  

\numberwithin{equation}{section} 

\definecolor{mygreen}{rgb}{0.1,0.75,0.2}

\newcommand{\nc}{\normalcolor}

\newcommand{\Id}{{I}} 

\providecommand{\mathbbm}{\mathbb} 

\newcommand{\R}{\mathbbm{R}}

\newcommand{\N}{\mathbbm{N}}

\usepackage[font = small, margin=30pt]{caption}

\usepackage[]{algorithm}
\usepackage{algpseudocode}
\usepackage{enumerate}

\usepackage{authblk}
\usepackage[square,numbers]{natbib}
\usepackage{chngcntr}
\usepackage{mathrsfs}
\counterwithin{table}{section}
\counterwithin{algorithm}{section}

\newcommand{\E}{\mathbb{E}}

\usepackage{multirow}
\usepackage{enumitem}

\newcommand{\Tr}{\operatorname{Tr}}

\usepackage[scr = esstix, cal = cm, frak=euler]{mathalfa}

\definecolor{mygreen}{rgb}{0.1,0.75,0.2}

\newcommand{\1}{\boldsymbol{1}}

\usepackage{cleveref}
\usepackage[scr = esstix, cal = cm, frak=euler]{mathalfa}

\newcommand{\cB}{\mathcal{B}}
\newcommand{\Lip}{\mathrm{Lip}}
\newcommand{\dtint}{\Delta t_{\mathrm{int}}}
\DeclareMathOperator{\Pdim}{Pdim}

\usepackage[normalem]{ulem}

\usepackage{authblk}
\title{\Huge 
Continuous Data Assimilation \\ with Learned Surrogate Dynamics} 
\author{\LARGE Wenwen Li}
\author{\LARGE Daniel Sanz-Alonso}
\affil{\Large  University of Chicago, USA}
\vspace{-4cm} 
\date{ }

\makeatletter\@addtoreset{section}{part}\makeatother%

\newcommand{\upperRomannumeral}[1]{\uppercase\expandafter{\romannumeral#1}}

\renewcommand{\hat}{\widehat}
\newcommand{\normH}[1]{\left\lVert #1\right\rVert_{H}}
\newcommand{\normV}[1]{\left\lVert #1\right\rVert_{V}}
\newcommand{\ip}[2]{\left\langle #1,#2\right\rangle}  

\providecommand{\keywords}[1]{\textbf{{Keywords:}} #1}

\begin{document}

\maketitle 
\vspace{-1.2cm}
\abstract{  
Continuous data assimilation seeks to estimate the state of a dynamical system from partial observations. In many applications, however, the state dynamics are unknown or prohibitively expensive to simulate at the required resolution, leading to model error. Motivated by this challenge and the increasing adoption of machine learning surrogates in data assimilation, this paper develops a unified finite-dimensional analysis of nudging algorithms that employ learned surrogate models of the dynamics. 
We first establish general conditions on the dynamics and observations that guarantee accurate tracking for nudging with the true dynamics model, both in the noise-free and noisy settings.
We then show that 
nudging algorithms that employ surrogate models retain exponential convergence up to an explicit error floor that quantifies the effects of surrogate approximation error and observation noise. 
Finally, we analyze surrogate models obtained by learning either the vector field or the short-time solution map of the system, and quantify the amount of training data needed to ensure accurate nudging in the noise-free setting. 
Numerical experiments support the theory. 
}

\section{Introduction}
\label{sec:intro}

Consider a dynamical system \begin{equation}\label{eq:truthintro}
\dot u = F(u),
\qquad
u(0)=u_0,
\end{equation}
with unknown initial condition $u_0.$ \emph{Continuous data assimilation} (CDA) is concerned with estimating, or \emph{tracking}, the state $u(t)$ from observations \(\{I_hu(s)\}_{0 \le s \le t}\) where $I_h$ is a given observation map. 

Nudging algorithms play a central role in CDA due to their simplicity, wide applicability, and strong theoretical guarantees. The idea is to initialize an assimilating trajectory at an arbitrary state
\(v_0\) and define the nudged system
\begin{equation}\label{eq:aotintro}
\dot v
=
F(v)-\mu\bigl(I_hv-I_hu\bigr),
\qquad
v(0)=v_0,
\end{equation}
where $\mu>0$ is a nudging parameter. The particular
feedback-based nudged dynamics in \eqref{eq:aotintro} belongs to the class of Azouani--Olson--Titi (AOT)  algorithms, in which the feedback is imposed through general observables \cite{AOT2013}.
Under suitable assumptions on the dynamics and observations, the systems \eqref{eq:truthintro} and \eqref{eq:aotintro} synchronize: the nudged trajectory $\{v(t)\}_{t \ge 0}$ converges to the true trajectory $\{u(t)\}_{t \ge 0}$ exponentially fast.

Motivated by applications in which the dynamics are unknown or expensive to simulate, this paper studies nudging algorithms of the form
\begin{equation}\label{eq:sur-aotintro}
\dot v
=
F_M(v)-\mu\bigl(I_hv-I_hu\bigr),
\qquad 
v(0)=v_0,
\end{equation}
where $F_M$ is a surrogate model of the true dynamics given by $F.$ In particular, we consider surrogate models learned offline from different forms of training data and quantify the amount of data required to guarantee accurate tracking of the true trajectory \eqref{eq:truthintro} using the nudged system \eqref{eq:sur-aotintro}.


\subsection{Main contributions and outline}
This paper develops a unified theory for nudging with learned surrogate models. Our analysis provides new insights into nudging accuracy under model error and helps establish a rigorous foundation for the use of machine learning surrogate models in CDA. The main contributions and organization of the paper are as follows:
\label{subsec:contrib}
\begin{itemize}
    \item Section \ref{sec:baseline} establishes exponential convergence of the nudged system \eqref{eq:aotintro} to the true system \eqref{eq:truthintro} under general conditions on the dynamics $F$ and the observation map $I_h.$ 
    We show that these conditions---used as standing assumptions throughout the paper---hold for several finite-dimensional dissipative systems commonly used as testbeds for data assimilation algorithms, as well as for many natural observation models. The main results are Theorem \ref{thm:conv} for noise-free observations and Theorem \ref{thm:stoch-tracking} for noisy observations. 
    \item Section \ref{sec:surrogate} studies the convergence of the nudged system with surrogate dynamics \eqref{eq:sur-aotintro} to the true system \eqref{eq:truthintro}. We work under our standing assumptions on the dynamics model $F$ and observation map $I_h,$ and show exponential convergence up to an error level determined by the accuracy of the surrogate model $F_M$. The main results are Theorem \ref{thm:sur-track} in the noise-free setting and Theorem \ref{thm:sur-track-noisy} under noisy observations. 
    These results show that nudging can remain accurate even with model error.
    \item Section \ref{sec:bridge} introduces two approaches for learning surrogate models from different types of training data.  The first approach directly learns the vector field $F,$ whereas the second learns the solution map for a short time step. We instantiate these two approaches using dictionary learning and neural networks, respectively. 
    For each approach, we show how surrogate learning errors affect tracking accuracy.
    \item Section \ref{sec:complexity} establishes sample complexity bounds for learning the vector-field and solution-map surrogates introduced in Section \ref{sec:bridge}. Thus, in this section we quantify the amount of training data required for nudging algorithms  based on learned surrogate models to provide accurate tracking, focusing on noise-free observations. The main results are Corollary \ref{cor:dict-aot-valid} for nudging based on vector-field surrogates obtained via dictionary learning and Corollary \ref{cor:r2-aot-valid} for nudging based on solution-map surrogates learned via deep super ReLU networks. 
    \item Section \ref{sec:numerics} contains numerical experiments on the Lorenz-96 system that demonstrate accurate tracking with surrogate models based on dictionary learning and neural networks under a variety of noiseless and noisy observation models.
\end{itemize}

\subsection{Related work}
\label{subsec:related}

We review the literature most closely related to our work along several directions. We begin with the broader data assimilation literature and the more specific CDA framework, with emphasis on nudging accuracy and tracking guarantees. We then discuss extensions of CDA to different systems, observation models, and discretizations, followed by work on data assimilation with model error and learned surrogate dynamics. Finally, we recall the learning-theoretic tools used to control surrogate approximation errors.

\paragraph{Data assimilation and tracking accuracy guarantees}
Data assimilation provides a general framework for combining dynamical models with partial and noisy observations in order to estimate the evolving state of a system; see, for example, \cite{LawStuartZygalakis2015,ReichCotter2015,Evensen2009,Kalnay2003,sanz2023inverse}. Within this broad area, continuous data assimilation (CDA) refers to methods that assimilate observations continuously in time. A central class of CDA methods is nudging, which can be viewed as an online filtering approach: the estimate \(v(t)\) at time \(t\) depends only on observations collected up to time \(t\). A particularly influential rigorous framework is the Azouani--Olson--Titi (AOT) algorithm \cite{AOT2013}, where feedback through general interpolant observables is used to synchronize an assimilating trajectory with the true state under suitable resolution and nudging conditions.
Since the seminal work \cite{AOT2013}, the analysis of AOT algorithms has been extended to a variety of dissipative models arising in fluid dynamics and geophysical flows, including the three-dimensional Navier--Stokes equations \cite{biswas2021cda3dns}, the three-dimensional Boussinesq system \cite{balakrishna2022determining}, the simplified Bardina model \cite{AlbanezBenvenutti}, the three-dimensional Navier--Stokes-\(\alpha\) model \cite{AlbanezLopesTitiNSalpha}, the three-dimensional Ladyzhenskaya model \cite{CaoGiorginiJollyPakzad}, the two-dimensional B\'enard convection problem with velocity measurements alone \cite{FarhatJollyTitiBenard}, and the subcritical surface quasi-geostrophic equation \cite{JollyMartinezTiti2017}.

Beyond these model-specific developments, AOT algorithms have also been studied in more general observation and discretization settings. Examples include blurred-in-time measurements \cite{JollyMartinezOlsonTiti2019}, higher-order finite element interpolants \cite{JollyPakzadFEM}, and fully discrete schemes with uniform-in-time error estimates \cite{IbdahMondainiTiti2020}. 
More abstract CDA frameworks for semilinear parabolic equations have also been developed; see, for example, \cite{delsarto2026cda}, as well as a recent stochastic extension that treats multiplicative observation noise and establishes mean-square and almost-sure synchronization \cite{BrockerDelSartoHieberPalmaZochling2026}.
Much of this literature is formulated for infinite-dimensional dissipative PDEs or their numerical approximations. In this paper, we do not aim to extend the continuous-space PDE theory. Rather, we use this literature as the conceptual and analytical background for a unified finite-dimensional framework that yields nudging-accuracy guarantees for both noise-free and noisy observations. Our unified framework relaxes the requirements on both the dynamics and observations required for accurate discrete-time data assimilation in \cite{law2016filter,sanz2015long}.
This analysis then serves as the basis for the learned-surrogate tracking results we develop. 

\paragraph{Model error and parameter estimation in data assimilation}
The exact-model theory overviewed above assumes that the nudged system leverages the dynamics model describing the true state evolution. In practice, however, the forecast model may be misspecified, simplified, or affected by unresolved dynamics. Data assimilation with model error has been investigated for example in \cite{ChenFarhatLunasin,cibik2025modelerror,moodey2013nonlinear,sanz2025long}. Another related direction uses CDA not only to recover the state but also to infer unknown parameters or identify missing model components online; see \cite{CarlsonHudsonLarios2020,NeweyWhiteheadCarlson2025}. These works provide important insight into robustness and model mismatch, but they typically focus on particular systems or on structured forms of model error, such as parametric mismatch or missing terms, and often rely on prior knowledge of the underlying model form. In contrast, the present work develops a unified CDA framework in which the imperfect model is represented by learned surrogate dynamics.

\paragraph{Data assimilation with machine-learned surrogate models}
Recent work has increasingly incorporated machine learning into data assimilation, either by replacing the forecast model with a learned surrogate or by using data assimilation to train, correct, or stabilize learned dynamics. Broader perspectives on the interaction between machine learning, data assimilation, inverse problems, and uncertainty quantification can be found in \cite{ChengEtAl2023,BachEtAl2024,chen2023stochastic}. Specific examples include data assimilation with FourCastNet-based surrogate models \cite{adrian2025data}, ensemble Kalman filtering with learned or climate-model surrogates \cite{takeshima2025bridging,sanz2025long}, the combination of data assimilation and machine learning to emulate dynamics from sparse and noisy observations \cite{BrajardCarrassiBocquetBertino2020}, and deep-learning surrogates for data assimilation in dynamic subsurface flow problems \cite{TangLiuDurlofsky2020}. These studies demonstrate the practical promise of learned surrogate models in data assimilation. The present work is complementary: rather than focusing on a specific learned forecast architecture or a specific filtering method, we analyze continuous-time nudging with learned surrogate dynamics and quantify how surrogate approximation errors propagate into long-time tracking accuracy. In particular, our results provide verifiable criteria under which a learned surrogate can be used in the nudging algorithm without compromising synchronization.

\paragraph{Learning theory and derivative-informed training}
Our analysis also connects CDA accuracy with learning-theoretic guarantees for the surrogate model, showing that AOT-valid surrogates can be learned under explicit sample-complexity conditions. Motivated by the need to control both the value and derivative errors of the learned surrogate, we use derivative-informed training objectives, which are well established in machine learning and scientific computing. Examples include Sobolev training \cite{Czarnecki2017}, Sobolev-type physics-informed neural network (PINN) losses \cite{SobolevPINN2021}, gradient-enhanced PINNs \cite{Yu2022gPINN}, and first-order formulations of PINNs \cite{Gladstone2022FOPINNs}. For approximation and sample-complexity tools, we mainly build on Sobolev approximation results for deep super ReLU networks \cite{yang2023dsrn} and classical learning-theoretic foundations for neural networks \cite{AnthonyBartlett1999}. In this paper, these tools are used to connect surrogate learning errors to the residual quantities that determine the tracking error of the nudged dynamics, thereby ensuring that the learned model error remains sufficiently small to preserve synchronization.

\section{Nudging accuracy with true dynamics}
\label{sec:baseline}
In this section, we establish exponential convergence of the nudged system \eqref{eq:aotintro} to the true system \eqref{eq:truthintro} under general conditions on the dynamics $F$ and the observation map $I_h.$ We introduce our standing assumptions on the true dynamics and on the observation model in Subsections \ref{subsec:baseline-setting} and \ref{subsec:obsmodel}, respectively. We then establish the well-posedness and dissipativity properties of the true and nudged systems in Subsection \ref{subsec:baseline-wp}. The main result of this section is Theorem \ref{thm:conv} in Subsection \ref{subsec:baseline-conv}, which establishes the desired exponential convergence. We generalize the theory to noisy observations in Subsection \ref{subsec:baseline-noisy}.


\subsection{Dynamics model: setting and examples}
\label{subsec:baseline-setting}
Throughout this paper, we consider dynamical systems on $\R^d.$ To emphasize the connection with CDA theory for infinite-dimensional dynamical systems, we introduce the Hilbert space structure
\[
H :=\R^d,
\qquad
\ip{z}{z'}_H := z^\top z',
\qquad
\normH{z}:=\sqrt{\ip{z}{z}_H}.
\]
We consider a reference trajectory \(u\), which plays the role of the truth,
and a nudged trajectory \(v\), driven by coarse observations of \(u\), whose
purpose is to track the reference trajectory.

\paragraph{True system}
We assume that the vector field $F:\R^d \to \R^d$ governing the state evolution in \eqref{eq:truthintro} admits the form $F(z) = f - \nu Az - N(z),$ so that
\begin{equation}\label{eq:truth}
\dot u(t)+\nu A u(t)+N(u(t))=f,
\qquad
u(0)=u_0.
\end{equation}
Here $u_0 \in \R^d$ is an unknown initial condition, \(\nu>0\) is a fixed dissipation parameter, \(f\in\R^d\) is a constant
forcing term, \(A\in\R^{d\times d}\) is a linear map, and
\(N:\R^d\to\R^d\) is a nonlinear map. The aim of this subsection is to introduce our standing assumptions on this dynamics model, and to verify these assumptions in several examples. We refer to \eqref{eq:truth} as the true system and to the trajectory $\{u(t)\}_{t \ge 0}$ as the truth.

\paragraph{Exact-model nudged system}
Substituting the vector field \(F(z)=f-\nu Az-N(z)\) into \eqref{eq:aotintro}, we obtain the nudged dynamics
\begin{equation}\label{eq:aot}
\dot v(t)+\nu A v(t)+N(v(t))
=
f-\mu\bigl(I_h v(t)-I_h u(t)\bigr),
\qquad
v(0)=v_0.
\end{equation}
Here \(v_0\in\R^d\) is an arbitrary initialization, \(\mu>0\) is the nudging parameter, and \(I_h:\R^d\to\R^d\) is a given linear feedback/interpolant operator depending on a resolution parameter \(h>0\). Our assumption on \(I_h\) will be introduced in Subsection~\ref{subsec:obsmodel}. We refer to \eqref{eq:aot} as the exact-model nudged system, since the nudging equation uses the exact vector field \(F\). The corresponding AOT-type data assimilation procedure is referred to as exact-model AOT.

\bigskip

We make the following assumptions on the matrix $A \in \R^{d \times d}$ and nonlinearity $N: \R^d \to \R^d$ defining the dynamics model \eqref{eq:truth}.

\begin{assumption}[Linear dissipation]\label{ass:H1}
The matrix \(A\in\R^{d\times d}\) is symmetric positive definite.
\end{assumption}

\begin{assumption}[\(H\)-dissipativity]
\label{ass:H4}
There exist constants \(\alpha>0\) and \(\beta\ge0\) such that, for every
\(z\in\R^d\),
\begin{equation}\label{eq:alpha-beta}
\nu\,z^\top A z+\ip{N(z)}{z}_H
\ge
\alpha\,\normH{z}^2-\beta.
\end{equation}
\end{assumption}

\begin{assumption}[Nonlinearity]\label{ass:H2}
The mapping \(N:\R^d\to\R^d\) is locally Lipschitz and satisfies:
\begin{enumerate}
\item \textbf{One-sided energy bound.}
There exists a constant \(C_{\mathrm E}\ge 0\) such that, for every
\(z\in\R^d\),
\begin{equation}\label{eq:onesided}
\ip{N(z)}{z}_H \ge -C_{\mathrm E}\,\normH{z}^2.
\end{equation}

\item \textbf{Local monotonicity on bounded sets.}
For every \(R>0\), there exists a constant \(C_R\ge 0\) such that, whenever
\(z,z'\in\R^d\) satisfy \(\normH{z}\le R\) and \(\normH{z'}\le R\), one has
\begin{equation}\label{eq:local-mono}
-\ip{N(z')-N(z)}{z'-z}_H
\le
C_R\,\normH{z'-z}^2.
\end{equation}
\end{enumerate}
\end{assumption}

In the following remarks, we discuss these assumptions in turn. 

\begin{remark}[Dissipation and energy norm]\label{rem:energynorm}
Under Assumption \ref{ass:H1}, we denote by
\[
0<\lambda_1:=\lambda_{\min}(A)\le \lambda_{\max}(A)=:\lambda_{\max}<\infty
\]
the smallest and largest eigenvalues of \(A\). We can equip \(\R^d\) with
the associated energy norm
\[
\normV{z}:=\bigl\|A^{1/2}z \bigr\|_H=\sqrt{z^\top A z},
\qquad z\in\R^d.
\]
Since \(A\) is symmetric positive definite, the norms \(\normH{\cdot}\) and
\(\normV{\cdot}\) are equivalent. More precisely,
\begin{equation}\label{eq:norm-equivalence}
\lambda_1\,\normH{z}^2
\le
\normV{z}^2
\le
\lambda_{\max}\,\normH{z}^2,
\qquad
\forall z\in\R^d.
\end{equation}
Despite the $V$ and $H$ norms being equivalent in our finite-dimensional setting, we will use both norms to highlight connections with existing theory developed in functional settings. 
\end{remark}

\begin{remark}[H-dissipativity]
    Assumption~\ref{ass:H4} ensures dissipativity at the
\(H\)-level. In particular, it provides the quantitative coercivity needed to
derive absorbing-ball estimates for the dynamics. It also includes, as a
special case, the standard regime in which the linear dissipation dominates the
nonlinear contribution. Indeed, if
\[
\ip{N(z)}{z}_H\ge -C_{\mathrm E}\normH{z}^2
\qquad\forall z\in\R^d,
\]
and \(\nu\lambda_1>C_{\mathrm E}\), then Assumption~\ref{ass:H4} holds with
\[
\alpha=\nu\lambda_1-C_{\mathrm E},
\qquad
\beta=0.
\]
Thus, the classical situation in which the linear part controls the possible
energy production of the nonlinearity is recovered within this framework. We formulate Assumption~\ref{ass:H4} at the level of the full drift
\(-\nu Az-N(z)\), rather than solely through a lower bound on the nonlinear
term, because this form is more flexible and is better suited to situations in
which dissipativity results from the combined effect of the linear and
nonlinear parts.
\end{remark}

\begin{remark}[Local monotonicity]
\label{rem:examples-local-mono}
In this remark we obtain a sufficient, easy-to-verify condition to ensure local monotonicity. 
Let \(N\in C^1(\R^d;\R^d)\). Then for every \(R>0\) there exists
\[
L_R:=\sup_{\|z\|_2\le R}\|DN(z)\|_{\mathrm{op}}<\infty.
\]
By the mean value theorem, for all \(z,z'\in\R^d\) with
\(\|z\|_2,\|z'\|_2\le R\),
\[
\|N(z')-N(z)\|_2\le L_R\|z'-z\|_2,
\]
and hence
\[
-\ip{N(z')-N(z)}{z'-z}_H
\le
L_R\|z'-z\|_2^2.
\]
Thus Assumption~\ref{ass:H2}\emph{(ii)} holds on bounded sets, with
\(C_R=L_R\).
\end{remark}

We next present several illustrative examples of dynamical systems that satisfy our assumptions. 
\subsubsection{Stuart--Landau oscillator}
\label{subsec:stuart-landau}

The Stuart--Landau oscillator \cite{stuart1958} 
is a canonical normal form for smooth limit-cycle dynamics arising from oscillatory instability. In real variables it is given by
\begin{equation}\label{eq:stuart-landau}
\dot u
=
\lambda u+\omega J u-|u|^2u,
\qquad
u(t)\in\R^2,
\end{equation}
where
\[
J=
\begin{pmatrix}
0&-1\\
1&0
\end{pmatrix}.
\]

We can write \eqref{eq:stuart-landau} in the form \eqref{eq:truth} by setting
\[
H=V=\R^2,
\qquad
A=\Id,
\qquad
\nu=1,
\qquad
f=0,
\]
and
\[
N(u):=-(\lambda+1)u-\omega Ju+|u|^2u.
\]
Assumption~\ref{ass:H1} is immediate. The key point for
Assumption~\ref{ass:H4} is that the positive quartic contribution in
\(\ip{N(u)}{u}_H\) compensates for the potentially unfavorable quadratic part
and yields \(H\)-dissipativity of the full drift. The corresponding bounds and
an explicit admissible choice of \((\alpha,\beta)\) are summarized in
Appendix~\ref{app:examples}. The skew-symmetry of \(J\) and the
quartic damping term imply Assumption~\ref{ass:H2}\emph{(i)}, while
Assumption~\ref{ass:H2}\emph{(ii)} follows from the smoothness of \(N\) and
Remark~\ref{rem:examples-local-mono}.

\subsubsection{FitzHugh--Nagumo system}
\label{subsec:fhn}

The FitzHugh--Nagumo system is a classical model for excitable and oscillatory
dynamics \cite{fitzhugh1961,nagumo1962}. It is designed to capture threshold response,
recovery, and fast--slow effects typical of excitable systems. We consider the
form
\begin{equation}\label{eq:fhn}
\dot u_1 = u_1-\frac{u_1^3}{3}-u_2+I,
\qquad
\dot u_2 = \varepsilon(u_1+a-bu_2),
\end{equation}
with \(\varepsilon>0\) and \(b>0\). Let \(u=(u_1,u_2)^\top\). We fit \eqref{eq:fhn} into \eqref{eq:truth} by taking
\[
H=V=\R^2,
\qquad
A=
\begin{pmatrix}
1&0\\
0&\varepsilon b
\end{pmatrix},
\qquad
\nu=1,
\qquad
f=
\begin{pmatrix}
I\\
\varepsilon a
\end{pmatrix},
\]
and
\[
N(u)=
\begin{pmatrix}
-2u_1+\frac{u_1^3}{3}+u_2\\
-\varepsilon u_1
\end{pmatrix}.
\]
Then Assumption~\ref{ass:H1} holds. For Assumption~\ref{ass:H4}, the main
mechanism is that the cubic nonlinearity provides the dominant dissipative
effect, whereas the remaining lower-order terms can be absorbed into the
combined \((\alpha,\beta)\)-form. The detailed estimates can be found in
Appendix~\ref{app:examples}. The cubic term gives the required one-sided energy bound, while
Assumption~\ref{ass:H2}\emph{(ii)} again follows from
Remark~\ref{rem:examples-local-mono}.


\subsubsection{Lorenz--96}
\label{subsec:l96}

The Lorenz--96 system is a classical benchmark in predictability, model error,
and data assimilation \cite{lorenz1996}. It is given by 
\begin{equation}\label{eq:l96}
\dot u_i=(u_{i+1}-u_{i-2})u_{i-1}-u_i+\mathsf{f},
\qquad i=1,\dots,d,
\end{equation}
with periodic indexing.
We fit \eqref{eq:l96} into \eqref{eq:truth} by taking
\[
H=V=\R^d,
\qquad
A=\Id,
\qquad
\nu=1,
\qquad
f=\mathsf{f}\mathbf 1,
\]
and
\[
N(u)_i:=-(u_{i+1}-u_{i-2})u_{i-1}.
\]
A key property of this system is that the quadratic nonlinearity $N$ is energy conserving: it holds that $\langle N(z),z \rangle_H = 0$ for all $z \in \R^d.$
Assumption~\ref{ass:H1} is immediate. The fact that the nonlinearity is energy-conserving implies Assumption~\ref{ass:H4} with \(\alpha=1\) and \(\beta=0\) and Assumption~\ref{ass:H2}\emph{(i)} with
\(C_{\mathrm E}=0\). Finally, since \(N\) is smooth,
Assumption~\ref{ass:H2}\emph{(ii)} follows from
Remark~\ref{rem:examples-local-mono}. A full verification of these claims is
included in Appendix~\ref{app:examples}.

\subsubsection{Dissipative models with energy-conserving bilinear nonlinearity}
\label{rem:l63}
A similar verification applies to the Lorenz--63 system
\cite{lorenz1963}. In particular, after a standard reformulation, the
nonlinear term also satisfies conservation of energy, while the linear part provides dissipation, see e.g. \cite{sanz2015long}. More broadly, the same mechanism also applies to
finite-dimensional Galerkin-type models for viscous incompressible flows,
built on divergence-free reduced spaces, such as those arising from
Fourier--Galerkin truncations. In such cases, incompressibility can be retained at
the reduced level, and the
reduced convective term inherits the usual conservation of energy from
the continuous incompressible model. Consequently, the reduced dynamics take
the form of a dissipative linear part together with an energy-conserving
bilinear term, so that Assumptions~\ref{ass:H1}, \ref{ass:H4} and \ref{ass:H2}
 hold. Further details are provided in Appendix~\ref{app:examples}.

\subsection{Observation model: setting and examples}\label{subsec:obsmodel}
Recall the energy norm $\| \cdot \|_V$ introduced in Remark \ref{rem:energynorm}. We make the following assumption on the observation map $I_h$. 
\begin{assumption}[Feedback operator]\label{ass:H3}
For each resolution parameter \(h\in[0,1]\), the operator
\(I_h:\R^d\to\R^d\) is linear and satisfies
\begin{equation}\label{eq:interp}
\normH{z-I_h z}
\le
c_0\,h\,\normV{z},
\qquad \forall\, z\in\R^d,
\end{equation}
for some constant \(c_0>0\) independent of both \(z\) and \(h\).
\end{assumption}

\begin{remark}[\(H\)-boundedness]\label{rem:Ih-bounded}
Since the \(H\)- and \(V\)-norms are equivalent on \(\R^d\), \eqref{eq:interp}
implies that \(I_h\) is uniformly bounded on \(H\). More precisely, there
exists a constant \(C_I>0\), independent of \(h\in[0,1]\), such that
\begin{equation}\label{eq:stab}
\normH{I_h z}\le C_I\,\normH{z},
\qquad \forall\, z\in\R^d.
\end{equation}
\end{remark}

Feedback operators arising in applications can often be expressed in the form 
\[
I_h = R_h \circ O_h,
\]
where \(O_h: \R^d \to \R^k\) maps the state to the available observations and \(R_h:\R^k\to \R^d\)
reconstructs a state vector from those observations. In the classical
infinite-dimensional AOT literature, the operator \(I_h\) is typically realized
through nodal interpolation, local spatial averages, or spectral projection. In
the present finite-dimensional setting, the same idea is encoded at the matrix
level: \(I_h\) is simply a \(d \times d\) matrix obtained by composing a sensing
map with a reconstruction map.  Assumption~\ref{ass:H3} is a natural structural requirement on
the sensing and reconstruction mechanism: it requires \(I_h\) to be near the
identity in the \(H\)-norm. 
We now provide three paradigmatic examples of feedback operators satisfying Assumption \ref{ass:H3}.

\subsubsection{Linear sensing measurements.}

This setting arises when the state is represented by a finite-dimensional
vector of physically meaningful variables, while the observations are
indirect sensor outputs rather than direct access to the state coordinates.
Such a measurement structure is common in power-system monitoring, process
engineering, and networked control, where the recorded quantities are linear or
linearized combinations of the underlying state variables
\cite{abur2004power,simon2006}. A similar finite-dimensional observation
structure also appears when a spatially distributed system has already been
discretized in space, so that the state is a vector in \(\R^d\), but the
instrumentation provides only indirect linear observations of that discretized
state. In that setting, blurred or coarse-resolution sensing may be viewed as a
natural special case of the same measurement paradigm. In all of these
situations, \(I_h\) represents the resulting measurement-and-reconstruction map
from the available sensor outputs back to the state space.

\begin{example}
Suppose the observations are of the form
\[
y = Gu,
\qquad
G \in \R^{k \times d}.
\]
A natural reconstruction is given
by the Tikhonov-regularized inverse
\[
R_h = (G^\top G + hI)^{-1} G^\top,
\qquad
O_h = G,
\]
so that
\begin{equation}
\label{eq:redundant linear}
    I_h = R_h O_h = (G^\top G + hI)^{-1} G^\top G.
\end{equation}
If the sensing matrix \(G\) has full column rank, then \(I_h\) is globally
close to the identity, and the parameter \(h\) plays the role of an effective
resolution or regularization scale. In particular, the singular values of
\(G\) quantify how well the observation mechanism resolves the state space.
Indeed, the singular value decomposition yields
\[
\|I-I_h\|_{H\to H}\le \frac{h}{\sigma_{\min}(G)^2+h}=\mathcal{O}(h),
\]
so Assumption~\ref{ass:H3} is naturally satisfied.
An important special case is that of blurred or averaged observations,
where \(G \in \R^{k \times d}\) is a smoothing or averaging matrix. 
In that case, \(h\) may be
interpreted either as a regularization parameter or as the effective smoothing
scale of the observation mechanism.
\end{example}

\subsubsection{Band-limited spectral measurements.}

Here the state represents a sampled or discretized spatial field, while the measurements are frequency-domain observations that capture only its resolved large-scale content. Such observations arise in imaging and array-sensing modalities in which the device records partial spectral information. For example, in medical imaging modalities such as magnetic resonance imaging (MRI), the machine does not directly record the final image point by point; instead, it first collects transformed frequency-domain data from which the image is later reconstructed. Likewise, in radio interferometry the directly observed quantities are visibilities, namely sampled Fourier components of the sky brightness distribution \cite{hansen2014image,thompson2017interferometry}. In this situation, the observations are spectral coefficients of the state, whereas \(I_h\) represents the corresponding reconstruction operator that maps the measured band-limited spectral data back to the physical state space, with the stronger \(V\)-norm controlling the unresolved high-frequency tail.

\begin{example}
Let \(\mathscr F\) denote the discrete Fourier transform and \(\mathscr F^\ast\) its inverse. Given a cutoff level \(K\), let \(P_K\) be the projection onto the lowest \(K\) frequency modes. One may then define
\[
O_h = P_K \mathscr F,
\qquad
R_h = \mathscr F^\ast,
\qquad
I_h = \mathscr F^\ast P_K \mathscr F.
\]
This is the finite-dimensional analogue of the determining-modes viewpoint in the classical CDA literature. Here the parameter \(h\) may be interpreted as an inverse resolution scale, typically with \(h \asymp K^{-1}\). The operator \(I_h\) is the projection onto the resolved low-frequency subspace, and the unresolved high-frequency tail is controlled by the stronger \(V\)-norm in a way consistent with the approximation property in Assumption~\ref{ass:H3}.
\end{example}

\subsubsection{Dominant modal coefficient measurements.}

This setting arises when the observations used for
feedback consist only of a small number of dominant modal coefficients, obtained for instance via leading terms of a principal orthogonal decomposition (POD) or Karhunen--Lo\`eve (KL) expansion.
This type of
measurement structure is common in flow reconstruction from sparse sensors and in
reduced-order state estimation for fluid and thermal-fluid systems
\cite{willcox2006,loiseau2018,nair2020}, where full-state sensing,
storage, or simulation is expensive. Here, the observations
are the reduced modal coordinates, and \(I_h\) denotes the
associated lifting operator that reconstructs an approximate full state from
the measured dominant modes.

\begin{example}
Let \(U_k \in \R^{d \times k}\) have orthonormal columns given, for instance, by the leading
POD or KL modes extracted from training data. Define
\[
O_h = U_k^\top,
\qquad
R_h = U_k,
\qquad
I_h = U_k U_k^\top.
\]
Then \(O_h\) extracts the leading modal coefficients of the state, \(R_h\)
lifts them back to the original state space, and \(I_h\) is the orthogonal
projector onto the resulting \(k\)-dimensional reduced subspace. This
construction is natural when the informative part of the measurement or
feedback is well captured by a low-dimensional modal structure. In this case,
\(h\) may be interpreted as a rank-resolution parameter, and one chooses
\(k\) sufficiently large so that the truncated energy, or equivalently the
neglected covariance spectrum, falls below a threshold of order \(h\).
Assumption~\ref{ass:H3} is then natural in this setting, provided that
states bounded in \(V\)-norm are well approximated by their projections onto
the dominant POD/KL modes.
\end{example}

\subsection{Well-posedness and dissipativity analysis}
\label{subsec:baseline-wp}

We begin with the basic well-posedness statement. In finite dimensions, local
existence and uniqueness follow from the local Lipschitz continuity of the
vector fields. Global existence then follows from the a priori estimates
established below. The proofs are postponed to Appendix~\ref{app:baseline}.

\begin{proposition}[Global well-posedness of the true and nudged systems]
\label{prop:wp}
Let Assumptions~\ref{ass:H1},\ref{ass:H4}, \ref{ass:H2}, and
\ref{ass:H3} be in force. Then, for every initial
condition pair \(u_0,v_0\in\R^d\), the true system \eqref{eq:truth} and the
nudged system \eqref{eq:aot} admit unique global solutions
\[
u,v\in C^1([0,\infty);\R^d).
\]
\end{proposition}

We next establish dissipativity at the \(H\)-level. The next result gives separate \(H\)-bounds for the true trajectory and for the exact-model nudged trajectory. The truth estimate follows from the \(H\)-dissipativity of the underlying dynamics, while the nudged estimate additionally uses the  approximation property of the observation map together with suitable compatibility conditions among the nudging strength, the feedback resolution, and the dissipation.

\begin{proposition}[Dissipativity and absorbing balls in \(H\)]
\label{prop:dissipativity}
Let
Assumptions~\ref{ass:H1},\ref{ass:H4}, \ref{ass:H2}, and \ref{ass:H3}
be in force. Let \(u\) and \(v\) denote the corresponding global solutions
of \eqref{eq:truth} and \eqref{eq:aot}, respectively, associated with an
arbitrary initial condition pair \((u_0,v_0)\in\R^d\times\R^d\).
Then the following hold.

\begin{enumerate}
\item[(i)] The true solution satisfies
\begin{equation}\label{eq:H-abs-bound}
\|u(t)\|_H^2
\le
e^{-\alpha t}\|u_0\|_H^2
+
\frac{2\beta+\|f\|_H^2/\alpha}{\alpha}
\bigl(1-e^{-\alpha t}\bigr),
\qquad t\ge0.
\end{equation}
In particular,
\[
\limsup_{t\to\infty}\|u(t)\|_H^2
\le
\frac{2\beta+\|f\|_H^2/\alpha}{\alpha}.
\]

\item[(ii)] Suppose, in addition, that
\begin{equation}\label{eq:h-small-H}
\mu c_0^2 h^2<\nu .
\end{equation}
Set
\begin{equation}\label{eq:delta-def}
\delta
:=
\frac{\nu\lambda_1}{2}
+\frac{\mu}{2}
-C_{\mathrm E}.
\end{equation}
If \(\delta>0\), then there exist constants \(T_H\ge0\) and \(R_H>0\),
depending only on the system parameters, the forcing, and the initial data,
such that the nudged solution satisfies
\begin{equation}\label{eq:nudged-H-bound}
\sup_{t\ge T_H}\|v(t)\|_H\le R_H .
\end{equation}
\end{enumerate}
\end{proposition}

The previous proposition yields eventual boundedness in the \(H\)-norm for both the truth and the exact-model nudged trajectory. Since the \(H\)- and \(V\)-norms are equivalent in our finite-dimensional setting, this immediately gives a common post-absorption region that is bounded in both norms. We record this consequence separately, since it will be used repeatedly in the sequel.

\begin{corollary}[Common post-absorption region in \(H\) and \(V\)]
\label{cor:common-absorbing}
In the setting of Proposition~\ref{prop:dissipativity}, suppose in addition that the parameter conditions \eqref{eq:h-small-H} and \(\delta>0\), with
\(\delta\) defined in \eqref{eq:delta-def}, hold. Then, for every
initial condition pair \(u_0,v_0\in\R^d\), there exist constants
\(T_\ast\ge0\) and \(R_\ast>0\) such that the corresponding solutions \(u\)
and \(v\) satisfy
\begin{equation}\label{eq:common-absorbing-HV}
\sup_{t\ge T_\ast}\|u(t)\|_H\le R_\ast,
\qquad
\sup_{t\ge T_\ast}\|v(t)\|_H\le R_\ast,
\end{equation}
and
\begin{equation}\label{eq:common-absorbing-V}
\sup_{t\ge T_\ast}\|u(t)\|_V\le R_\ast,
\qquad
\sup_{t\ge T_\ast}\|v(t)\|_V\le R_\ast.
\end{equation}
\end{corollary}

\subsection{Exact-model convergence}
\label{subsec:baseline-conv}

We now turn to the synchronization property of the exact-model nudged dynamics. After both trajectories have entered a common absorbing region, the local monotonicity assumption on the nonlinearity holds with a uniform constant. Together with the coercive contribution of the nudging term, this implies exponential convergence in the \(H\)-norm; see Theorem~\ref{thm:conv}. A detailed proof is provided in Appendix~\ref{app:baseline}.

\begin{theorem}[Exponential convergence in \(H\)]
\label{thm:conv}
Let
Assumptions~\ref{ass:H1},\ref{ass:H4}, \ref{ass:H2}, and \ref{ass:H3} be in force. Suppose further that
\[
\mu c_0^2 h^2<\nu,
\qquad
\delta>0,
\]
where \(\delta\) is defined in \eqref{eq:delta-def}.
Let \((u,v)\) be the corresponding global solutions of
\eqref{eq:truth} and \eqref{eq:aot} associated with an arbitrary initial pair \((u_0,v_0)\in\R^d\times\R^d\). Let \(T_\ast\) and \(R_\ast\) be as in Corollary~\ref{cor:common-absorbing}, and define
\(C_\ast:=C_{R_\ast},\) where \(C_{R_\ast}\) denotes the local monotonicity constant from Assumption~\ref{ass:H2}\emph{(ii)} on the \(H\)-ball of radius \(R_\ast\).
If \(\mu>2C_\ast\), then, for every \(t\ge T_\ast\),
\begin{equation}\label{eq:exp-conv}
\|v(t)-u(t)\|_H^2
\le
\exp\bigl(-(\mu-2C_\ast)(t-T_\ast)\bigr)
\|v(T_\ast)-u(T_\ast)\|_H^2.
\end{equation}
\end{theorem}
In particular, Theorem \ref{thm:conv} shows that the nudged trajectory \(v\) synchronizes exponentially fast to
the truth \(u\) in the \(H\)-norm after the entrance time \(T_\ast\).

\subsection{Noisy observations and stochastic nudging}
\label{subsec:baseline-noisy}

Recall from Subsection~\ref{subsec:obsmodel} that the feedback operator is
written in the form
\(I_h = R_h \circ O_h,\)
where \(O_h:\R^d\to\R^k\) denotes the observation map and
\(R_h:\R^k\to\R^d\) denotes the reconstruction operator.
So far, we have worked in the idealized setting in which the exact observations \(I_hu(t)=R_h(O_hu(t))\) are available to the nudged
dynamics. In practice, however, the exact observations \(O_hu(t)\) are
typically not directly accessible. Rather, the available data are contaminated
by random measurement errors, and the feedback must therefore be
constructed from noisy observations. In this subsection, we introduce a
stochastic extension of the exact-model nudged dynamics that incorporates such
observation errors.

In the absence of
observation errors, the exact-model nudged dynamics is given by \eqref{eq:aot},
\[
\dot v(t)+\nu A v(t)+N(v(t))
=
f-\mu\bigl(I_hv(t)-I_hu(t)\bigr),
\]
where the nudging term is driven by the exact coarse observable \(I_hu(t)\).

We now replace the exact observations by noisy ones. More precisely, we
introduce the noisy observation process
\begin{equation}\label{eq:noisy-obs-level}
\widetilde O_h(u(t)) = O_hu(t) + \mathcal E(t),
\end{equation}
where \(\mathcal E(t)\in\R^k\) denotes the observation error.

Following \cite{Bessaih_2015}, we assume that the components of the
measurement error are of  Gaussian type. More precisely, let
\[
(\Omega,\mathcal F,(\mathcal F_t)_{t\ge0},\mathbb P)
\]
be a filtered probability space supporting a standard \(k\)-dimensional
Brownian motion
\[
W_t=(b_1(t),\dots,b_k(t))^\top,
\]
and let \(\Sigma\in\R^{k\times k}\) be a deterministic covariance factor.
We represent the observation error in differential form by
\begin{equation}\label{eq:noisy-obs-error}
\mathcal E(t)\,dt = \Sigma\,dW_t.
\end{equation}
Thus the observation-noise increments are centered Gaussian with covariance
\[
\mathbb E \left[\Sigma\bigl(W_t-W_s\bigr)\right]=0,
\qquad
\mathbb E \left[
\Sigma\bigl(W_t-W_s\bigr)\bigl(\Sigma(W_t-W_s)\bigr)^\top
\right]
=
(t-s)\Sigma\Sigma^\top,
\]
for all \(0\le s\le t\). More generally, for every \(p\ge1\), there exists a
constant \(C_p>0\) such that
\[
\mathbb E \bigl\|\Sigma\bigl(W_t-W_s\bigr)\bigr\|_2^{2p}
\le
C_p \|\Sigma\|_{\mathrm{op}}^{2p} |t-s|^p.
\]
In the isotropic case \(\Sigma=\sigma I_k\), the components of the
observation error are independent and identically distributed Gaussian noises
with variance parameter \(\sigma^2\).

Applying the reconstruction operator \(R_h\) to the noisy observations
\eqref{eq:noisy-obs-level}, and using the linearity of \(R_h\), we obtain the
noisy reconstructed observable
\begin{equation}\label{eq:noisy-reconstructed-observable}
\widetilde I_hu(t)\,dt
:=
R_h\bigl(\widetilde O_h(u(t))\bigr)\,dt
=
I_hu(t)\,dt + \Gamma_h\,dW_t,
\qquad
\Gamma_h:=R_h\Sigma\in\R^{d\times k}.
\end{equation}
Thus the observation error in the measurement space induces, after
reconstruction, an additive random perturbation in the feedback channel. Since
the state space is finite-dimensional, one automatically has
\[
\Tr(\Gamma_h\Gamma_h^\top)=\|\Gamma_h\|_{\mathrm F}^2<\infty.
\]

Accordingly, when the inaccessible exact feedback \(I_hu(t)\) in
\eqref{eq:aot} is replaced by the noisy reconstructed observable
\(\widetilde I_hu(t)\), the feedback law is driven by the available quantity
\(\widetilde I_hu(t)\). Thus the noisy nudged dynamics may first be written in
the form
\begin{equation}\label{eq:aot-noisy-pre}
dv(t)+\nu A v(t)\,dt+N(v(t))\,dt
=
f\,dt-\mu\bigl(I_hv(t)-\widetilde I_hu(t)\bigr)\,dt,
\qquad
v(0)=v_0\in\R^d.
\end{equation}
Using \eqref{eq:noisy-reconstructed-observable}, we may equivalently rewrite
\eqref{eq:aot-noisy-pre} as
\begin{equation}\label{eq:aot-noisy}
dv(t)+\nu A v(t)\,dt+N(v(t))\,dt
=
f\,dt-\mu\bigl(I_hv(t)-I_hu(t)\bigr)\,dt
+\mu \Gamma_h\,dW_t,
\qquad
v(0)=v_0\in\R^d.
\end{equation}
In other words, the measurement error enters the nudging term through the
reconstruction map and appears in the nudged dynamics as an additive
stochastic forcing.

We first establish the basic well-posedness and dissipativity properties of
\eqref{eq:aot-noisy}. As in the deterministic setting, the parameter condition
\(\mu c_0^2 h^2<\nu\) ensures that the coercive effect of the linear part is
not destroyed by the feedback term. The following result is the stochastic
counterpart of Proposition~\ref{prop:dissipativity}: although one can no longer expect a deterministic absorbing ball, one still obtains a uniform \(H\)-moment bound.
\begin{proposition}[Global strong well-posedness and \(H\)-moment boundedness]
\label{prop:stoch-wp-diss}
Suppose that the assumptions of Proposition~\ref{prop:dissipativity} are in force,  and that the parameter condition \eqref{eq:h-small-H} holds together with \(\delta>0\), where \(\delta\) is defined in \eqref{eq:delta-def}.
Let \(u\) be
the global solution of \eqref{eq:truth}, and let
\(\Gamma_h\in\R^{d\times k}\) be defined by
\eqref{eq:noisy-reconstructed-observable}. Then, for every
\(v_0\in\R^d\), the following statements hold.

\begin{enumerate}
\item[\emph{(i)}] The stochastic nudged system \eqref{eq:aot-noisy} admits a
unique global adapted strong solution with continuous paths,
\[
v\in C([0,\infty);\R^d)
\qquad\text{a.s.}
\]

\item[\emph{(ii)}] For every \(t\ge0\),
\begin{equation}\label{eq:stoch-moment-bound}
\mathbb E\normH{v(t)}^2
\le
e^{-\delta t}\normH{v_0}^2
+
\int_0^t e^{-\delta(t-s)}
\left(
\frac{2}{\delta}\normH{f}^2
+
\frac{2\mu^2 C_I^2}{\delta}\normH{u(s)}^2
+
\mu^2 \Tr(\Gamma_h\Gamma_h^\top)
\right)\,ds,
\end{equation}
where \(C_I\) is the \(H\)-operator bound from
Remark~\ref{rem:Ih-bounded}, and \(\delta>0\) is the constant defined in
\eqref{eq:delta-def}.

\item[\emph{(iii)}] Let
\(R_u^2:=\frac{2\alpha\beta+\normH{f}^2}{\alpha^2},\)
where \(\alpha\) and \(\beta\) are the constants in
Assumption~\ref{ass:H4}. Then the deterministic dissipativity estimate gives
\[
\limsup_{t\to\infty}\normH{u(t)}^2 \le R_u^2,
\]
and consequently
\begin{equation}\label{eq:stoch-absorbing}
\limsup_{t\to\infty}\mathbb E\normH{v(t)}^2
\le
\frac{2}{\delta^2}\normH{f}^2
+
\frac{2\mu^2 C_I^2}{\delta^2}R_u^2
+
\frac{\mu^2}{\delta}\Tr(\Gamma_h\Gamma_h^\top).
\end{equation}
In particular, there exist constants \(T_H^{\rm sto}\ge0\) and
\(R_H^{\rm sto}>0\), depending only on the system parameters, the forcing, the
noise level, and the initial data, such that
\[
\sup_{t\ge T_H^{\rm sto}}\mathbb E\normH{v(t)}^2
\le (R_H^{\rm sto})^2 .
\]
\end{enumerate}
\end{proposition}

The previous proposition gives the stochastic analogue of the dissipative \(H\)-bound from the deterministic analysis. However, unlike
Corollary~\ref{cor:common-absorbing}, uniform second-moment bounds do not, in general, provide a fixed deterministic post-absorption region on which the local monotonicity constant in Assumption~\ref{ass:H2}\emph{(ii)} can be
frozen. For this reason, the tracking theorem below is stated under the following global two-point condition.

\begin{assumption}[Global one-sided Lipschitz condition]
\label{ass:H2-global}
There exists a constant \(C_{\rm gl}\ge0\) such that, for every
\(z,z'\in\R^d\),
\begin{equation}\label{eq:global-mono}
-\ip{N(z')-N(z)}{z'-z}_H
\le
C_{\rm gl}\normH{z'-z}^2.
\end{equation}
\end{assumption}

Under this additional assumption, one obtains a global mean-square tracking estimate. This is the stochastic counterpart of Theorem~\ref{thm:conv}: the nudged trajectory converges exponentially fast to the truth up to an explicit noise-dependent floor. The proofs of Proposition~\ref{prop:stoch-wp-diss} and Theorem~\ref{thm:stoch-tracking} are deferred to Appendix~\ref{app:baseline-noisy}.

\begin{theorem}[Mean-square tracking under noisy observations]
\label{thm:stoch-tracking}
In the setting of Proposition~\ref{prop:stoch-wp-diss}, suppose in addition that Assumption~\ref{ass:H2-global} holds and that
\begin{equation}\label{eq:stoch-track-smallness}
\mu>2C_{\rm gl}.
\end{equation}
Let \(u\) be the global solution of \eqref{eq:truth}, and let \(v\) be the unique global strong solution of \eqref{eq:aot-noisy}, associated with an arbitrary initial condition pair \((u_0,v_0)\in\R^d\times\R^d\). Then, for
every \(t\ge0\),
\begin{equation}
\mathbb E \normH{v(t)-u(t)}^2
\le
e^{-(\mu-2C_{\rm gl})t}\normH{v_0-u_0}^2
+
\frac{\mu^2\,\Tr(\Gamma_h\Gamma_h^\top)}
{\mu-2C_{\rm gl}}
\Bigl(1-e^{-(\mu-2C_{\rm gl})t}\Bigr).
\label{eq:stoch-tracking}
\end{equation}
Consequently,
\begin{equation}\label{eq:stoch-floor}
\limsup_{t\to\infty}\mathbb E \normH{v(t)-u(t)}^2
\le
\frac{\mu^2\,\Tr(\Gamma_h\Gamma_h^\top)}
{\mu-2C_{\rm gl}}.
\end{equation}
\end{theorem}
Thus noisy observations prevent exact synchronization in general, but the mean-square tracking error remains uniformly controlled by an explicit noise-dependent floor. In the noiseless case \(\Gamma_h=0\), one recovers global exponential synchronization under the stronger global condition \eqref{eq:global-mono}.

\begin{remark}[All-time stochastic tracking estimate]
Unlike Theorem~\ref{thm:conv}, we state Theorem~\ref{thm:stoch-tracking} for
all \(t\ge0\). The reason is that, in the deterministic setting, the local
monotonicity condition in Assumption~\ref{ass:H2}\emph{(ii)} can be used only after the trajectories enter the common post-absorption region from Corollary~\ref{cor:common-absorbing}, which introduces the waiting time \(T_\ast\). In the stochastic setting, however,
Proposition~\ref{prop:stoch-wp-diss} yields only a uniform \(H\)-moment bound, rather than a fixed deterministic absorbing ball on which the local monotonicity constant can be frozen. We therefore replace the post-absorption argument by the global condition in
Assumption~\ref{ass:H2-global}, which allows us to derive the mean-square estimate directly for all \(t\ge0\).
\end{remark}

\section{Nudging accuracy with surrogate dynamics}
\label{sec:surrogate}

In this section, we extend the theory from Section~\ref{sec:baseline} to the setting in which the nudged dynamics are driven by a surrogate model. We introduce the problem set-up in 
Subsection \ref{subsec:sur-setup}.
Subsection \ref{subsec:sur-cutoff} establishes global well-posedness of the nudged surrogate dynamics, followed by an analysis of its dissipative properties in 
Subsection \ref{subsec:sur-postabsorb}. The main result is Theorem \ref{thm:sur-track} in 
Subsection \ref{subsec:sur-track}, which establishes tracking accuracy. 
Subsection \ref{subsec:sur-noisy} generalizes the theory to the noisy-observation setting.

\subsection{True and surrogate model nudged systems}
\label{subsec:sur-setup}

Throughout this section, we consider the following setting:

\paragraph{True system}
As in the previous section, we assume that the true system is given by \begin{equation}\label{eq:sur-truth}
\dot u = F(u),
\qquad
u(0)=u_0\in\R^d,
\end{equation}
where the vector field $F(z):=f-\nu Az-N(z)$ satisfies Assumptions~\ref{ass:H1}, \ref{ass:H4}, and \ref{ass:H2}. We return to the compact notation of \eqref{eq:truthintro} by collecting the forcing term, the linear dissipative part, and the nonlinearity into the full exact vector field $F.$

\paragraph{Surrogate model nudged system}
We consider the nudged system
\begin{equation}\label{eq:sur-aot}
\dot v
=
F_M(v)-\mu\bigl(I_hv-I_hu\bigr),
\qquad
v(0)=v_0\in\R^d ,
\end{equation}
where \(I_h\) is an observation map satifying Assumption~\ref{ass:H3}, \(\mu>0\) is the nudging parameter, and \(F_M:\R^d\to\R^d\) is a surrogate drift. We refer to \eqref{eq:sur-aot} as the \emph{surrogate nudged system} and to the corresponding data assimilation procedure as \emph{surrogate AOT}. 

\bigskip

In this section we identify general and checkable residual conditions on the surrogate drift that ensure global well-posedness of the surrogate-based nudged dynamics and preserve tracking of the truth up to an explicit model-error floor. These conditions do not require the surrogate \(F_M\) to share the structural decomposition of the exact model \(F\) into forcing, linear dissipative, and nonlinear components.
Thus, the learned surrogate \(F_M\) can be viewed as an approximation of the full vector field \(F\), rather than of its individual components.

\subsection{Cutoff construction and global well-posedness}
\label{subsec:sur-cutoff}

In practice, a data-driven surrogate is typically trained only on a bounded
region of phase space visited by the trajectories of interest. One should
therefore not expect the learned model to be globally accurate. For the
tracking analysis below, however, global accuracy is not needed. What is needed
is a globally defined drift that agrees with the learned model on the relevant
region and remains well behaved outside it. We construct such a drift by
blending the learned local surrogate with a reference dissipative field through
a smooth cutoff.

Let \(\widehat F_M\) denote a learned local surrogate drift, defined at least on an open neighborhood of
\(\{z\in\R^d:\normV{z}\le R_{\mathrm{ext}}^+\}.\)
Fix radii
\[
R_{\mathrm{ext}}^{+}>R_{\mathrm{ext}}>0,
\]
and choose a smooth cutoff function
\(\psi\in C^\infty([0,\infty);[0,1])\) such that
\[
\psi(r)=1 \quad \text{for } r\le R_{\mathrm{ext}}^2,
\qquad
\psi(r)=0 \quad \text{for } r\ge (R_{\mathrm{ext}}^{+})^2.
\]
Define
\begin{equation}
    \label{eq:cutoff}
    \chi(z):=\psi(\normV{z}^2),
\qquad z\in\R^d.
\end{equation}
Then \(\chi\) is smooth and bounded, and satisfies
\[
\chi(z)=1 \quad \text{for } \normV{z}\le R_{\mathrm{ext}},
\qquad
\chi(z)=0 \quad \text{for } \normV{z}\ge R_{\mathrm{ext}}^{+}.
\]

Let \(F_{\mathrm{diss}}:\R^d\to\R^d\) be a reference dissipative drift. A
simple choice is
\begin{equation}\label{eq:f_diss}
F_{\mathrm{diss}}(z)=f-Bz,
\end{equation}
where \(B\in\R^{d\times d}\) is any fixed symmetric positive definite matrix.
For simplicity of analysis, in the remainder of this section we restrict
attention to this linear choice.

We then define the cutoff-extended surrogate drift by
\begin{equation}\label{eq:sur-FM-ext}
F_M(z)
:=
\chi(z)\,\widehat F_M(z)
+
\bigl(1-\chi(z)\bigr)\,F_{\mathrm{diss}}(z),
\qquad z\in\R^d.
\end{equation}
By construction,
\[
F_M(z)=\widehat F_M(z)
\qquad\text{whenever }\normV{z}\le R_{\mathrm{ext}},
\]
so the extension does not alter the learned model on the region where its
accuracy is intended to be used.

In view of Proposition~\ref{prop:wp}, the global well-posedness of the true system follows from Assumptions~\ref{ass:H1}, \ref{ass:H4}, and \ref{ass:H2}. The next result addresses the global well-posedness of the surrogate nudged system \eqref{eq:sur-aot} generated by the cutoff-extended drift \eqref{eq:sur-FM-ext}. Its proof is deferred to Appendix~\ref{app:sur-extension}.

\begin{proposition}[Global well-posedness under the cutoff construction]
\label{prop:sur-wp-cutoff}
Let Assumptions~\ref{ass:H1}, \ref{ass:H4}, \ref{ass:H2} and \ref{ass:H3} hold. Assume moreover that
\(\widehat F_M\) is locally Lipschitz on an open neighborhood of
\[
\{z\in\R^d:\normV{z}\le R_{\mathrm{ext}}^{+}\}.
\]
Let \(F_{\mathrm{diss}}\) be defined as in \eqref{eq:f_diss}, and let \(F_M\)
be defined by \eqref{eq:sur-FM-ext}. Then, for every initial pair
\((u_0,v_0)\in\R^d\times\R^d\), if \(u\) denotes the corresponding global
true solution of \eqref{eq:sur-truth}, the surrogate nudged system
\eqref{eq:sur-aot} admits a unique global solution
\[
v\in C^1([0,\infty);\R^d).
\]
\end{proposition}

\subsection{Post-absorption ball for tracking}
\label{subsec:sur-postabsorb}

For the tracking analysis, global well-posedness alone is not sufficient. We also need a bounded region in phase space that eventually contains both the true trajectory and the surrogate nudged trajectory, so that the local properties of the exact and surrogate drifts can be applied. This is the surrogate analogue of the post-absorption regime used in the exact-model analysis of Section~\ref{sec:baseline}. The next proposition establishes such an absorbing region for the cutoff-extended surrogate dynamics under a compatibility condition between the exterior dissipation and the nudging strength. The proofs of the results in this subsection are deferred to Appendix~\ref{app:sur-extension}.

\begin{proposition}[Absorbing ball for the surrogate dynamics]
\label{prop:sur-absorb}
Suppose Assumptions~\ref{ass:H1}, \ref{ass:H4}, \ref{ass:H2}, and
\ref{ass:H3} are satisfied. Suppose moreover that \(F_{\mathrm{diss}}\) is defined as in \eqref{eq:f_diss}, namely \(F_{\mathrm{diss}}(z)=f-Bz,\, z\in\R^d,\) where \(B\in\R^{d\times d}\) is symmetric positive definite. Let
\(\lambda_B:=\lambda_{\min}(B)\), and assume that
\begin{equation}\label{eq:sur-betaB}
\lambda_B-\mu C_I>0.
\end{equation}
Then, for every initial pair \((u_0,v_0)\in\R^d\times\R^d\), if \(u\) and
\(v\) denote the corresponding global solutions of \eqref{eq:sur-truth} and
\eqref{eq:sur-aot}, respectively, there exist constants \(T_M\ge0\) and
\(R_M^H>0\), depending only on the system parameters, the cutoff construction,
and the initial data, such that
\begin{equation}\label{eq:sur-H-absorb}
\sup_{t\ge T_M}\normH{v(t)}\le R_M^H.
\end{equation}
Consequently, there exists a constant \(R_M>0\) such that
\begin{equation}\label{eq:sur-V-absorb}
\sup_{t\ge T_M}\normV{v(t)}\le R_M.
\end{equation}
\end{proposition}

Combining the preceding proposition with the dissipativity of the true system established in
Section~\ref{sec:baseline}, we obtain a common post-absorption ball for the
true and surrogate trajectories.

\begin{corollary}[Common post-absorption ball]
\label{cor:sur-common-ball}
Under the assumptions and notation
of Proposition~\ref{prop:sur-absorb}, there exist constants \(T_\ast\ge0\) and
\(R_\ast>0\) such that
\begin{equation}\label{eq:sur-common-ball}
\sup_{t\ge T_\ast}\normV{u(t)}\le R_\ast,
\qquad
\sup_{t\ge T_\ast}\normV{v(t)}\le R_\ast.
\end{equation}
\end{corollary}

The values of \(T_\ast\) and \(R_\ast\) here may differ from the correponding constants used in the exact-model analysis of Section~2; see Corollary~\ref{cor:common-absorbing}. For simplicity, we use this notation generically throughout the paper for the waiting time and radius of the relevant common post-absorption ball, both in the exact-model AOT and surrogate AOT settings. 
We define the corresponding post-absorption ball by
\begin{equation}\label{eq:sur-Bstar}
\cB_\ast:=\{z\in\R^d:\normV{z}\le R_\ast\}.
\end{equation}
In what follows,  we require the cutoff radii to be chosen so that
\begin{equation}\label{eq:choice-of-radii}
R_\ast\le R_{\mathrm{ext}}.
\end{equation}
 This condition ensures that the post-absorption region used in the tracking analysis lies inside the region where the cutoff is inactive. The cutoff construction is designed so that the exterior dissipative field provides global control outside the cutoff region, while the tracking error analysis uses the learned surrogate only on the post-absorption ball. In the numerical experiments, this post-absorption region is estimated from the long-time dynamics as an empirical proxy. Since \(R_{\rm ext}\) and the exterior dissipation matrix \(B\) are design parameters, the cutoff radii are chosen so that this empirical region lies inside the cutoff-inactive set.
Hence the cutoff is inactive on \(\cB_\ast\), and
\[
F_M\equiv \widehat F_M
\qquad\text{on }\cB_\ast.
\]
Therefore, the residual conditions used below are purely local: actually they are conditions on the learned surrogate \(\widehat F_M\) only on the
post-absorption ball.

We next record two structural estimates that will be used repeatedly. Both
follow directly from our assumptions on the observation model and the true dynamics.

\begin{lemma}[Lower bound for the feedback term]
\label{lem:ih-lower-finitedim}
Let Assumption~\ref{ass:H3} hold. Then, for every \(z\in\R^d\),
\begin{equation}\label{eq:ih-lower-finitedim}
\ip{I_hz}{z}_H
\ge
\frac12\,\normH{z}^2
-
\frac{c_0^2h^2}{2}\,\normV{z}^2.
\end{equation}
\end{lemma}

\begin{lemma}[Squeezing property of the exact drift on \(\cB_\ast\)]
\label{lem:sur-squeeze-F}
Let Assumptions~\ref{ass:H1} and \ref{ass:H2}\emph{(ii)} hold. Let
\(\cB_\ast\) be defined by \eqref{eq:sur-Bstar}, and set
\[
R_H^\ast:=\frac{R_\ast}{\sqrt{\lambda_1}}.
\]
Then there exists a constant \(C_{\mathrm{sq}}:=C_{R_H^\ast}\ge0\) such that,
for all \(z,z'\in\cB_\ast\),
\begin{equation}\label{eq:sur-squeeze-F}
\ip{F(z')-F(z)}{z'-z}_H
\le
-\nu\,\normV{z'-z}^2
+
C_{\mathrm{sq}}\,\normH{z'-z}^2.
\end{equation}
\end{lemma}

We now introduce the only new quantities needed for the
surrogate tracking analysis: the sup-norm and Lipschitz constant of the residual $r_M: = F_M-F$ on the post-absorption ball.

\begin{definition}[Surrogate residual on \(\cB_\ast\)]
\label{def:sur-residual}
Let
\[
r_M:=F_M-F:\R^d\to\R^d.
\]
We define
\begin{equation}\label{eq:sur-residual-def}
\delta_M
:=
\sup_{z\in\cB_\ast}\normH{r_M(z)},
\qquad
\ell_M
:=
\Lip(r_M;\cB_\ast)
:=
\sup_{\substack{z,z'\in\cB_\ast\\ z\neq z'}}
\frac{\normH{r_M(z')-r_M(z)}}{\normH{z'-z}}.
\end{equation}
\end{definition}

The next lemma shows that the surrogate model inherits the squeezing property from the true dynamics.
\begin{lemma}[Inherited squeezing for \(F_M\) on \(\cB_\ast\)]
\label{lem:sur-squeeze-FM}
Retain the hypotheses and notation of Lemma~\ref{lem:sur-squeeze-F}, and let \(\ell_M\) be as in Definition~\ref{def:sur-residual}. Then, for all \(z,z'\in\cB_\ast\),
\begin{equation}\label{eq:sur-squeeze-FM}
\ip{F_M(z')-F_M(z)}{z'-z}_H
\le
-\nu\,\normV{z'-z}^2
+
\bigl(C_{\mathrm{sq}}+\ell_M\bigr)\,\normH{z'-z}^2.
\end{equation}
\end{lemma}
This lemma will be used to establish tracking accuracy guarantees under model error. 

\subsection{Tracking under model error}
\label{subsec:sur-track}

The main result of this section is the following:

\begin{theorem}[Exponential tracking up to model-error floor]
\label{thm:sur-track}
Suppose that Assumptions~\ref{ass:H1}, \ref{ass:H4},  \ref{ass:H2}, and \ref{ass:H3}  hold. Let \(F_M\) be the cutoff-extended surrogate drift defined in \eqref{eq:sur-FM-ext}, with exterior drift \(F_{\mathrm{diss}}(z)=f-Bz,\) where \(B\in\R^{d\times d}\) is symmetric positive definite, and with local learned surrogate \(\widehat F_M\) locally Lipschitz on an open neighborhood of \(\{z\in\R^d:\normV{z}\le R_{\mathrm{ext}}^+\}.\) Set \(\lambda_B:=\lambda_{\min}(B)\), and suppose that
\[
\lambda_B-\mu C_I>0,
\]
where \(C_I\) is the \(H\)-operator bound for \(I_h\) from Remark~\ref{rem:Ih-bounded}. For arbitrary \(u_0,v_0\in\R^d\), let \(u\) and \(v\) be the corresponding
solutions of \eqref{eq:sur-truth} and \eqref{eq:sur-aot}, respectively. Let \(T_\ast\), \(R_\ast\), and \(\cB_\ast\) be the corresponding post-absorption time, radius, and ball from Corollary~\ref{cor:sur-common-ball}, and let \(C_{\mathrm{sq}}\) be the constant from Lemma~\ref{lem:sur-squeeze-F}. Suppose that \(\delta_M,\ell_M<\infty\), where \(\delta_M\) and \(\ell_M\) are as in Definition~\ref{def:sur-residual}. If
\begin{equation}\label{eq:sur-cond-mu-h}
\mu>2\bigl(C_{\mathrm{sq}}+\ell_M\bigr),
\qquad
\mu c_0^2h^2<\nu,
\end{equation}
then, with
\[
\gamma_M:=\mu-2\bigl(C_{\mathrm{sq}}+\ell_M\bigr)>0,
\qquad
\nu_{\mathrm{eff}}:=\nu-\frac{\mu c_0^2h^2}{2}>0,
\]
for every \(t\ge T_\ast\),
\begin{equation}\label{eq:sur-track-bound}
\|v(t)-u(t)\|_H^2
\le
e^{-\gamma_M(t-T_\ast)}\|v(T_\ast)-u(T_\ast)\|_H^2
+
\frac{\delta_M^2}{\lambda_1\nu_{\mathrm{eff}}\gamma_M}
\Bigl(1-e^{-\gamma_M(t-T_\ast)}\Bigr).
\end{equation}
Consequently,
\[
\limsup_{t\to\infty}\|v(t)-u(t)\|_H^2
\le
\frac{\delta_M^2}{\lambda_1\nu_{\mathrm{eff}}\gamma_M}.
\]
\end{theorem}

\begin{proof}
Set
\(w:=v-u.\)
Subtracting \eqref{eq:sur-truth} from \eqref{eq:sur-aot} gives
\[
\dot w
=
F_M(v)-F(u)-\mu I_hw.
\]
Taking the \(H\)-inner product with \(w\) yields
\begin{equation}\label{eq:sur-track-energy}
\frac12\frac{d}{dt}\normH{w}^2
=
\ip{F_M(v)-F(u)}{w}_H
-\mu\,\ip{I_hw}{w}_H.
\end{equation}
Since \(r_M=F_M-F\), we can write
\[
\ip{F_M(v)-F(u)}{w}_H
=
\ip{F_M(v)-F_M(u)}{w}_H+\ip{r_M(u)}{w}_H.
\]

For \(t\ge T_\ast\), Corollary~\ref{cor:sur-common-ball} ensures that
\(u(t),v(t)\in\cB_\ast\). Hence Lemma~\ref{lem:sur-squeeze-FM} implies
\begin{equation}\label{eq:sur-FM-sq-use}
\ip{F_M(v)-F_M(u)}{w}_H
\le
-\nu\,\normV{w}^2
+
\bigl(C_{\mathrm{sq}}+\ell_M\bigr)\,\normH{w}^2.
\end{equation}
Moreover, by Cauchy--Schwarz and the definition of \(\delta_M\),
\begin{equation}\label{eq:sur-resid-term}
\ip{r_M(u)}{w}_H
\le
\normH{r_M(u)}\,\normH{w}
\le
\delta_M\,\normH{w}
\le
\lambda_1^{-1/2}\delta_M\,\normV{w},
\end{equation}
where we used \(\normH{w}\le \lambda_1^{-1/2}\normV{w}\), which follows from
\eqref{eq:norm-equivalence}.

Finally, Lemma~\ref{lem:ih-lower-finitedim} yields
\begin{equation}\label{eq:sur-ih-use}
-\mu\,\ip{I_hw}{w}_H
\le
-\frac{\mu}{2}\normH{w}^2
+\frac{\mu c_0^2h^2}{2}\normV{w}^2.
\end{equation}

Substituting \eqref{eq:sur-FM-sq-use}--\eqref{eq:sur-ih-use} into
\eqref{eq:sur-track-energy}, we obtain for \(t\ge T_\ast\),
\[
\frac12\frac{d}{dt}\normH{w}^2
\le
-\Bigl(\nu-\frac{\mu c_0^2h^2}{2}\Bigr)\normV{w}^2
-\Bigl(\frac{\mu}{2}-\bigl(C_{\mathrm{sq}}+\ell_M\bigr)\Bigr)\normH{w}^2
+\lambda_1^{-1/2}\delta_M\,\normV{w}.
\]
Introduce
\[
\nu_{\mathrm{eff}}:=\nu-\frac{\mu c_0^2h^2}{2}>0,
\qquad
\gamma_M:=\mu-2\bigl(C_{\mathrm{sq}}+\ell_M\bigr)>0.
\]
Then the inequality becomes
\[
\frac12\frac{d}{dt}\normH{w}^2
\le
-\nu_{\mathrm{eff}}\normV{w}^2
-\frac{\gamma_M}{2}\normH{w}^2
+\lambda_1^{-1/2}\delta_M\,\normV{w}.
\]
Applying Young's inequality to the last term with parameter
\(\nu_{\mathrm{eff}}\), we obtain
\[
\lambda_1^{-1/2}\delta_M\,\normV{w}
\le
\frac{\nu_{\mathrm{eff}}}{2}\normV{w}^2
+\frac{\delta_M^2}{2\lambda_1\nu_{\mathrm{eff}}}.
\]
Substituting this bound and discarding the remaining nonpositive term
\(-\frac{\nu_{\mathrm{eff}}}{2}\normV{w}^2\), we arrive at
\begin{equation}
\label{eq:sur-track-error}
    \frac{d}{dt}\normH{w}^2
\le
-\gamma_M\,\normH{w}^2
+\frac{\delta_M^2}{\lambda_1\nu_{\mathrm{eff}}},
\qquad t\ge T_\ast.
\end{equation}
By Gronwall's inequality applied on \([T_\ast,t]\), this scalar differential
inequality gives
\[
\normH{w(t)}^2
\le
e^{-\gamma_M(t-T_\ast)}\normH{w(T_\ast)}^2
+
\frac{\delta_M^2}{\gamma_M\lambda_1\nu_{\mathrm{eff}}}
\bigl(1-e^{-\gamma_M(t-T_\ast)}\bigr),
\qquad t\ge T_\ast.
\]
In particular,
\[
\normH{w(t)}^2
\le
e^{-\gamma_M(t-T_\ast)}\normH{w(T_\ast)}^2
+
\frac{\delta_M^2}{\gamma_M\lambda_1\nu_{\mathrm{eff}}},
\qquad t\ge T_\ast,
\]
which gives \eqref{eq:sur-track-bound}.
\end{proof}

\begin{remark}[Exact-model limit]\label{rem:sur-exact}
If \(F_M\equiv F\) on \(\cB_\ast\), then \(\delta_M=\ell_M=0\), and
\eqref{eq:sur-track-bound} reduces to exponential synchronization in the
\(H\)-norm.
\end{remark}

\begin{remark}[Regularity for well-posedness and residual control for tracking]
\label{rem:sur-requirement}
It is useful to distinguish two roles of the surrogate assumptions. The local Lipschitz regularity of \(\widehat F_M\) on a neighborhood of \(\{z\in\R^d:\normV{z}\le R_{\mathrm{ext}}^+\}\) is used in
Proposition~\ref{prop:sur-wp-cutoff} to ensure global well-posedness of the cutoff-extended surrogate system. By contrast, the model-error contribution in the tracking estimate of Theorem~\ref{thm:sur-track} is determined only by the residual quantities \(\delta_M\) and \(\ell_M\) on the post-absorption ball \(\cB_\ast\). Thus, in the later applications and complexity analysis in Section~\ref{sec:complexity}, we focus on controlling the model error on \(\cB_\ast\), while the required local Lipschitz regularity on the larger cutoff region can be easily guaranteed by the choice of hypothesis class.
\end{remark}

\subsection{Noisy observations and stochastic surrogate nudging}
\label{subsec:sur-noisy}

We conclude this section by recording the corresponding noisy-observation
extension of the surrogate nudging framework. We retain the observation model
introduced in Subsection~\ref{subsec:baseline-noisy}, in which the exact coarse
observable \(I_hu(t)\) is no longer directly accessible. Instead, after
reconstruction, the available feedback takes the form
\begin{equation}\label{eq:sur-noisy-feedback}
\widetilde I_hu(t)\,dt
=
I_hu(t)\,dt+\Gamma_h\,dW_t,
\qquad
\Gamma_h:=R_h\Sigma\in\R^{d\times k},
\end{equation}
where \(W\) is a standard \(k\)-dimensional Brownian motion and
\(\Sigma\in\R^{k\times k}\) is the observation-noise covariance factor.

Under this observation model, the stochastic surrogate nudged system is given by
\begin{equation}\label{eq:sur-aot-noisy}
dv
=
\Bigl(
F_M(v)-\mu\bigl(I_hv-I_hu\bigr)
\Bigr)\,dt
+
\mu \Gamma_h\,dW_t,
\qquad
v(0)=v_0\in\R^d,
\end{equation}
which is the natural noisy counterpart of the deterministic surrogate nudged
system \eqref{eq:sur-aot}. Here the observation error enters, after
reconstruction, as an additive stochastic forcing term in the nudged
dynamics.

We first show a stochastic boundedness result for the surrogate dynamics. In
contrast with the deterministic setting, one can no longer expect a fixed
deterministic absorbing ball based solely on local information. For this
reason, an additional global requirement on the surrogate drift is needed in
order to derive a uniform second-moment bound.

\begin{assumption}[Global dissipativity of the surrogate drift]
\label{ass:sur-global-diss}
There exist constants \(\alpha_M>0\) and \(\beta_M\ge0\) such that
\begin{equation}\label{eq:sur-global-diss}
\ip{F_M(z)}{z}_H
\le
\beta_M-\alpha_M\normH{z}^2,
\qquad
\forall z\in\R^d.
\end{equation}
\end{assumption}

\begin{remark}[Need for an additional global surrogate assumption]
Assumption~\ref{ass:sur-global-diss} is different from Assumption~\ref{ass:H4}. The latter is a dissipativity condition for the true system, whereas here we impose a global dissipativity condition directly on the surrogate drift \(F_M\). In the previous deterministic surrogate analysis, the relevant estimates were local, since they were used only on the post-absorption ball \(\cB_\ast\). In the present stochastic setting, however, one does not a priori have a fixed deterministic bounded region that contains the surrogate trajectory. For this reason, local control on \(\cB_\ast\) is not sufficient to derive a uniform \(H\)-moment bound for \(v\), and a global dissipativity condition on \(F_M\) is imposed instead.
\end{remark}

\begin{proposition}[Global strong well-posedness and \(H\)-moment boundedness]
\label{prop:sur-wp-noisy-bdd}
Suppose that the hypotheses and notation of
Proposition~\ref{prop:sur-wp-cutoff} are in force, and that
Assumption~\ref{ass:sur-global-diss} holds. Let \(W\) be a standard
\(k\)-dimensional Brownian motion, and let \(\Gamma_h\in\R^{d\times k}\) be as
in \eqref{eq:sur-noisy-feedback}. Suppose moreover that
\begin{equation}\label{eq:sur-noisy-bdd-cond}
\vartheta_M:=\alpha_M-\mu C_I>0,
\end{equation}
where \(C_I\) is the \(H\)-operator bound from
Remark~\ref{rem:Ih-bounded}. Then, for every initial pair
\((u_0,v_0)\in\R^d\times\R^d\), the stochastic surrogate nudged system
\eqref{eq:sur-aot-noisy} admits a unique global adapted strong solution
\[
v\in C([0,\infty);\R^d)
\qquad \text{a.s.}
\]
Moreover, for every \(t\ge0\),
\begin{equation}
\mathbb E\normH{v(t)}^2
\le
e^{-\vartheta_M t}\normH{v_0}^2
+
\int_0^t e^{-\vartheta_M(t-s)}
\left(
2\beta_M
+\frac{\mu^2 C_I^2}{\vartheta_M}\normH{u(s)}^2
+\mu^2\Tr(\Gamma_h\Gamma_h^\top)
\right)\,ds.
\label{eq:sur-noisy-moment-bound}
\end{equation}
In particular, there exist constants \(T_M^{\rm sto}\ge0\) and
\(R_M^{\rm sto}>0\), depending only on the system parameters, the surrogate
drift, the noise level, and the initial data, such that
\begin{equation}\label{eq:sur-noisy-moment-absorb}
\sup_{t\ge T_M^{\rm sto}}\mathbb E\normH{v(t)}^2
\le
(R_M^{\rm sto})^2.
\end{equation}
\end{proposition}

To derive a global tracking estimate, noting that the stochastic surrogate trajectory does not remain in a fixed deterministic post-absorption ball, we strengthen the structural assumption on the surrogate drift from a local squeezing condition to a global one.

\begin{assumption}[Global squeezing of the surrogate drift]
\label{ass:sur-global-sq}
There exists a constant \(C_{M,\mathrm{gl}}\ge0\) such that, for every
\(z,z'\in\R^d\),
\begin{equation}\label{eq:sur-global-sq}
\ip{F_M(z')-F_M(z)}{z'-z}_H
\le
-\nu\,\normV{z'-z}^2
+
C_{M,\mathrm{gl}}\normH{z'-z}^2.
\end{equation}
\end{assumption}

Under this additional assumption, one obtains a noisy analogue of
Theorem~\ref{thm:sur-track} without imposing a separate post-absorption
assumption on the stochastic surrogate trajectory.
To control the residual, we still use the same local quantity \(\delta_M\) from Definition~\ref{def:sur-residual},
since it is evaluated only along the true trajectory, which enters
\(\cB_\ast\) after time \(T_\ast\). The proof of Theorem~\ref{thm:sur-track-noisy} is deferred to Appendix~\ref{app:sur-noisy}.

\begin{theorem}[Mean-square tracking up to model and noise floors]
\label{thm:sur-track-noisy}
Suppose that the hypotheses of Proposition~\ref{prop:sur-wp-noisy-bdd} hold,
and that the global squeezing condition
Assumption~\ref{ass:sur-global-sq} is satisfied. Let \(u\) and \(v\) be the
corresponding solutions of \eqref{eq:sur-truth} and
\eqref{eq:sur-aot-noisy}, respectively. Let \(T_\ast\) and \(\cB_\ast\) be as
in Corollary~\ref{cor:sur-common-ball}, and suppose that the residual
quantity \(\delta_M\) from Definition~\ref{def:sur-residual} is finite. If
\begin{equation}\label{eq:sur-cond-mu-h-noisy}
\mu>2C_{M,\mathrm{gl}},
\qquad
\mu c_0^2h^2<\nu,
\end{equation}
then, with
\[
\gamma_{M,\mathrm{gl}}:=\mu-2C_{M,\mathrm{gl}}>0,
\qquad
\nu_{\mathrm{eff}}:=\nu-\frac{\mu c_0^2h^2}{2}>0,
\]
the following estimate holds for every \(t\ge T_\ast\):
\begin{equation}
\mathbb E\|v(t)-u(t)\|_H^2
\le
e^{-\gamma_{M,\mathrm{gl}}(t-T_\ast)}
\mathbb E\|v(T_\ast)-u(T_\ast)\|_H^2
+
\frac{1-e^{-\gamma_{M,\mathrm{gl}}(t-T_\ast)}}{\gamma_{M,\mathrm{gl}}}
\left(
\frac{\delta_M^2}{\lambda_1\nu_{\mathrm{eff}}}
+
\mu^2\Tr(\Gamma_h\Gamma_h^\top)
\right).
\label{eq:sur-track-bound-noisy}
\end{equation}
Consequently,
\begin{equation}\label{eq:sur-track-floor-noisy}
\limsup_{t\to\infty}\mathbb E\|v(t)-u(t)\|_H^2
\le
\frac{1}{\gamma_{M,\mathrm{gl}}}
\left(
\frac{\delta_M^2}{\lambda_1\nu_{\mathrm{eff}}}
+
\mu^2\Tr(\Gamma_h\Gamma_h^\top)
\right).
\end{equation}
\end{theorem}

Theorem \ref{thm:sur-track-noisy} shows that the stochastic surrogate nudged dynamics tracks the truth up to two
explicit contributions: a model-error floor, determined by \(\delta_M\), and a
noise-induced floor, determined by \(\Tr(\Gamma_h\Gamma_h^\top)\). In the
noise-free case \(\Gamma_h=0\), the estimate reduces to the corresponding
deterministic-type tracking bound with the global squeezing constant
\(C_{M,\mathrm{gl}}\).

\begin{remark}[Persistence of the waiting time]
In contrast to Theorem~\ref{thm:stoch-tracking}, here we do not state Theorem~\ref{thm:sur-track-noisy} for all \(t\ge0\), but only for \(t\ge T_\ast\). The reason is that, although
Assumption~\ref{ass:sur-global-sq} removes the need to place the stochastic surrogate trajectory in a fixed deterministic post-absorption ball, the residual quantity \(\delta_M\) in Definition~\ref{def:sur-residual} remains a
local quantity defined on \(\cB_\ast\). Since this residual term is evaluated
along the true trajectory, one still needs to wait until \(u(t)\) enters
\(\cB_\ast\), which is guaranteed only after time \(T_\ast\) by
Corollary~\ref{cor:sur-common-ball}. For this reason, the mean-square
tracking estimate in Theorem~\ref{thm:sur-track-noisy} is stated only for
\(t\ge T_\ast\).
\end{remark}

\section{From learning errors to nudging accuracy}
\label{sec:bridge}

In Section~\ref{sec:surrogate}, the surrogate tracking theorem, Theorem~\ref{thm:sur-track}, is formulated in terms of the deterministic residual quantities
\[
\delta_M := \sup_{z\in\cB_\ast}\normH{F_M(z)-F(z)},
\qquad
\ell_M := \Lip(F_M-F;\cB_\ast),
\]
defined on the post-absorption ball \(\cB_\ast\); see Definition~\ref{def:sur-residual}. As discussed in Remark~\ref{rem:sur-requirement}, the learned surrogate is required to be regular on the larger cutoff region, whereas the model-error contribution to tracking is quantified only on \(\cB_\ast\). The purpose of this section is to relate the abstract residual quantities $\delta_M$ and $\ell_M$ to concrete learning errors through two surrogate-construction routes.
 The learning-error quantities introduced below for the two learning routes should be understood as theoretical measures of surrogate accuracy. They are not required to be directly observable or exactly computable in practice. Rather, they provide a bridge between the abstract residual quantities \(\delta_M\) and \(\ell_M\) in the tracking analysis and more standard approximation errors arising from offline surrogate learning.
In \emph{vector-field learning}, one constructs a local learned surrogate \(\widehat F_M\) for the drift \(F\); see Subsection~\ref{subsec:bridge-route1}. Since the cutoff extension in Subsection~\ref{subsec:sur-cutoff} satisfies \(F_M=\widehat F_M\) on \(\cB_\ast\), the residual quantities $\delta_M$ and $\ell_M$ can be bounded directly by the value and Lipschitz errors of \(\widehat F_M\) on \(\cB_\ast\).

In \emph{solution-map learning}, one learns a short-time solution map and uses a first-order difference quotient to construct a local surrogate drift \(\widehat F_M\); see Subsection~\ref{subsec:bridge-route2}. In this case, an additional bridging argument is needed to convert map-level learning errors into drift residual bounds.

\subsection{Vector-field learning}
\label{subsec:bridge-route1}

We first consider the case in which the drift \(F\) is learned directly. This route is natural when prior structural information is available, for instance when \(F\) is known or expected to belong to a low-complexity hypothesis class such as a polynomial space, a trigonometric dictionary, or a sparse feature model.

Following the cutoff construction in Subsection~\ref{subsec:sur-cutoff}, let \(\widehat F_M\) be a local learned approximation of \(F\), constructed from \(M\) training samples and defined on a neighborhood of
\[
\bigl\{z\in\R^d:\normV{z}\le R_{\mathrm{ext}}^+ \bigr\}.
\]
We assume that \(\widehat F_M\) is locally Lipschitz on this neighborhood, as required in Proposition~\ref{prop:sur-wp-cutoff}. The corresponding global surrogate drift \(F_M\) is obtained from \(\widehat F_M\) by the cutoff extension \eqref{eq:sur-FM-ext}. In particular, by the choice of radii in \eqref{eq:choice-of-radii}, one has \(F_M=\widehat F_M\) on \(\cB_\ast\).

We quantify the direct vector-field learning error on \(\cB_\ast\) by
\begin{equation}\label{eq:bridge-direct-errors}
\varepsilon_M^F
:=
\sup_{z\in\cB_\ast}\normH{\widehat F_M(z)-F(z)},
\qquad
\eta_M^F
:=
\sup_{z\in\cB_\ast}
\bigl\|D\widehat F_M(z)-DF(z)\bigr\|_{\mathcal{L}(H,H)}.
\end{equation}

The next proposition shows that, in this direct-learning setting, the residual quantities from Definition~\ref{def:sur-residual} are controlled directly by the value and Jacobian errors of the learned drift on \(\cB_\ast\). The proof is provided in Appendix~\ref{app:bridge}.

\begin{proposition}[Direct learning errors imply AOT-valid residual bounds]
\label{prop:bridge-route1}
Assume that \(\widehat F_M\in C^1(\cB_\ast;H)\), and let \(F_M\) denote a global cutoff extension satisfying \(F_M\equiv \widehat F_M\) on \(\cB_\ast\). Then the residual quantities in Definition~\ref{def:sur-residual} satisfy
\begin{equation}\label{eq:bridge-route1-residual}
\delta_M\le \varepsilon_M^F,
\qquad
\ell_M\le \eta_M^F.
\end{equation}
Consequently, if
\[
\mu>2\bigl(C_{\mathrm{sq}}+\eta_M^F\bigr),
\qquad
\mu c_0^2 h^2<\nu,
\]
then the hypotheses of Theorem~\ref{thm:sur-track} are satisfied. In particular, the corresponding surrogate nudged dynamics converges exponentially to the truth up to the error floor obtained from Theorem~\ref{thm:sur-track}, with \(\delta_M\) and \(\ell_M\) replaced by \(\varepsilon_M^F\) and \(\eta_M^F\), respectively.
\end{proposition}

Proposition~\ref{prop:bridge-route1} is useful because it makes the residual conditions from Section~\ref{sec:surrogate} directly checkable from standard learning quantities on \(\cB_\ast\). Once one has uniform control of both the vector-field error and its Jacobian error, the tracking theorem can be immediately applied.

\subsection{Solution-map learning and induced surrogate drift}
\label{subsec:bridge-route2}

We next consider the second route, in which one learns a short-time solution map rather than the drift itself. This route is natural when training data are available as short trajectories, or when the dynamics are regular but no reliable parametric representation of the drift is available.

For \(z\in H\), let \(u(t;z)\) denote the state at time \(t\) of the exact dynamics initialized at \(z\), namely
\begin{equation}\label{eq:u-t-z}
\dot u(t;z)=F\bigl(u(t;z)\bigr),
\qquad
u(0;z)=z.
\end{equation}
By Proposition~\ref{prop:wp}, this solution is well defined for all \(t\ge0\). For a fixed lag \(\Delta t>0\), we define the exact lag-\(\Delta t\) solution map by
\begin{equation}\label{eq:solution-map}
S_{\Delta t}(z):=u(\Delta t;z),
\qquad z\in H.
\end{equation}

Let \(\widehat S_{\Delta t}^{(M)}\) be a learned local approximation of \(S_{\Delta t}\), constructed from \(M\) training samples and defined on a neighborhood of
\[
\bigl\{z\in\R^d:\normV{z}\le R_{\mathrm{ext}}^+ \bigr\}.
\]
We assume that \(\widehat S_{\Delta t}^{(M)}\) is sufficiently regular on this neighborhood so that the locally induced drift defined below satisfies the
regularity required in Proposition~\ref{prop:sur-wp-cutoff}. We then define the locally induced surrogate drift by
\begin{equation}\label{eq:bridge-def-FhatM}
\widehat F_M(z)
:=
\frac{\widehat S_{\Delta t}^{(M)}(z)-z}{\Delta t},
\qquad
\normV{z}\le R_{\mathrm{ext}}^+ .
\end{equation}
The global surrogate dynamics used in the continuous-time surrogate nudged system are then generated by the cutoff-extended drift \(F_M\), which is obtained from the locally induced surrogate drift \(\widehat F_M\) through the construction in Subsection~\ref{subsec:sur-cutoff}; see \eqref{eq:sur-FM-ext}. Accordingly, the global object entering the surrogate nudged dynamics is the drift \(F_M\), rather than a separately defined global surrogate solution map.

The solution-map learning errors relevant for tracking are measured on
\(\cB_\ast\):
\begin{equation}\label{eq:bridge-operator-errors}
\varepsilon_M^S
:=
\sup_{z\in\cB_\ast}
\normH{\widehat S_{\Delta t}^{(M)}(z)-S_{\Delta t}(z)},
\qquad
\eta_M^S
:=
\sup_{z\in\cB_\ast}
\bigl\|D\widehat S_{\Delta t}^{(M)}(z)-DS_{\Delta t}(z)\bigr\|_{\mathcal{L}(H,H)}.
\end{equation}
To convert these solution-map learning errors into residual bounds for \(\widehat F_M\), we use a first-order expansion of the exact flow on a slightly larger neighborhood of \(\cB_\ast\).

\subsection{From solution-map learning errors to residual bounds}
\label{subsec:bridge-flow}

We now state the bridge connecting the solution-map learning errors in \eqref{eq:bridge-operator-errors} to the residual quantities in Definition~\ref{def:sur-residual}. The key ingredient is a first-order flow expansion on a slightly larger ball.

\begin{lemma}[First-order flow expansion on \(\cB_\ast\)]
\label{lem:bridge-flow-expansion}
Fix \(R_\ast^+>R_\ast\), and define
\[
\cB_\ast^+:= \bigl\{z\in H:\normV{z}<R_\ast^+ \bigr\}.
\]
Assume that, for every \(z\in\cB_\ast\), the exact trajectory remains in \(\cB_\ast^+\) over the interval \([0,\Delta t]\), namely
\begin{equation}\label{eq:bridge-stay-in-ball}
S_\tau(z)\in \cB_\ast^+,
\qquad 0\le \tau\le \Delta t.
\end{equation}
Assume moreover that \(F\in C^{1,1}(\cB_\ast^+;H)\), with \(F\) and \(DF\) uniformly bounded on \(\overline{\cB_\ast^+}\), and with \(DF\) Lipschitz on \(\overline{\cB_\ast^+}\). Then there exists a constant \(C_{\mathrm{flow}}<\infty\), depending only on these bounds and on \(\Delta t\), such that, for every \(z\in\cB_\ast\),
\begin{equation}\label{eq:bridge-flow-expansion}
S_{\Delta t}(z)
=
z+\Delta t\,F(z)+R_{\Delta t}(z),
\qquad
\sup_{z\in\cB_\ast}\normH{R_{\Delta t}(z)}
\le
C_{\mathrm{flow}}\Delta t^2,
\end{equation}
and
\begin{equation}\label{eq:bridge-Dflow-expansion}
DS_{\Delta t}(z)
=
I+\Delta t\,DF(z)+\widetilde R_{\Delta t}(z),
\qquad
\sup_{z\in\cB_\ast}
\bigl\|\widetilde R_{\Delta t}(z)\bigr\|_{\mathcal{L}(H,H)}
\le
C_{\mathrm{flow}}\Delta t^2.
\end{equation}
\end{lemma}

The lemma above isolates the deterministic truncation error associated with reconstructing the drift from a lag-\(\Delta t\) flow map. Combining this truncation error with the solution-map learning errors
\(\varepsilon_M^S\) and \(\eta_M^S\) gives the residual bounds required in Theorem~\ref{thm:sur-track}.

\begin{proposition}[Solution-map learning errors imply residual bounds]
\label{prop:bridge-route2}
Assume the hypotheses of Lemma~\ref{lem:bridge-flow-expansion}. Suppose that \(\widehat S_{\Delta t}^{(M)}\in C^1(\cB_\ast;H)\), and define \(\widehat F_M\) by \eqref{eq:bridge-def-FhatM}. Then the residual quantities in Definition~\ref{def:sur-residual} satisfy
\begin{equation}\label{eq:bridge-delta-bound}
\delta_M
\le
\frac{\varepsilon_M^S}{\Delta t}
+
C_{\mathrm{flow}}\Delta t,
\end{equation}
and
\begin{equation}\label{eq:bridge-ell-bound}
\ell_M
\le
\frac{\eta_M^S}{\Delta t}
+
C_{\mathrm{flow}}\Delta t.
\end{equation}
Consequently, under the remaining structural and cutoff assumptions of Theorem~\ref{thm:sur-track}, if
\[
\mu>
2\left(
C_{\mathrm{sq}}
+
\frac{\eta_M^S}{\Delta t}
+
C_{\mathrm{flow}}\Delta t
\right),
\qquad
\mu c_0^2 h^2<\nu,
\]
then the residual condition in Theorem~\ref{thm:sur-track} is satisfied. In particular, the corresponding surrogate nudged dynamics converges exponentially to the truth up to the error floor from Theorem~\ref{thm:sur-track}, with \(\delta_M\) and \(\ell_M\) controlled by
\eqref{eq:bridge-delta-bound}--\eqref{eq:bridge-ell-bound}.
\end{proposition}

\begin{remark}[Regularity of \(F\) on an enlarged region]
\label{rem:Bstar-plus-role}
The regularity requirements in Lemma~\ref{lem:bridge-flow-expansion} and Proposition~\ref{prop:bridge-route2} are imposed on different regions for different purposes. In Lemma~\ref{lem:bridge-flow-expansion}, we use the larger ball \(\cB_\ast^+\) in order to justify the short-time flow expansion uniformly for initial conditions \(z\in\cB_\ast\). Indeed, even if \(z\in\cB_\ast\), the trajectory \(S_\tau(z)\), \(0\le \tau \le\Delta t\), need not remain inside \(\cB_\ast\). Thus the Taylor expansion requires regularity of \(F\) on a slightly larger region containing these short-time trajectories. By contrast, the residual quantities \(\delta_M\) and \(\ell_M\) are evaluated only at points \(z\in\cB_\ast\). This is why
Proposition~\ref{prop:bridge-route2} states the resulting residual bounds on \(\cB_\ast\), while the auxiliary expansion lemma uses \(\cB_\ast^+\).
\end{remark}

\section{Sample complexity: learning surrogate models for nudging}
\label{sec:complexity}
In the previous section, we showed that verifying the surrogate tracking result, Theorem~\ref{thm:sur-track}, reduces to controlling learning errors on the post-absorption ball \(\cB_\ast\). In this section, we study the amount of training data needed to achieve such control. We consider both vector-field and surrogate-map learning.
Although these two learning problems are different, we analyze them in a unified framework. We provide explicit sample-size conditions ensuring that the learned
model satisfies the residual bounds in Proposition~\ref{prop:bridge-route1} or
Proposition~\ref{prop:bridge-route2}. Theorem~\ref{thm:sur-track} then implies
accurate tracking for the corresponding surrogate-based nudged dynamics.


We focus here on noise-free observations, as in Theorem~\ref{thm:sur-track},
to avoid imposing additional global structural assumptions. The results of Section~\ref{sec:bridge} reduce nudging accuracy to local approximation quantities on \(\cB_\ast\), such as \(\delta_M\) and \(\ell_M\),
which are exactly what the sample-complexity analysis controls. The noisy surrogate theory in Subsection~\ref{subsec:sur-noisy}, by contrast, requires
global assumptions on the cutoff-extended surrogate drift, including global dissipativity and squeezing, which do not follow directly from local learning bounds.


\subsection{Sample complexity for vector-field learning via dictionary learning}
\label{subsec:complexity-route1}

We begin with a concrete realization of the vector-field learning approach, in which the drift is learned directly from noisy pointwise evaluations by regression onto a finite dictionary. This setting is natural when the drift is expected to admit, at least locally on the post-absorption ball \(\cB_\ast\), a low-complexity representation in terms of prescribed features. The purpose of the present subsection is to convert statistical recovery of the dictionary coefficients into quantitative control of the direct learning errors \(\varepsilon_M^F\) and \(\eta_M^F\), and hence into the AOT-validity conditions appearing in Proposition~\ref{prop:bridge-route1}.

Following the cutoff construction in Subsection~\ref{subsec:sur-cutoff}, let \(\widehat F_M\) denote a learned local drift defined on a neighborhood of the cutoff region. Although the cutoff construction requires this local field to be regular on the larger region, the statistical estimates below concern only its restriction to \(\cB_\ast\), since the residual quantities entering Theorem~\ref{thm:sur-track} are evaluated only on this post-absorption ball. We measure the direct learning error by
\begin{equation}\label{eq:complexity-route1-errors}
\varepsilon_M^F
:=
\sup_{z\in\cB_\ast}\normH{\widehat F_M(z)-F(z)},
\qquad
\eta_M^F
:=
\sup_{z\in\cB_\ast}
\bigl\|D\widehat F_M(z)-DF(z)\bigr\|_{\mathcal L(H,H)}.
\end{equation}
By Proposition~\ref{prop:bridge-route1}, these are precisely the quantities that need to be controlled in order to obtain an AOT-valid surrogate through vector-field learning. All proofs in this subsection are deferred to Appendix~\ref{app:route1}.

\subsubsection{Training data and noisy drift labels}
\label{ssubsec:complexity-route1-data}

We assume access to a training dataset of noisy pointwise observations of the drift. More precisely, let
\begin{equation}\label{eq:dict-dataset}
\mathcal D_M
:=
\Big\{\bigl(z^{(m)},Y^{(m)}\bigr)\Big\}_{m=1}^M,
\qquad
Y^{(m)}:=F\bigl(z^{(m)}\bigr)+\xi^{(m)},
\end{equation}
where the input states \(\{z^{(m)}\}_{m=1}^M\) are drawn i.i.d.\ from a probability measure \(\varrho_F\) supported on \(\cB_\ast\), and \(\xi^{(m)}\in\R^d\) models observational or numerical noise in the drift label.

\begin{assumption}[Sub-Gaussian label noise]\label{ass:dict-subg}
The label-noise vectors \(\{\xi^{(m)}\}_{m=1}^M\) are independent and mean-zero. Moreover, there exists \(\sigma>0\) such that, for every \(m=1,\ldots,M\), every \(a\in\R^d\), and every \(t\in\R\),
\[
\E \left[
\exp\bigl(t\,a^\top \xi^{(m)}\bigr)
\right]
\le
\exp\left(\frac{\sigma^2 t^2\|a\|_2^2}{2}\right).
\]
\end{assumption}

We next specify the finite-dimensional hypothesis class used to regress these noisy drift labels.

\subsubsection{Dictionary hypothesis class}
\label{ssubsec:complexity-route1-dictionary}

Fix an integer \(p\ge1\), and let \(\{\varphi_k\}_{k=1}^p\subset C^2(\mathcal U_{\mathrm{ext}})\), where \(\mathcal U_{\mathrm{ext}}\subset\R^d\) is an open neighborhood of
\(\{z\in\R^d:\normV{z}\le R_{\mathrm{ext}}^+\}.\)
Define the feature map
\[
\phi(z):=
\bigl(\varphi_1(z),\dots,\varphi_p(z)\bigr)^\top\in\R^p,
\qquad z\in\cB_\ast.
\]
For each parameter matrix \(\Theta\in\R^{p\times d}\), define the associated candidate drift
\begin{equation}\label{eq:dict-candidate}
F_\Theta(z):=\Theta^\top \phi(z),
\qquad z\in\cB_\ast.
\end{equation}

\begin{assumption}[Local realizability on \(\cB_\ast\)]
\label{ass:dict-realizable}
There exists \(\Theta^\star\in\R^{p\times d}\) such that
\[
F(z)=(\Theta^\star)^\top \phi(z),
\qquad z\in\cB_\ast.
\]
\end{assumption}

\begin{remark}[Misspecified case]
If Assumption~\ref{ass:dict-realizable} fails, one may replace \(\Theta^\star\) by a best-in-class parameter, for instance an \(L^2(\varrho_F)\) projection of \(F\) onto the dictionary class. The resulting bounds then acquire an additional approximation term. Since the goal of this subsection is to illustrate how vector-field learning can be incorporated into the AOT framework, we state the main result in the realizable setting.
\end{remark}

\begin{lemma}[Uniform feature bounds on \(\cB_\ast\)]
\label{lem:dict-bounds}
There exist finite constants \(B_0,B_1>0\) such that
\[
\sup_{z\in\cB_\ast}\|\phi(z)\|_2\le B_0,
\qquad
\sup_{z\in\cB_\ast}\|D\phi(z)\|_{\mathcal L(H,\R^p)}\le B_1.
\]
\end{lemma}

\begin{lemma}[Value and Jacobian errors reduce to parameter error]
\label{lem:dict-vJ-vs-theta}
Under Assumption~\ref{ass:dict-realizable}, for every
\(\Theta\in\R^{p\times d}\),
\[
\sup_{z\in\cB_\ast}\normH{F_\Theta(z)-F(z)}
\le
B_0\,\|\Theta-\Theta^\star\|_{\mathrm{op}},
\]
and
\[
\sup_{z\in\cB_\ast}\|DF_\Theta(z)-DF(z)\|_{\mathcal L(H,H)}
\le
B_1\,\|\Theta-\Theta^\star\|_{\mathrm{op}}.
\]
\end{lemma}

\begin{remark}[Polynomial dictionaries]
\label{rem:dict-poly}
A canonical example is the polynomial dictionary of total degree at most \(r\), in which case \(p=\binom{d+r}{r}\). Then \(\{\varphi_k\}_{k=1}^p\subset C^\infty(\R^d)\), and the bounds in Lemma~\ref{lem:dict-bounds} are automatic on \(\cB_\ast\).
\end{remark}

It remains to estimate the dictionary coefficients from the noisy labels.

\subsubsection{OLS estimator and parameter recovery}
\label{ssubsec:complexity-route1-ols}

We fit the drift by multivariate ordinary least squares:
\begin{equation}\label{eq:dict-ols}
\hat\Theta_M
\in
\arg\min_{\Theta\in\R^{p\times d}}
\frac1M\sum_{m=1}^M
\Bigl\|
Y^{(m)}-\Theta^\top\phi\bigl(z^{(m)}\bigr)
\Bigr\|_2^2.
\end{equation}

The corresponding population feature covariance matrix is
\begin{equation}\label{eq:dict-sigma}
\Sigma
:=
\E\bigl[\phi(Z)\phi(Z)^\top\bigr]\in\R^{p\times p},
\qquad
Z\sim\varrho_F.
\end{equation}

\begin{assumption}[Nondegenerate feature covariance]
\label{ass:dict-sigma}
There exists \(\lambda_{\min}>0\) such that
\(\Sigma\succeq \lambda_{\min} I_p.\)
\end{assumption}

\begin{proposition}[OLS parameter error \cite{HsuKakadeZhang2014}]
\label{prop:dict-param}
Suppose Assumptions~\ref{ass:dict-subg}, \ref{ass:dict-realizable}, and \ref{ass:dict-sigma} are satisfied. Then there exist absolute constants \(C_{\mathrm{cov}},C_{\mathrm{ols}}>0\) such that the following holds. If
\begin{equation}\label{eq:dict-M-lower}
M
\ge
C_{\mathrm{cov}}\,
\frac{B_0^2}{\lambda_{\min}}
\Bigl(p+\log(3d/\delta)\Bigr),
\end{equation}
then, with probability at least \(1-\delta\),
\begin{equation}\label{eq:dict-param-bound}
\|\hat\Theta_M-\Theta^\star\|_{\mathrm{op}}
\le
\frac{C_{\mathrm{ols}}\sqrt d}{\sqrt{M\,\lambda_{\min}}}
\,
\sigma\sqrt{p+\log(3d/\delta)}.
\end{equation}
\end{proposition}

\subsubsection{Vector-field learning sample complexity via dictionary learning}
\label{ssubsec:complexity-route1-main}

We now turn the parameter estimate from Proposition~\ref{prop:dict-param} into a sample-complexity statement for \(\varepsilon_M^F\) and \(\eta_M^F\). Once these two errors are below prescribed tolerances, Proposition~\ref{prop:bridge-route1} transfers them to the residual bounds required by the deterministic surrogate tracking theorem.

\begin{theorem}[Sample complexity of vector-field learning]
\label{thm:dict-complexity}
Let Assumptions~\ref{ass:dict-subg}, \ref{ass:dict-realizable}, and \ref{ass:dict-sigma} be in force. Fix target tolerances \(\bar\varepsilon,\bar\eta>0\) and confidence level \(\delta\in(0,1)\), and let \(\widehat F_M:=F_{\hat\Theta_M}\), with \(F_{\Theta}\) defined in \eqref{eq:dict-candidate}. With \(B_0,B_1\) as in Lemma~\ref{lem:dict-bounds}, \(\lambda_{\min}\) as in Assumption~\ref{ass:dict-sigma}, and \(C_{\mathrm{ols}}\) denoting the constant from Proposition~\ref{prop:dict-param}, define
\[
C_{\mathrm{dict}}
:=
\frac{C_{\mathrm{ols}}\sqrt d\,\max\{B_0,B_1\}}{\sqrt{\lambda_{\min}}}.
\]
Let \(C_{\mathrm{cov}}\) denote the constant from Proposition~\ref{prop:dict-param}. If
\begin{equation}\label{eq:dict-sample-complexity}
M
\ge
\max\left\{
C_{\mathrm{cov}}\,
\frac{B_0^2}{\lambda_{\min}}
\bigl(p+\log(3d/\delta)\bigr),
\quad
\frac{C_{\mathrm{dict}}^2\,\sigma^2\bigl(p+\log(3d/\delta)\bigr)}
{\min\{\bar\varepsilon,\bar\eta\}^2}
\right\},
\end{equation}
then, with probability at least \(1-\delta\),
\[
\varepsilon_M^F\le \bar\varepsilon,
\qquad
\eta_M^F\le \bar\eta.
\]
Moreover, up to multiplicative constants depending only on \(B_0,B_1,\lambda_{\min},C_{\mathrm{cov}},C_{\mathrm{ols}}\), the sample size requirement is of the order
\[
M
\gtrsim
\max\left\{
p+\log(d/\delta),
\quad
\frac{d\,\sigma^2\bigl(p+\log(d/\delta)\bigr)}
{\min\{\bar\varepsilon,\bar\eta\}^2}
\right\}.
\]
\end{theorem}

\begin{corollary}[AOT-valid surrogates from vector-field learning]
\label{cor:dict-aot-valid}
Assume the structural and cutoff hypotheses of Theorem~\ref{thm:sur-track} are in force, and suppose that the statistical hypotheses of Theorem~\ref{thm:dict-complexity} hold. Fix target direct-learning tolerances \(\bar\varepsilon,\bar\eta>0\), and suppose that \(M\) satisfies \eqref{eq:dict-sample-complexity}. Assume in addition that
\[
\mu>2\bigl(C_{\mathrm{sq}}+\bar\eta\bigr),
\qquad
\mu c_0^2 h^2<\nu.
\]
Define
\[
\gamma_{\mathrm F}
:=
\mu-2\bigl(C_{\mathrm{sq}}+\bar\eta\bigr)>0,
\qquad
\nu_{\mathrm{eff}}
:=
\nu-\frac{\mu c_0^2 h^2}{2}>0.
\]
Then, with probability at least \(1-\delta\), the following hold.

\begin{enumerate}
\item[\emph{(i)}] The cutoff-extended surrogate drift obtained from the learned vector field satisfies the residual bounds
\[
\delta_M\le \bar\varepsilon,
\qquad
\ell_M\le \bar\eta.
\]

\item[\emph{(ii)}] The corresponding surrogate nudged dynamics satisfies, for every \(t\ge T_\ast\),
\begin{equation}\label{eq:dict-aot-valid-bound}
\|v(t)-u(t)\|_H^2
\le
e^{-\gamma_{\mathrm F}(t-T_\ast)}
\|v(T_\ast)-u(T_\ast)\|_H^2
+
\frac{\bar\varepsilon^2}{\lambda_1\,\nu_{\mathrm{eff}}\,\gamma_{\mathrm F}}
\Bigl(1-e^{-\gamma_{\mathrm F}(t-T_\ast)}\Bigr).
\end{equation}
Consequently,
\begin{equation}\label{eq:dict-aot-valid-floor}
\limsup_{t\to\infty}\|v(t)-u(t)\|_H^2
\le
\frac{\bar\varepsilon^2}
{\lambda_1\,\nu_{\mathrm{eff}}\,\gamma_{\mathrm F}}.
\end{equation}
\end{enumerate}
\end{corollary}

\begin{remark}[From sample complexity to tracking guarantees in vector-field learning]
\label{rem:dict-closes-loop}
Combining Theorem~\ref{thm:dict-complexity} with Corollary~\ref{cor:dict-aot-valid}, we complete the chain
\[
\text{sample size}
\ \Longrightarrow\
\text{parameter recovery}
\ \Longrightarrow\
(\varepsilon_M^F,\eta_M^F)
\ \Longrightarrow\
(\delta_M,\ell_M)
\ \Longrightarrow\
\text{tracking guarantee}.
\]
The theorem converts statistical information from the regression problem into uniform control of the learned drift and its Jacobian on \(\cB_\ast\). The corollary then transfers these direct learning errors to the residual bounds required by Proposition~\ref{prop:bridge-route1}.
\end{remark}

\subsection{Sample complexity for solution-map learning via deep super ReLU networks}
\label{subsec:complexity-route2}

We now turn to solution-map learning, in which one learns the lag-\(\Delta t\)
solution map \(S_{\Delta t}\) and then defines a surrogate drift through the
first-order difference quotient introduced in
Subsection~\ref{subsec:bridge-route2}. In contrast with direct vector-field
learning, the object to be learned is the map \(S_{\Delta t}\), rather than
the drift \(F\) itself.

Let \(\widehat S_{\Delta t}^{(M)}\) be a learned solution map defined on a
neighborhood of the cutoff region, with the regularity needed to form the
induced local surrogate drift and to apply Proposition~\ref{prop:bridge-route2}.
The statistical estimates below, however, concern only its restriction to
\(\cB_\ast\). This is sufficient for the tracking analysis, since the
solution-map learning errors entering Proposition~\ref{prop:bridge-route2} are
evaluated only on this post-absorption ball. The goal of this subsection is to
derive finite-sample conditions under which
\begin{equation}\label{eq:r2-operator-errors-recall}
\varepsilon_M^S
:=
\sup_{z\in\cB_\ast}
\normH{\widehat S_{\Delta t}^{(M)}(z)-S_{\Delta t}(z)},
\qquad
\eta_M^S
:=
\sup_{z\in\cB_\ast}
\bigl\|D\widehat S_{\Delta t}^{(M)}(z)-DS_{\Delta t}(z)\bigr\|_{\mathcal L(H,H)}
\end{equation}
are small enough to invoke Proposition~\ref{prop:bridge-route2}.

To obtain uniform control of both the learned map and its derivative, we work
with Sobolev training on \(\cB_\ast\) using deep super ReLU networks (DSRNs),
following the Sobolev approximation framework developed in
\cite{yang2023dsrn}. As in Subsection~\ref{subsec:complexity-route1}, all
assumptions and estimates in this subsection are local on \(\cB_\ast\). All
proofs in this subsection are deferred to Appendix~\ref{app:dsrn}.

\subsubsection{DSRN hypothesis class and Sobolev loss}
\label{ssubsec:complexity-route2-loss}

To obtain uniform control of both the learned map and its first derivative
from a Sobolev-type training error, we fix
\begin{equation}\label{eq:r2-s-def}
s:=\Big\lfloor\frac d2\Big\rfloor+2.
\end{equation}
Then \(s>1+d/2\), and the Sobolev embedding theorem gives
\(H^s(\cB_\ast;H)\hookrightarrow W^{1,\infty}(\cB_\ast;H).\)
Thus an \(H^s\)-error bound controls both the sup-norm error of the learned
solution map and the sup-norm error of its derivative on \(\cB_\ast\).
Here \(H^s(\cB_\ast;H)\) is understood as the Sobolev space on the interior of
\(\cB_\ast\), with continuous representatives evaluated on the closure.

\paragraph{DSRN hypothesis class}
For architecture bounds \(W,D\in\N\), let
\(\mathcal H_{W,D}^{\mathrm{DSRN}}\) denote the class of functions
\(\mathsf h:\cB_\ast\to H\) realized by deep super ReLU networks of width at
most \(W\) and depth at most \(D\). A DSRN is essentially a ReLU network,
with a small terminal block using super ReLU activations. This terminal block
provides the smoothness needed for Sobolev approximation, while most of the
architecture remains in standard ReLU form. In the analysis below, we use
DSRNs through their Sobolev approximation properties and the corresponding
width-depth scaling.

Since the empirical loss below uses pointwise evaluations of derivatives, we
assume throughout this subsection that
\(S_{\Delta t}\in W^{m,\infty}(\cB_\ast;H)\) for some \(m>s\). Then
\(D^\alpha S_{\Delta t}\) has a bounded Borel representative on \(\cB_\ast\)
for every \(|\alpha|\le s\), so the derivative labels below are well defined.
The analysis below also uses an envelope condition on the hypothesis class, a
comparability condition on the sampling distribution, and an approximate
empirical risk minimization condition; see
Assumptions~\ref{ass:r2-dsrn-envelope}, \ref{ass:r2-mu-density}, and
\ref{ass:r2-aerm}, respectively.

\begin{assumption}[Uniform Sobolev envelope for the DSRN class]
\label{ass:r2-dsrn-envelope}
For every pair of architecture parameters \(W,D\) considered below, the
hypothesis class \(\mathcal H_{W,D}^{\mathrm{DSRN}}\) is restricted so that
each \(\mathsf h\in\mathcal H_{W,D}^{\mathrm{DSRN}}\) has weak derivatives
\(D^\alpha\mathsf h\), \(|\alpha|\le s\), admitting Borel representatives on
\(\cB_\ast\). Moreover, there exists a constant \(C_{\mathrm{env}}\ge1\),
independent of \(W\) and \(D\) along the regime considered below, such that
\begin{equation}\label{eq:r2-dsrn-envelope}
\sup_{z\in\cB_\ast}
\sum_{|\alpha|\le s}\|D^\alpha \mathsf h(z)\|_2^2
\le
C_{\mathrm{env}}^2
\qquad
\forall\,\mathsf h\in\mathcal H_{W,D}^{\mathrm{DSRN}}.
\end{equation}
\end{assumption}

\begin{remark}[Role of the envelope assumption]
\label{rem:r2-envelope}
Assumption~\ref{ass:r2-dsrn-envelope} is a theoretical boundedness condition
on the Sobolev loss class. It restricts the admissible networks to have
uniformly bounded values and derivatives up to order \(s\) on \(\cB_\ast\).
Such envelope restrictions are common in empirical-process analyses of
neural-network estimators, where bounded outputs, bounded derivatives, bounded
weights, or related restrictions are used to obtain covering and concentration
bounds \cite{ShenJiaoLinHorowitzHuang2024,OuBolcskei2026}.
\end{remark}

We next introduce notation for collecting the value and all derivatives that
enter the Sobolev loss. Let
\[
A_s:=\#\{\alpha\in\N_0^d:\ |\alpha|\le s\}
=\binom{d+s}{s},
\qquad
r_s:=dA_s.
\]
For \(\mathsf h\in\mathcal H_{W,D}^{\mathrm{DSRN}}\), define the order-\(s\)
derivative feature map
\[
\mathcal Z_s\mathsf h(z)
:=
\bigl(D^\alpha \mathsf h_\ell(z)\bigr)_{\substack{|\alpha|\le s\\1\le \ell\le d}}
\in\R^{r_s},
\]
where the case \(\alpha=0\) corresponds to the value of \(\mathsf h\) itself.
We use the same notation \(\mathcal Z_s S_{\Delta t}(z)\) for the corresponding
collection of derivatives of the true solution map.

We assume access to a Sobolev training dataset
\begin{equation}\label{eq:r2-dataset}
\mathcal D_M^{S}
:=
\bigl\{(Z_j,Y_j^{S})\bigr\}_{j=1}^M,
\qquad
Z_j\overset{\mathrm{i.i.d.}}{\sim}\varrho,
\qquad
Y_j^{S}:=\mathcal Z_s S_{\Delta t}(Z_j)\in\R^{r_s},
\end{equation}
where \(\varrho\) is a probability distribution supported on \(\cB_\ast\).

For \(\mathsf h\in\mathcal H_{W,D}^{\mathrm{DSRN}}\), the empirical Sobolev
loss associated with \(\mathcal D_M^{S}\) is
\begin{equation}\label{eq:r2-loss-emp}
\widehat{\mathcal L}_{s,M}(\mathsf h)
:=
\frac1M\sum_{j=1}^M
\|\mathcal Z_s\mathsf h(Z_j)-Y_j^{S}\|_2^2.
\end{equation}
We also use the corresponding population Sobolev losses with respect to
Lebesgue measure and the sampling measure \(\varrho\):
\begin{equation}\label{eq:r2-pop-losses}
\mathcal L_s^{\mathrm{Leb}}(\mathsf h)
:=
\sum_{|\alpha|\le s}
\|D^\alpha(\mathsf h-S_{\Delta t})\|_{L^2(\cB_\ast;H)}^2,
\qquad
\mathcal L_s^\varrho(\mathsf h)
:=
\sum_{|\alpha|\le s}
\|D^\alpha(\mathsf h-S_{\Delta t})\|_{L^2(\varrho;H)}^2.
\end{equation}

Such Sobolev-type objectives are standard in derivative-informed training, including Sobolev training, Sobolev PINNs, and gradient-enhanced PINNs \cite{Czarnecki2017,SobolevPINN2021,Yu2022gPINN}. We use this objective as an analytical device because it gives a direct route from finite-sample learning to uniform control of value and derivative errors on \(\cB_\ast\). In the numerical experiments (see Section~\ref{sec:numerics}), however, we use a lighter  first-order training objective, since high-order derivative losses are expensive and often unnecessary for accurate downstream assimilation.

\begin{assumption}[Sampling distribution comparable to Lebesgue measure on \(\cB_\ast\)]
\label{ass:r2-mu-density}
The sampling distribution \(\varrho\) is absolutely continuous with respect to
Lebesgue measure on \(\cB_\ast\), with density
\(\omega:=\frac{d\varrho}{dz}\) satisfying
\[
0<\omega_{\min}\le \omega(z)\le \omega_{\max}<\infty
\qquad\text{for a.e. }z\in\cB_\ast.
\]
\end{assumption}

\begin{assumption}[Approximate empirical risk minimization]
\label{ass:r2-aerm}
Given a dataset \(\mathcal D_M^{S}\), the training algorithm returns
\(\widehat S_{\Delta t}^{(M)}\in \mathcal H_{W,D}^{\mathrm{DSRN}}\)
such that
\[
\widehat{\mathcal L}_{s,M}\bigl(\widehat S_{\Delta t}^{(M)}\bigr)
\le
\inf_{\mathsf h\in\mathcal H_{W,D}^{\mathrm{DSRN}}}
\widehat{\mathcal L}_{s,M}(\mathsf h)+\varepsilon_{\mathrm{opt}},
\]
where \(\varepsilon_{\mathrm{opt}}\ge0\) denotes the optimization error.
\end{assumption}

\subsubsection{Approximation and generalization inputs}
\label{ssubsec:complexity-route2-inputs}

We first state the Sobolev approximation guarantees provided by the DSRN theory.

\begin{proposition}[DSRN Sobolev approximation on \(\cB_\ast\)]
\label{prop:r2-dsrn-approx}
Let \(m\in\N\) satisfy \(m>s\), and suppose
\(S_{\Delta t}\in W^{m,\infty}(\cB_\ast;H).\)
Then there exist constants
\(\varepsilon_0,\ C_W,\ C_D>0,\)
depending only on \(d\), \(m\), \(s\), the geometry of \(\cB_\ast\), and
\(\|S_{\Delta t}\|_{W^{m,\infty}(\cB_\ast;H)}\), such that for every
\(0<\varepsilon_{\rm app}\le \varepsilon_0\), there exist integers
\(W_{\rm app},D_{\rm app}\in\N\) and a DSRN
\(\mathsf h^\star_{\varepsilon_{\rm app}}
\in
\mathcal H^{\mathrm{DSRN}}_{W_{\rm app},D_{\rm app}}\)
satisfying
\begin{equation}\label{eq:r2-approx}
\|\mathsf h^\star_{\varepsilon_{\rm app}}-S_{\Delta t}\|_{W^{s,2}(\cB_\ast;H)}
\le
\varepsilon_{\rm app},
\end{equation}
with
\begin{equation}\label{eq:r2-WD-from-epsapp}
W_{\rm app}
\le
C_W\,\varepsilon_{\rm app}^{-\frac{d}{4(m-s)}}
\log \bigl(2+\varepsilon_{\rm app}^{-1}\bigr),
\qquad
D_{\rm app}
\le
C_D\,\varepsilon_{\rm app}^{-\frac{d}{4(m-s)}}
\log \bigl(2+\varepsilon_{\rm app}^{-1}\bigr).
\end{equation}
\end{proposition}

This approximation result controls the best Sobolev accuracy achievable by the
DSRN class. We next compare the Sobolev losses induced by the sampling
distribution \(\varrho\) and by Lebesgue measure.

\begin{lemma}[Comparison of \(\varrho\)- and Lebesgue-based Sobolev losses]
\label{lem:r2-loss-compare}
Under Assumption~\ref{ass:r2-mu-density}, one has
\begin{equation}\label{eq:r2-loss-compare}
\omega_{\min}\,\mathcal L_s^{\mathrm{Leb}}(\mathsf h)
\le
\mathcal L_s^\varrho(\mathsf h)
\le
\omega_{\max}\,\mathcal L_s^{\mathrm{Leb}}(\mathsf h)
\qquad
\forall\,\mathsf h\in\mathcal H_{W,D}^{\mathrm{DSRN}}.
\end{equation}
\end{lemma}

For the generalization analysis, it is useful to view the Sobolev loss as an ordinary squared-loss class applied to the derivative feature map
\(\mathcal Z_s\). We therefore introduce
\begin{equation}\label{eq:r2-loss-class}
\mathcal F_{s,W,D}
:=
\Bigl\{
z\mapsto
\|\mathcal Z_s\mathsf h(z)-\mathcal Z_sS_{\Delta t}(z)\|_2^2
:\ \mathsf h\in\mathcal H_{W,D}^{\mathrm{DSRN}}
\Bigr\}.
\end{equation}

 To control the complexity of this loss class, we use the notion of pseudo-dimension \cite{pollard1990empirical}. For a class \(\mathcal A\) of real-valued functions on a domain \(\mathcal X\), its pseudo-dimension \(\Pdim(\mathcal A)\) represents the largest integer \(m\) such that there exist points \(z_1,\ldots,z_m\in\mathcal X\) and thresholds \(r_1,\ldots,r_m\in\mathbb R\) with the following property: every binary labeling of these points can be realized by thresholding a suitable function from \(\mathcal A\). Equivalently, for every \((b_1,\ldots,b_m)\in\{0,1\}^m\), there exists \(f\in\mathcal A\) such that
\[
f(z_i)>r_i \ \text{if } b_i=1,
\qquad
f(z_i)\le r_i \ \text{if } b_i=0,
\qquad i=1,\ldots,m.
\] 

\begin{proposition}[Derivative-space complexity input for DSRNs]
\label{prop:r2-deriv-complexity}
For every fixed order \(s\ge2\), there exists a constant \(C_{\mathrm{pdim}}=C_{\mathrm{pdim}}(d,s)>0\) such that, for every multi-index \(\alpha\) with \(|\alpha|\le s\) and every output coordinate \(\ell\in\{1,\dots,d\}\), the scalar derivative class
\[
\mathcal G_{\alpha,\ell;W,D}
:=
\bigl\{
z\mapsto D^\alpha \mathsf h_\ell(z):\
\mathsf h\in\mathcal H_{W,D}^{\mathrm{DSRN}}
\bigr\}
\]
satisfies the pseudo-dimension bound
\begin{equation}\label{eq:r2-pdim}
\Pdim(\mathcal G_{\alpha,\ell;W,D})
\le
C_{\mathrm{pdim}}\,
W^2D^2\log(2+W)\,\log(2+D).
\end{equation}
\end{proposition}

The preceding complexity input is stated for scalar derivative classes. We now lift these scalar bounds to the loss-composed class \(\mathcal F_{s,W,D}\). The first step is to record uniform envelopes for the derivative feature map and the induced squared loss.

\begin{lemma}[Feature and loss envelopes]
\label{lem:r2-envelope}
Suppose Assumption~\ref{ass:r2-dsrn-envelope} holds, and suppose \(S_{\Delta t}\in W^{m,\infty}(\cB_\ast;H)\) with \(m>s\). Then there exists a constant \(C_{\mathrm{feat}}\ge1\), depending only on \(d,m,s,\cB_\ast\), \(C_{\mathrm{env}}\), and \(\|S_{\Delta t}\|_{W^{m,\infty}(\cB_\ast;H)}\), such that
\begin{equation}\label{eq:r2-feature-envelope}
\sup_{z\in\cB_\ast}\|\mathcal Z_s\mathsf h(z)\|_2\le C_{\mathrm{feat}},
\qquad
\sup_{z\in\cB_\ast}\|\mathcal Z_sS_{\Delta t}(z)\|_2\le C_{\mathrm{feat}}
\qquad
\forall\,\mathsf h\in\mathcal H_{W,D}^{\mathrm{DSRN}}.
\end{equation}
Consequently, every \(g\in\mathcal F_{s,W,D}\) satisfies
\begin{equation}\label{eq:r2-loss-envelope}
0\le g(z)\le 4C_{\mathrm{feat}}^2
\qquad
\forall\,z\in\cB_\ast.
\end{equation}
\end{lemma}

We next combine the pseudo-dimension estimate from Proposition~\ref{prop:r2-deriv-complexity} with the envelope bound above to obtain a covering-number estimate for the loss-composed class \(\mathcal F_{s,W,D}\). Following \cite{AnthonyBartlett1999}, for a general class \(\mathcal A\) of real-valued functions on \(\cB_\ast\) and a sample \(Z_M=(z^{(1)},\ldots,z^{(M)})\in\cB_\ast^M\), we write
\[
\mathcal A|_{Z_M}
:=
\bigl\{
(f(z^{(1)}),\ldots,f(z^{(M)})):\ f\in\mathcal A
\bigr\}
\subset \mathbb R^M
\]
for the restriction of \(\mathcal A\) to the sample \(Z_M\). We denote by
\(\mathcal N_\infty(\varepsilon,\mathcal A|_{Z_M})
:=
\mathcal N(\varepsilon,\mathcal A|_{Z_M},\|\cdot\|_\infty)\)
the \(\varepsilon\)-covering number of this restricted set with respect to the \(\ell^\infty\)-metric on \(\mathbb R^M\). The corresponding uniform covering number is
\[
\mathcal N_\infty(\varepsilon,\mathcal A,M)
:=
\sup_{Z_M\in\cB_\ast^M}
\mathcal N_\infty(\varepsilon,\mathcal A|_{Z_M}).
\]

\begin{lemma}[Empirical covering number of the Sobolev loss class]
\label{lem:r2-covering}
Suppose Assumption~\ref{ass:r2-dsrn-envelope} holds, and suppose Proposition~\ref{prop:r2-deriv-complexity} is in force. Let \(C_{\mathrm{feat}}\) be as in Lemma~\ref{lem:r2-envelope}. Then there exists a constant \(C_{\mathrm{cov}}=C_{\mathrm{cov}}(d,s)>0\) such that, for every sample \(Z_M=(z^{(1)},\dots,z^{(M)})\in\cB_\ast^M\) and every \(\varepsilon\in(0,1]\),
\begin{equation}\label{eq:r2-covering}
\log \mathcal N_\infty \bigl(
\varepsilon,\mathcal F_{s,W,D}|_{Z_M}
\bigr)
\le
C_{\mathrm{cov}}\,
W^2D^2\log(2+W)\,\log(2+D)\,
\log \Bigl(
\frac{C_{\mathrm{cov}}\,C_{\mathrm{feat}}^2\,M}{\varepsilon}
\Bigr).
\end{equation}
\end{lemma}

This empirical covering estimate is the main complexity input for the uniform law of large numbers for the Sobolev loss class.

\begin{proposition}[Uniform generalization for the Sobolev loss class]
\label{prop:r2-uniform-generalization}
Assume the hypotheses of Lemma~\ref{lem:r2-covering}. Define
\[
\Gamma_{W,D,\delta}
:=
W^2D^2\log(2+W)\,\log(2+D)+\log\frac{2}{\delta}.
\]
Then there exists a constant \(\widetilde C_{\mathrm{gen}}\ge1\), depending
only on \(d\), \(s\), and \(C_{\mathrm{feat}}\), such that the following holds.
For every \(\varepsilon_{\rm gen}\in(0,1]\) and \(\delta\in(0,1)\), if
\begin{equation}\label{eq:r2-gen-sample}
M
\ge
\widetilde C_{\mathrm{gen}}
\frac{\Gamma_{W,D,\delta}}{\varepsilon_{\rm gen}^2}
\log \Bigl(
\frac{\widetilde C_{\mathrm{gen}}\Gamma_{W,D,\delta}}
{\varepsilon_{\rm gen}^3}
\Bigr),
\end{equation}
then, with probability at least \(1-\delta\),
\begin{equation}\label{eq:r2-generalization}
\sup_{\mathsf h\in\mathcal H_{W,D}^{\mathrm{DSRN}}}
\bigl|
\widehat{\mathcal L}_{s,M}(\mathsf h)-\mathcal L_s^\varrho(\mathsf h)
\bigr|
\le
\varepsilon_{\rm gen}.
\end{equation}
\end{proposition}

Combining uniform generalization with approximate empirical risk minimization
gives a population Sobolev loss bound for the learned solution map.

\begin{proposition}[Population Sobolev loss bound for the approximate ERM]
\label{prop:r2-aerm-generalization}
Suppose Assumption~\ref{ass:r2-aerm} holds, and suppose that the hypotheses of
Proposition~\ref{prop:r2-uniform-generalization} hold. Then, with probability
at least \(1-\delta\),
\begin{equation}\label{eq:r2-aerm-generalization}
\mathcal L_s^\varrho\bigl(\widehat S_{\Delta t}^{(M)}\bigr)
\le
\inf_{\mathsf h\in\mathcal H_{W,D}^{\mathrm{DSRN}}}\mathcal L_s^\varrho(\mathsf h)
+
2\varepsilon_{\rm gen}
+
\varepsilon_{\mathrm{opt}}.
\end{equation}
\end{proposition}

It remains to convert this population Sobolev control into the uniform value
and Jacobian errors required by \eqref{eq:r2-operator-errors-recall}.

\begin{lemma}[Sobolev loss implies uniform value and Jacobian control]
\label{lem:r2-loss-to-W1inf}
Suppose Assumption~\ref{ass:r2-mu-density} holds, \(s\) is chosen as in
\eqref{eq:r2-s-def}, and \(S_{\Delta t}\in W^{m,\infty}(\cB_\ast;H)\) for some \(m>s\). Then there
exists a constant \(C_{\mathrm{emb}}\ge1\), depending only on
\(d,s,\cB_\ast,\omega_{\min}\), such that for every
\(\mathsf h\in W^{s,2}(\cB_\ast;H)\),
\begin{equation}\label{eq:r2-loss-to-W1inf}
\|\mathsf h-S_{\Delta t}\|_{W^{1,\infty}(\cB_\ast;H)}
\le
C_{\mathrm{emb}}\,
\bigl(\mathcal L_s^\varrho(\mathsf h)\bigr)^{1/2}.
\end{equation}
In particular,
\begin{equation}\label{eq:r2-loss-to-operator}
\sup_{z\in\cB_\ast}\normH{\mathsf h(z)-S_{\Delta t}(z)}
\le
C_{\mathrm{emb}}\,
\bigl(\mathcal L_s^\varrho(\mathsf h)\bigr)^{1/2},
\end{equation}
and
\begin{equation}\label{eq:r2-loss-to-jac}
\sup_{z\in\cB_\ast}
\|D\mathsf h(z)-DS_{\Delta t}(z)\|_{\mathcal L(H,H)}
\le
C_{\mathrm{emb}}\,
\bigl(\mathcal L_s^\varrho(\mathsf h)\bigr)^{1/2}.
\end{equation}
\end{lemma}

\subsubsection{Solution-map learning sample complexity via DSRNs}
\label{ssubsec:complexity-route2-main}

We now combine the approximation, generalization, and Sobolev-embedding inputs
above to obtain a sample-complexity theorem for the solution-map learning
errors in \eqref{eq:r2-operator-errors-recall}. This converts statistical
accuracy of the learned flow map into the deterministic quantities required by
Proposition~\ref{prop:bridge-route2}.

\begin{theorem}[Sample complexity of solution-map learning]
\label{thm:r2-complexity}
Let \(s:=\Big\lfloor\frac d2\Big\rfloor+2\), and let \(m>s\). Suppose that \(S_{\Delta t}\in W^{m,\infty}(\cB_\ast;H)\), and suppose that Assumptions~\ref{ass:r2-dsrn-envelope}, \ref{ass:r2-mu-density}, and~\ref{ass:r2-aerm} hold. Fix target operator accuracies \(\bar\varepsilon_S,\bar\eta_S>0\) and confidence level \(\delta\in(0,1)\). Let \(C_{\rm emb}\) be as in Lemma~\ref{lem:r2-loss-to-W1inf}. Define 
\begin{equation}\label{eq:r2-eps-star}
\varepsilon_\ast
:=
\min\left\{
\frac{1}{2C_{\mathrm{emb}}\sqrt{\omega_{\max}+3}}
\min\{\bar\varepsilon_S,\bar\eta_S,1\},
\varepsilon_0
\right\},
\end{equation}
where \(\omega_{\max}\) and \(\varepsilon_0\) are as in Assumption~\ref{ass:r2-mu-density} and Proposition~\ref{prop:r2-dsrn-approx}, respectively. Suppose also that
\begin{equation}\label{eq:r2-opt-small}
\varepsilon_{\mathrm{opt}}\le \varepsilon_\ast^2.
\end{equation}
Then the following hold.

\begin{enumerate}
\item[\emph{(i)}] \emph{Architecture size.}
Let \(\varepsilon_{\rm app}:=\varepsilon_\ast\), and choose the DSRN hypothesis class \(\mathcal H^{\rm DSRN}_{W,D}\) with width \(W=W_{\rm app}\) and depth \(D=D_{\rm app}\), as provided by Proposition~\ref{prop:r2-dsrn-approx} for approximation tolerance \(\varepsilon_{\rm app}\). Then
\begin{equation}\label{eq:r2-WD-choice}
W
\le
C_W\,\varepsilon_\ast^{-\frac{d}{4(m-s)}}
\log \bigl(2+\varepsilon_\ast^{-1}\bigr),
\qquad
D
\le
C_D\,\varepsilon_\ast^{-\frac{d}{4(m-s)}}
\log \bigl(2+\varepsilon_\ast^{-1}\bigr).
\end{equation}

\item[\emph{(ii)}] \emph{Sample size and learning accuracy.}
There exists a constant \(C_{\mathrm{samp}}>0\), depending only on \(d,m,s,\cB_\ast\), the constants \(\omega_{\min},\omega_{\max}\) and \(C_{\mathrm{env}}\) from Assumptions~\ref{ass:r2-mu-density} and \ref{ass:r2-dsrn-envelope}, and \(\|S_{\Delta t}\|_{W^{m,\infty}(\cB_\ast;H)}\), such that if
\begin{equation}\label{eq:r2-sample-complexity}
M
\ge
C_{\mathrm{samp}}
\left[
\varepsilon_\ast^{-4-\frac{d}{m-s}}
\,\mathrm{polylog}\Bigl(\frac1{\varepsilon_\ast}\Bigr)
+
\varepsilon_\ast^{-4}\log\frac1\delta
\right],
\end{equation}
then, with probability at least \(1-\delta\),
\begin{equation}\label{eq:r2-final-error}
\sup_{z\in\cB_\ast}
\normH{\widehat S_{\Delta t}^{(M)}(z)-S_{\Delta t}(z)}
\le
\bar\varepsilon_S,
\qquad
\sup_{z\in\cB_\ast}
\bigl\|D\widehat S_{\Delta t}^{(M)}(z)-DS_{\Delta t}(z)\bigr\|_{\mathcal L(H,H)}
\le
\bar\eta_S.
\end{equation}
Equivalently,
\[
\varepsilon_M^S\le \bar\varepsilon_S,
\qquad
\eta_M^S\le \bar\eta_S.
\]
\end{enumerate}
\end{theorem}

\begin{corollary}[AOT-valid surrogates from solution-map learning]
\label{cor:r2-aot-valid}
Suppose that the structural and cutoff hypotheses of
Theorem~\ref{thm:sur-track} hold. Suppose also that the flow-expansion
hypotheses of Lemma~\ref{lem:bridge-flow-expansion} hold with constant
\(C_{\mathrm{flow}}\), and that the statistical hypotheses of
Theorem~\ref{thm:r2-complexity} hold. Fix target operator accuracies
\(\bar\varepsilon_S,\bar\eta_S>0\), and suppose that \(M\) satisfies
\eqref{eq:r2-sample-complexity}. Define
\[
\bar\delta_{\mathrm S}
:=
\frac{\bar\varepsilon_S}{\Delta t}+C_{\mathrm{flow}}\Delta t,
\qquad
\bar\ell_{\mathrm S}
:=
\frac{\bar\eta_S}{\Delta t}+C_{\mathrm{flow}}\Delta t.
\]
Suppose in addition that
\[
\mu>2\bigl(C_{\mathrm{sq}}+\bar\ell_{\mathrm S}\bigr),
\qquad
\mu c_0^2 h^2<\nu.
\]
Define
\[
\gamma_{\mathrm S}
:=
\mu-2\bigl(C_{\mathrm{sq}}+\bar\ell_{\mathrm S}\bigr)>0,
\qquad
\nu_{\mathrm{eff}}
:=
\nu-\frac{\mu c_0^2 h^2}{2}>0.
\]
Then, with probability at least \(1-\delta\), the following hold.

\begin{enumerate}
\item[\emph{(i)}] The cutoff-extended surrogate drift induced by the learned
solution map satisfies the residual bounds
\[
\delta_M\le \bar\delta_{\mathrm S},
\qquad
\ell_M\le \bar\ell_{\mathrm S}.
\]

\item[\emph{(ii)}] The corresponding surrogate nudged dynamics satisfies, for
every \(t\ge T_\ast\),
\begin{equation}\label{eq:r2-aot-valid-bound}
\|v(t)-u(t)\|_H^2
\le
e^{-\gamma_{\mathrm S}(t-T_\ast)}
\|v(T_\ast)-u(T_\ast)\|_H^2
+
\frac{\bar\delta_{\mathrm S}^2}
{\lambda_1\,\nu_{\mathrm{eff}}\,\gamma_{\mathrm S}}
\Bigl(1-e^{-\gamma_{\mathrm S}(t-T_\ast)}\Bigr).
\end{equation}
Consequently,
\begin{equation}\label{eq:r2-aot-valid-floor}
\limsup_{t\to\infty}\|v(t)-u(t)\|_H^2
\le
\frac{\bar\delta_{\mathrm S}^2}
{\lambda_1\,\nu_{\mathrm{eff}}\,\gamma_{\mathrm S}}.
\end{equation}
\end{enumerate}
\end{corollary}

\begin{remark}[From sample complexity to tracking guarantees in solution-map learning]
\label{rem:r2-closes-loop}
Theorem~\ref{thm:r2-complexity} and Corollary~\ref{cor:r2-aot-valid} close the chain
\[
\text{sample size}
\ \Longrightarrow\
(\varepsilon_M^S,\eta_M^S)
\ \Longrightarrow\
(\delta_M,\ell_M)
\ \Longrightarrow\
\text{tracking guarantee}.
\]
Thus the statistical accuracy of the learned solution map is converted, through Proposition~\ref{prop:bridge-route2}, into the residual bounds required by the surrogate tracking theorem.
\end{remark}

\begin{remark}[Nearly quartic sample complexity under very high regularity]
The exponent in \eqref{eq:r2-sample-complexity} is
\[
4+\frac{d}{m-s},
\qquad
s=\Big\lfloor\frac d2\Big\rfloor+2.
\]
Thus, if \(S_{\Delta t}\) is very smooth on \(\cB_\ast\), one may take \(m\)
large, making \(\frac{d}{m-s}\) arbitrarily small. In particular, if
\[
S_{\Delta t}\in W^{m,\infty}(\cB_\ast;H)
\qquad
\text{for every }m>s,
\]
then, at the level of the displayed exponent and with constants allowed to
depend on \(m\), the sample complexity in Theorem~\ref{thm:r2-complexity}
approaches the near-\(\varepsilon_\ast^{-4}\) regime, up to polylogarithmic
factors.
\end{remark}

\section{Numerical experiments}
\label{sec:numerics}

In this section, we illustrate the finite-dimensional AOT framework and the surrogate AOT theory developed in the previous sections using the Lorenz--96 system, which was verified in Subsection~\ref{subsec:l96} to satisfy our standing assumptions. The numerical study is organized around four goals.

First, we examine the performance of exact-model AOT under a conservative choice of the nudging parameter \(\mu\). This experiment is carried out in the noiseless setting and serves as a deterministic baseline for the feedback mechanism, allowing us to compare the observed synchronization behavior with the qualitative predictions of Section~\ref{sec:baseline}. Second, we test the surrogate AOT mechanism from Section~\ref{sec:surrogate} for the two learning routes introduced in Section~\ref{sec:bridge}: the direct vector-field learning route of Subsection~\ref{subsec:bridge-route1} and the solution-map learning route of Subsection~\ref{subsec:bridge-route2}. In both cases, we compare the surrogate nudged trajectory with the exact-model nudged trajectory, which serves as the exact-model AOT reference, and with the corresponding free surrogate run. We report these comparisons in both noiseless and noisy-observation settings: the noiseless experiments isolate the effect of surrogate model error, while the noisy experiments illustrate the additional tracking floor induced by observation noise. Third, we examine the sensitivity of the nudged dynamics to the feedback-resolution parameter \(h\). This study is carried out in the noiseless setting in order to isolate the effect of feedback resolution from noise-induced fluctuations. Finally, we study the sensitivity to the training sample size in both learning routes, using the sample size as a numerical proxy for surrogate accuracy. These experiments test how feedback resolution and surrogate model error affect the transient tracking behavior and the long-time tracking floor.

\subsection{Experimental setup}
\label{ssec:setup}
\paragraph{Lorenz--96 system}
We consider the Lorenz--96 system in dimension \(d=40\),
\begin{equation}
\label{eq:lorenz96}
F(u)_i=(u_{i+1}-u_{i-2})u_{i-1}-u_i+8,
\qquad i=1,\dots,d,
\end{equation}
with cyclic indexing and forcing term \(8\). This vector field defines both the true dynamics \eqref{eq:truth} and the exact-model nudged dynamics \eqref{eq:aot} and its stochastic counterpart \eqref{eq:aot-noisy}. All continuous-time systems simulated in the experiments are integrated numerically by a fourth-order Runge--Kutta scheme with internal step size \(\dtint=10^{-2}\). The construction of the training trajectories, the training state cloud, and the route-specific labels is described below.

\paragraph{Feedback operators}
We use the three feedback classes introduced in Subsection~\ref{subsec:obsmodel}. For \emph{linear sensing measurements}, 
the sensing matrix \(G\) is chosen to be a randomly generated overdetermined matrix, and the numerical parameter \(h\) is the regularization parameter in the associated Tikhonov-type reconstruction; see \eqref{eq:redundant linear}. For \emph{band-limited spectral measurements}, the observations are low-frequency Fourier coefficients, and the implementation is indexed by the retained rank \(K\). For \emph{dominant modal coefficient measurements}, the observations are the leading POD coefficients computed from the training state cloud \(\mathcal X_{\mathrm{train}}\), defined below in Equation~\eqref{eq:X_train}. For the latter two settings, we use the corresponding empirical tail-energy fraction as a numerical proxy for the feedback-resolution parameter \(h\). Since these three feedback classes display broadly similar qualitative behavior in the present regime, we focus in the remainder of this section on \emph{band-limited spectral measurements}. Additional numerical results for the other two feedback classes are available in the accompanying code repository~\cite{li2026cda_code}.

\paragraph{Reference trajectories, shared input states, and empirical absorbing ball}
To connect the numerics with the long-time dynamical regime relevant for data assimilation, we first construct a shared collection of input states from several long reference trajectories of the Lorenz--96 dynamics. These states are used as the input locations for the surrogate-learning datasets below, while the route-specific labels are described separately in the direct vector-field learning and solution-map learning experiments; see Subsections~\ref{subsec:numerics-route1} and \ref{subsec:numerics-route2}.

For each reference run, we discard an initial transient of \(1000\) time steps before recording a trajectory of length \(T_{\rm ref}=120\). Thus, we obtain 
post-transient trajectories
\[
\{u^{(m)}(t_j)\}_{j=0}^{N_{\rm traj}},
\qquad m=1,\dots,N_{\mathrm{ref}}.
\]
We use only these
post-transient portions to identify an empirical bounded region in phase space. We define
\[
R_{\rm emp}
:=
(1+\varepsilon)
\max_{1\le m\le N_{\mathrm{ref}}}
\max_{0\le j\le N_{\rm traj}}
\|u^{(m)}(t_j)\|_2,
\qquad \varepsilon=0.05,
\]
and regard \(B(0,R_{\rm emp})\) as an empirical proxy for the post-absorption region used to choose the cutoff radii in the construction of the globally defined surrogate dynamics; see Section~\ref{sec:surrogate}.

From each 
reference trajectory, we extract sample states using a fixed stride in order to reduce the strong temporal correlation between consecutive samples. Let \(\mathcal J\subset\{0,\dots,N_{\rm traj}\}\) denote the resulting set of sample indices. Taking the union over all reference trajectories gives the shared set of training input states
\begin{equation}
\label{eq:X_train}
\mathcal X_{\mathrm{train}}
=
\{u^{(m)}(t_j): j\in\mathcal J,\ m=1,\dots,N_{\mathrm{ref}}\}
\subset\R^d.
\end{equation}

We emphasize that the input states in \(\mathcal X_{\mathrm{train}}\) are sampled from multiple recorded long-time reference trajectories, rather than drawn uniformly from the empirical ball \(B(0,R_{\rm emp})\). This choice allows the training inputs to better capture the long-time dynamics that are most informative for data assimilation, by focusing on regions that are 
frequently visited by the trajectories. Using multiple reference trajectories improves coverage of the attractor-relevant region.
The ball \(B(0,R_{\rm emp})\) is used only as an empirical outer approximation of the post-absorption region for the cutoff construction.

\paragraph{Surrogates \(F_M\) via cutoff construction}
To obtain a globally defined surrogate from the learned local model, we use the same cutoff construction as in Section~\ref{sec:surrogate}. With the empirical radius \(R_{\rm emp}\) obtained from the reference trajectories, we take
\[
R_{\mathrm{ext}} = R_{\mathrm{emp}},
\qquad
R_{\mathrm{ext}}^{+} = 1.5\,R_{\mathrm{emp}},
\]
and define the cutoff function \(\chi(z)\) as in \eqref{eq:cutoff}. We then blend the learned local drift with the dissipative linear field
\[
F_{\mathrm{diss}}(z)=8\cdot\mathbf 1 - B z,
\qquad B=2I.
\]
Hence the surrogate coincides with the learned local model inside \(B(0,R_{\mathrm{ext}})\), agrees with \(F_{\mathrm{diss}}\) outside \(B(0,R_{\mathrm{ext}}^{+})\), and transitions smoothly between the two regions.

\paragraph{Evaluation metrics}

To assess tracking performance, we use the absolute tracking error
\[
e(t):=\|v(t)-u(t)\|_2
\]
as the primary pointwise metric. For the sensitivity studies with respect to the feedback-resolution parameter \(h\) and the training sample size, we further summarize the tracking performance using the final-time error \(e(T)\) and the integrated squared tracking error
\[
\int_0^T \|v(\tau)-u(\tau)\|_2^2\,d\tau .
\]
For these summary metrics, we take the terminal time \(T=5\) for the vector-field learning route (see Subsection~\ref{subsec:numerics-route1}) and use a slightly longer assimilation horizon, \(T=6\), for the solution-map learning route, where the learned surrogate is not built from prior knowledge of the vector-field structure (see Subsection~\ref{subsec:numerics-route2}). The final-time error measures the terminal tracking accuracy and uses the same tracking-error quantity that appears in the error-floor bounds of Theorems~\ref{thm:sur-track} and \ref{thm:sur-track-noisy}, while the integrated squared error captures the cumulative tracking error over the full assimilation window. The latter is especially useful in the noisy-observation setting, since averaging over time gives a more stable metric and reduces the influence of any particular noise realization near the final time. Unless otherwise stated, in all subsequent plots, solid curves represent the median over independent runs, and shaded regions indicate the interquartile range between the 25th and 75th percentiles. In the noiseless setting, independent runs correspond to different initializations, while in the noisy-observation setting, each independent run also uses a different observation-noise realization.

\subsection{Direct vector-field learning}
\label{subsec:numerics-route1}

\paragraph{Training data set}
For the direct vector-field learning approach, we learn the vector field \eqref{eq:lorenz96} directly from a training dataset composed of state--label pairs
\[
\Bigl(
u^{(m)}(t_j),\, F\bigl(u^{(m)}(t_j)\bigr)+\xi^{(m,j)}\Bigr),
\]
where \(u^{(m)}(t_j)\in \mathcal X_{\mathrm{train}}\) and \(\xi^{(m,j)}\) is Gaussian label noise with level \(\sigma_{\mathrm{label}}=0.1\). To reduce temporal correlation among the samples, we keep every tenth state along each reference trajectory, yielding approximately \(7000\) state--label pairs in total. Under this setting, we test both the noiseless case, in which exact observations are used in the nudged system \eqref{eq:sur-aot}, and the noisy case, in which noisy observations are used in the stochastic nudged system \eqref{eq:sur-aot-noisy}. In the noisy case, the observation-noise level is taken to be \(\sigma_{\mathrm{noi}}=0.1\).

\paragraph{Choice of dictionary}
We choose the dictionary to be the Lorenz--96-specific sparse polynomial dictionary \texttt{l96\_local}, which contains one constant feature, all \(40\) linear coordinates, and the \(80\) quadratic monomials needed to represent the local bilinear structure of the Lorenz--96 vector field, for a total of \(121\) shared features. The coefficient matrix is fitted by ordinary least squares, as discussed in Subsection~\ref{subsec:complexity-route1}.

\subsubsection{Exact-model AOT baseline}
For the exact-model baseline, we use the band-limited spectral measurement setting described in Subsection~\ref{ssec:setup}. Specifically, we consider the nudging-parameter and resolution grids
\[
\mu \in \{10,15,20,25\}, \qquad
h\in\{0,0.09,0.18\},
\]
where the three reported values of \(h\) denote the empirical tail-energy fractions associated with the DFT truncation orders \(K\in\{40,30,25\}\), respectively. We fix the resolution parameter at the median value in this tested \(h\)-grid and use the same choice in the subsequent surrogate experiments, so that the comparisons are made under a common feedback resolution. 
For the nudging parameter, by contrast,  throughout the current and subsequent comparison experiments we use the smallest value in the prescribed \(\mu\)-grid for the exact-model dynamics.
This provides a deliberately conservative exact-model reference and highlights that exact-model AOT already synchronizes reliably even under the weakest tested nudging strength. 

Figure~\ref{fig:numerics-exact-baseline-dft-1} shows that, even under this conservative choice of parameters, the exact-model nudged dynamics lead to rapid synchronization and maintain a small long-time tracking error, in qualitative agreement with Theorem~\ref{thm:conv}.

\begin{figure}[tbp]
\centering

\begin{subfigure}[t]{0.40\textwidth}
    \centering
    \includegraphics[width=\linewidth,trim={0 7 0 6},clip]
    {\detokenize{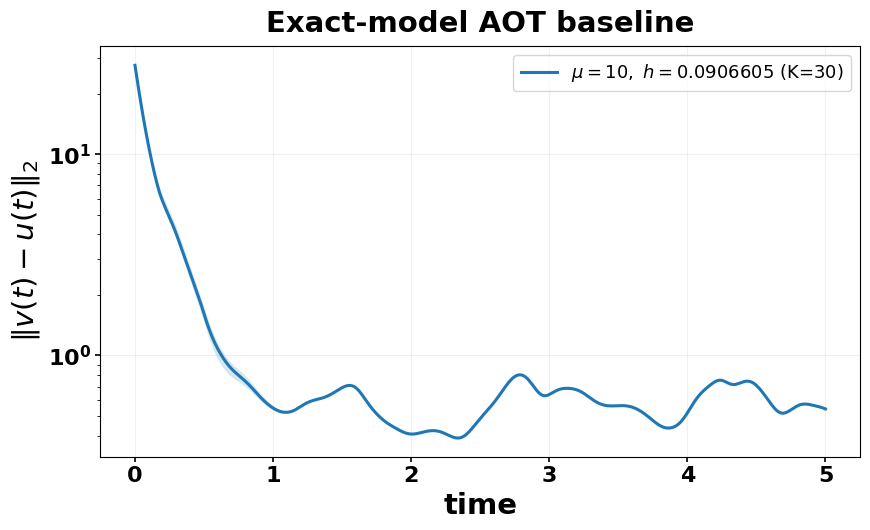}}
    \caption{Absolute tracking error.}
\end{subfigure}
\hspace{0.015\textwidth}
\begin{subfigure}[t]{0.40\textwidth}
    \centering
    \includegraphics[width=\linewidth,trim={0 7 0 6},clip]
    {\detokenize{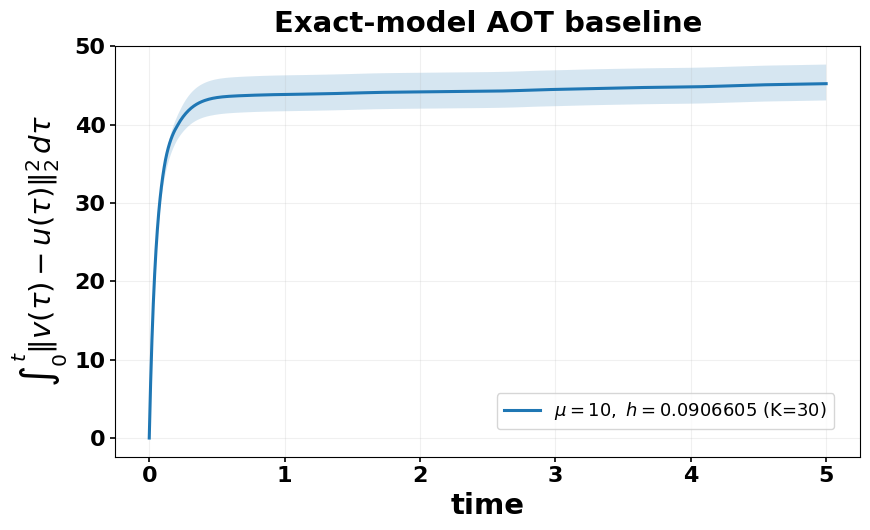}}
    \caption{Integrated squared error.}
\end{subfigure}

\caption{
Tracking performance of the exact-model AOT baseline.}
\label{fig:numerics-exact-baseline-dft-1}
\end{figure}

\subsubsection{Surrogate AOT via direct vector-field learning}

We next replace the exact drift \(F\) by the surrogate drift \(F_M\) constructed from direct vector-field learning, and study the resulting surrogate nudged dynamics under band-limited spectral measurements. We compare the following three dynamics across multiple initial conditions:
\begin{enumerate}
    \item the exact-model nudged system;
    \item the surrogate nudged system, obtained by replacing \(F\) with the cutoff-extended surrogate \(F_M\); 
    \item the surrogate free run, namely the learned dynamics without nudging.
\end{enumerate}

For this route, we use the fixed choices \(\mu^\ast=25\) in the noiseless setting and \(\mu^\ast=20\) in the noisy-observation setting throughout the following experiments, based on the tuning procedure described in Appendix~\ref{app:mu-selection-dft}. The exact-model AOT runs keep the conservative choice of \(\mu\) described above, providing a conservative exact-drift baseline.

\begin{figure}[tbp]
\centering

\begin{subfigure}[t]{0.44\textwidth}
    \centering
    \includegraphics[width=0.9\linewidth,trim={0 7 0 6},clip]
    {\detokenize{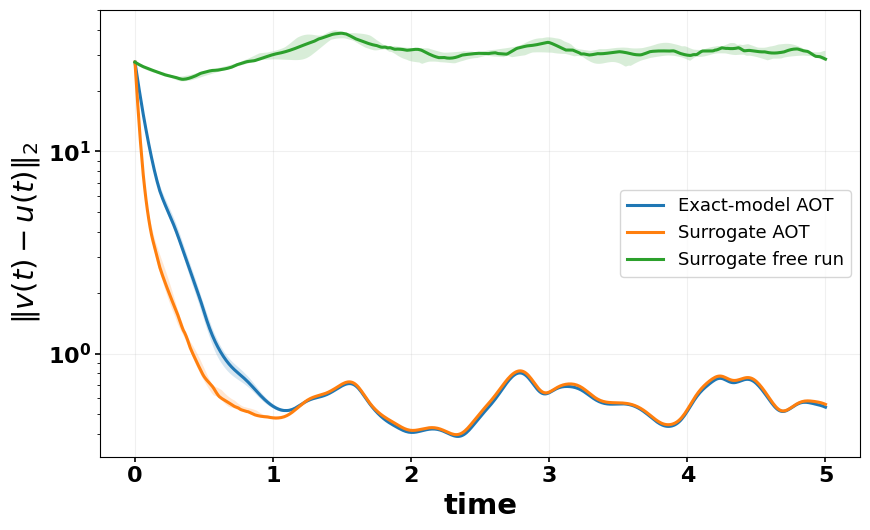}}
    \caption{Tracking error comparison.}
\end{subfigure}
\hspace{-0.015\textwidth}
\begin{subfigure}[t]{0.44\textwidth}
    \centering
    \includegraphics[width=0.9\linewidth,trim={0 7 0 6},clip]
    {\detokenize{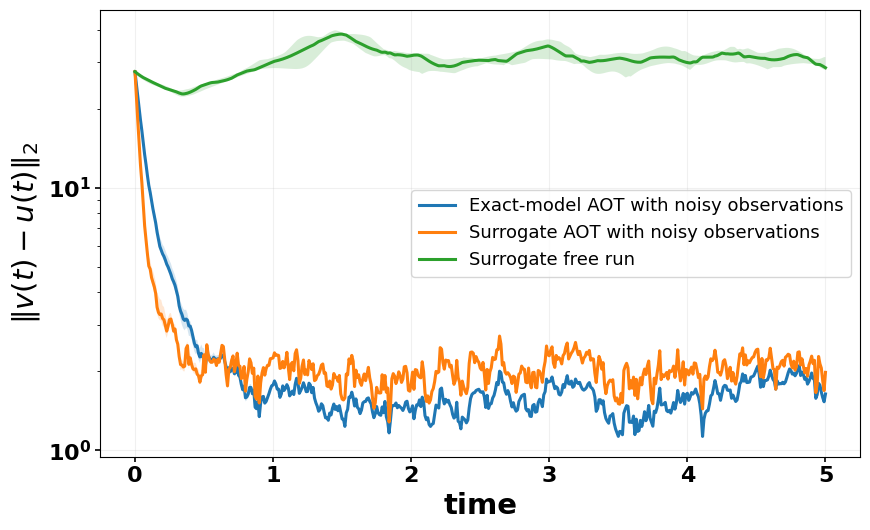}}
    \caption{Tracking error comparison (noisy observations).}
\end{subfigure}

\caption{
Tracking-error comparisons in the direct vector-field learning route.
}
\label{fig:numerics-surrogate-errors-dft-1}
\end{figure}

Figure~\ref{fig:numerics-surrogate-errors-dft-1} shows that the surrogate AOT algorithm successfully tracks the true trajectory in both the noiseless and noisy settings. In contrast, the surrogate free run, which evolves without nudging feedback, fails to track the truth. This demonstrates that the AOT feedback is essential for stabilizing the learned surrogate dynamics.

Compared with exact-model AOT, surrogate AOT exhibits a longer synchronization transient and a small residual error floor. This behavior appears in both the noiseless and noisy experiments, and is consistent with the model-error floor predicted by Theorem~\ref{thm:sur-track}. The effect is more pronounced in the noisy setting. Moreover, comparing the noisy and noiseless cases shows that both exact-model AOT and surrogate AOT have an additional error floor when the observations are noisy. This additional floor is caused by stochastic noise in the observation process, in agreement with the analysis in Theorem~\ref{thm:sur-track-noisy}. Overall, the surrogate model can still be successfully synchronized with the truth when coupled with the AOT feedback mechanism.

\paragraph{Coordinate-wise comparison}

Figure~\ref{fig:numerics-coordinates-dft-1} provides a coordinate-level view of the same behavior. The exact-model nudged trajectory synchronizes most rapidly, while the surrogate nudged trajectory stays much closer to the truth than the free surrogate trajectory. The remaining small discrepancy for the surrogate nudged trajectory is consistent with the model-error floor predicted by Theorem~\ref{thm:sur-track}.

In the noisy case, Figure~\ref{fig:numerics-coordinates-dft-noisy-1} shows the same qualitative ordering. Both exact-model AOT and surrogate AOT continue to track the truth, but with visible noise-induced fluctuations and a larger residual level, in agreement with the stochastic tracking result in Theorem~\ref{thm:sur-track-noisy}.

\begin{figure}[t]
\centering
\includegraphics[width=0.95\linewidth,trim={0 8 0 8},clip]
{\detokenize{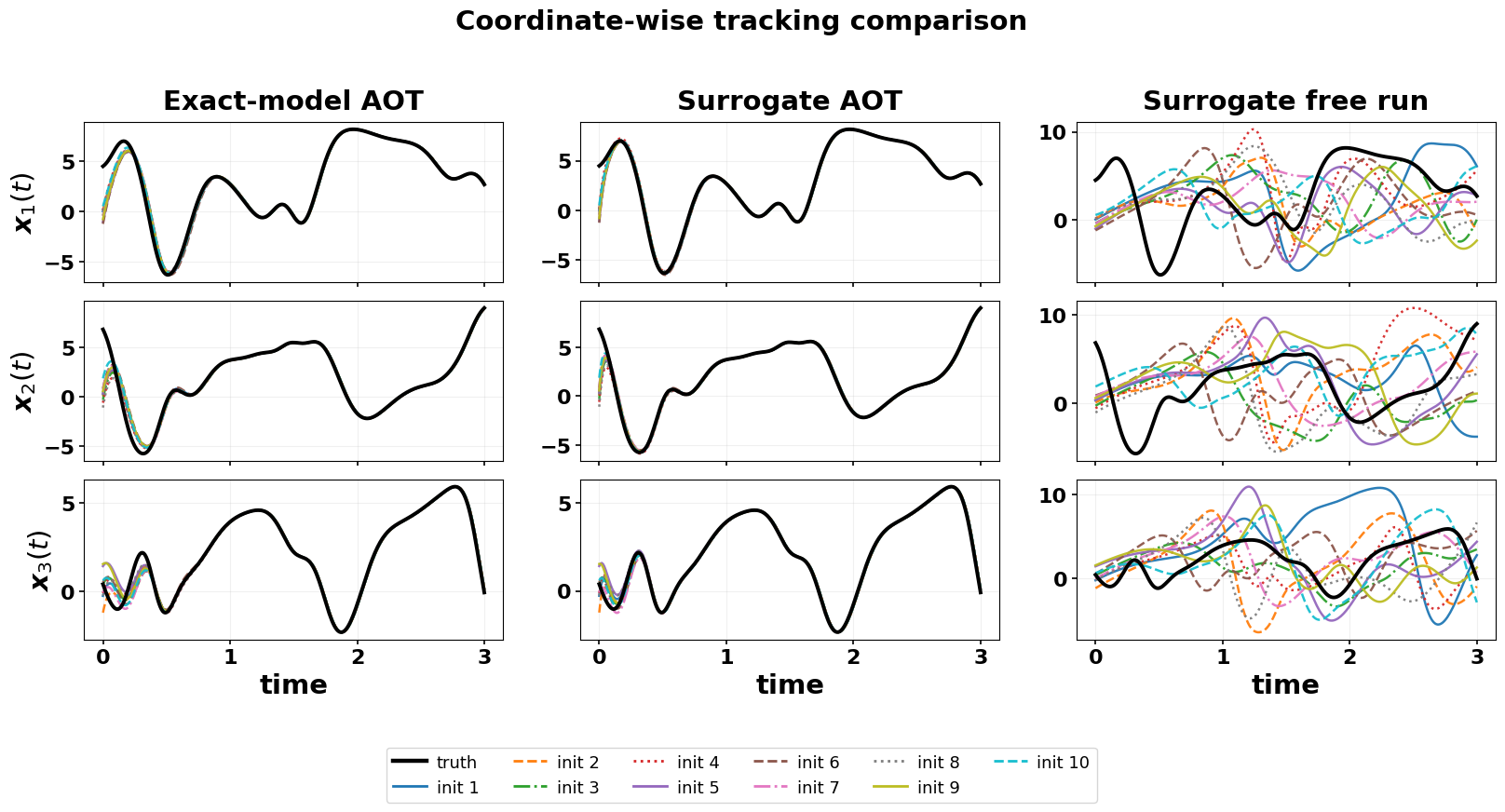}}
\caption{
Coordinate-wise tracking comparison in the direct vector-field learning route.
}
\label{fig:numerics-coordinates-dft-1}
\end{figure}

\begin{figure}[tbp]
\centering
\includegraphics[width=0.95\linewidth,trim={0 8 0 8},clip]
{\detokenize{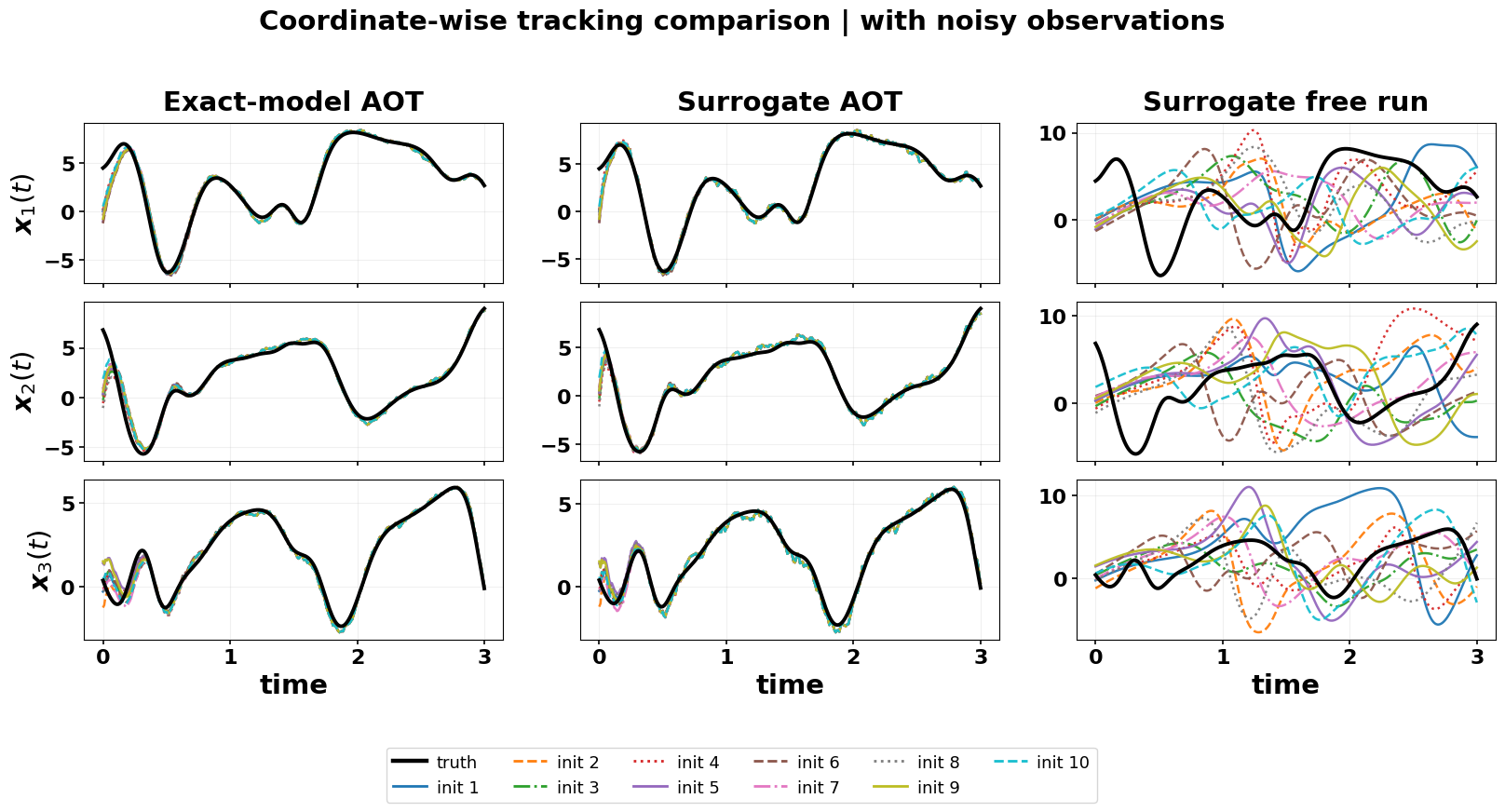}}
\caption{
Coordinate-wise tracking comparison in the direct vector-field learning route
with noisy observations.
}
\label{fig:numerics-coordinates-dft-noisy-1}
\end{figure}

\subsubsection{Sensitivity with respect to the feedback resolution}

We next study the sensitivity of surrogate AOT with respect to the feedback resolution parameter \(h\). In the stochastic setting, the influence of \(h\) is mixed with the observation noise entering the nudging term, making it difficult to isolate the role of the feedback resolution itself. Therefore, we focus on the noiseless setting for the \(h\)-sensitivity analysis.

For band-limited spectral measurements, we fix \(\mu=\mu^\ast=25\), as selected by the tuning procedure, and sweep over the admissible values \(h\in\{0,0.09,0.18\}\). This allows us to examine how the quality of the feedback operator affects the tracking performance of the learned surrogate dynamics.

\begin{figure}[t]
\centering

\begin{subfigure}[t]{0.40\textwidth}
    \centering
    \includegraphics[width=\linewidth,trim={0 7 0 6},clip]
    {\detokenize{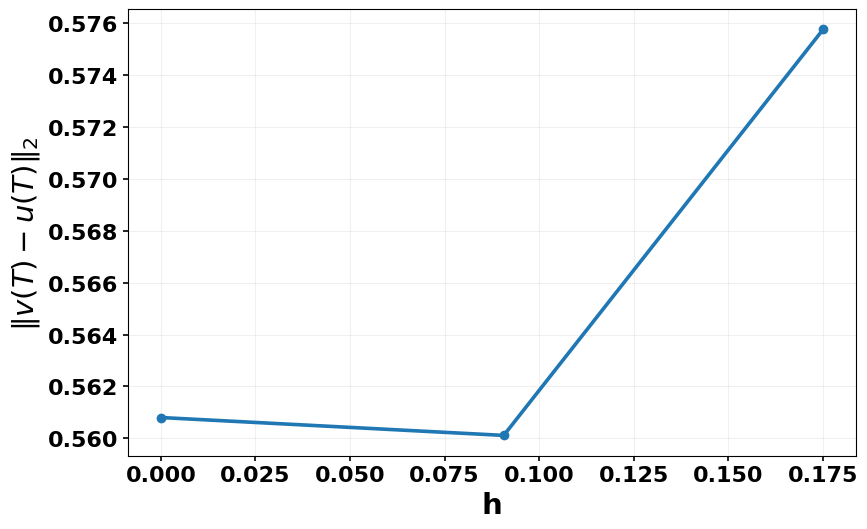}}
    \caption{Absolute tracking error.}
\end{subfigure}
\hspace{0.015\textwidth}
\begin{subfigure}[t]{0.40\textwidth}
    \centering
    \includegraphics[width=\linewidth,trim={0 7 0 6},clip]
    {\detokenize{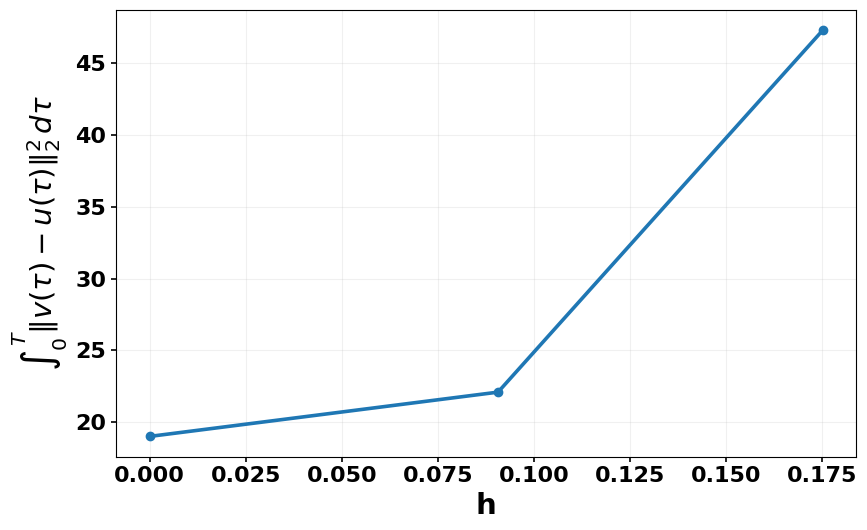}}
    \caption{Integrated squared error.}
\end{subfigure}

\caption{
Sensitivity of surrogate AOT performance to the feedback-resolution parameter
\(h\) in the direct vector-field learning route.
}
\label{fig:numerics-h-summary-dft-1}
\end{figure}

Figure~\ref{fig:numerics-h-summary-dft-1} shows the sensitivity of surrogate AOT performance to the feedback-resolution parameter \(h\) under band-limited spectral measurements. Smaller values of \(h\), corresponding to finer band-limited spectral resolution, generally improve tracking performance. The effect is more clearly visible in the integrated squared error, which accumulates the tracking error over the full assimilation window. In contrast, the final-time absolute error varies only mildly across the tested values of \(h\). This is consistent with the nearly realizable dictionary-learning setting: after synchronization, the surrogate AOT trajectory quickly reaches a residual error floor, so the terminal tracking errors are close to one another even when the transient behavior differs.

\subsubsection{Sensitivity with respect to the training sample size}

We also study the sensitivity of the direct vector-field learning route with respect to the training sample size, which serves as a numerical proxy for model error. To make the comparison as controlled as possible, we use a nested construction. We first generate the largest available training pool from \(N_{\mathrm{ref}}=5\) long reference trajectories, using a fixed stride of \(200\) time steps to reduce temporal correlation. From this common pool, we then form smaller nested training sets, so that the small, medium, and large sample-size regimes differ primarily in the number of training samples rather than in the qualitative coverage of phase space. This nested design makes the comparison more directly attributable to sample size, rather than to changes in data quality.

Figure~\ref{fig:numerics-sample-summary-dft-1} shows a clear improvement in tracking performance as the training sample size increases. This is consistent with the surrogate-tracking theory in Theorems~\ref{thm:sur-track} and \ref{thm:sur-track-noisy}: the long-time tracking floor is influenced by the surrogate error, and improved training generally reduces that error. The same overall trend remains visible in the noisy setting, although it is partially masked by the additional observation-noise floor. For the largest sample size, the surrogate AOT curves become comparable to, and in some metrics slightly outperform, the dashed exact-model AOT baseline. This is a finite-horizon effect reflecting the conservative choice of the nudging parameter used for the exact-model AOT baseline. Once the surrogate residual is sufficiently small, a different choice of \(\mu\) can lead to slightly smaller finite-time errors.

\begin{figure}[tbp]
\centering

\begin{subfigure}[t]{0.40\textwidth}
    \centering
    \includegraphics[width=\linewidth,trim={2 7 0 6},clip]
    {\detokenize{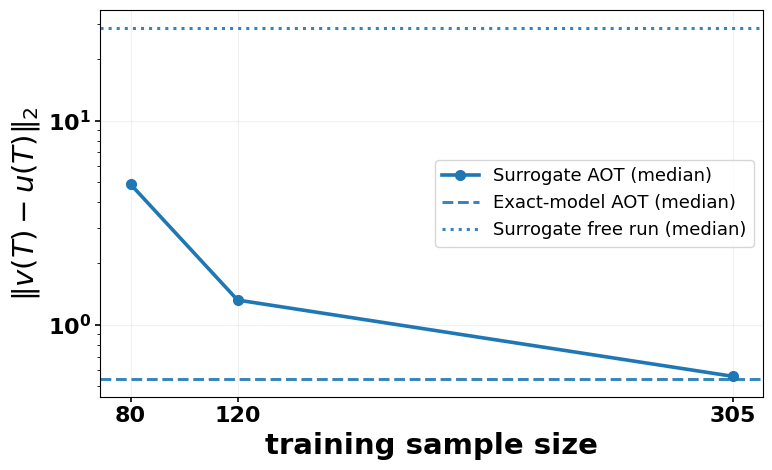}}
    \caption{Absolute tracking error.}
\end{subfigure}
\hspace{0.015\textwidth}
\begin{subfigure}[t]{0.40\textwidth}
    \centering
    \includegraphics[width=\linewidth,trim={0 7 0 6},clip]
    {\detokenize{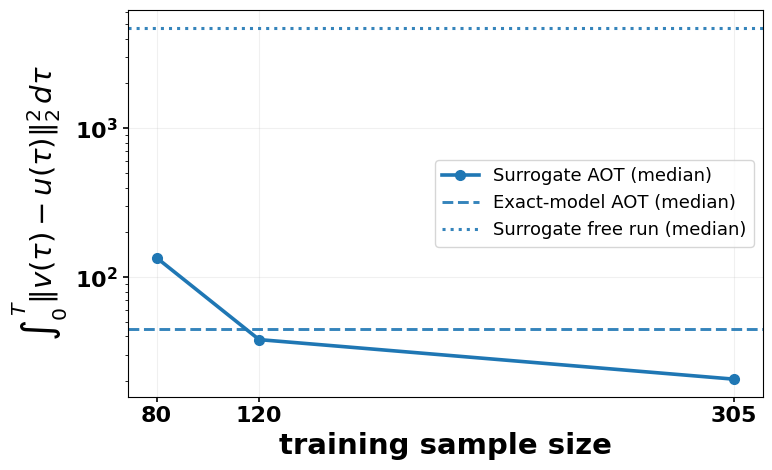}}
    \caption{Integrated squared error.}
\end{subfigure}

\vspace{0.6em}

\begin{subfigure}[t]{0.40\textwidth}
    \centering
    \includegraphics[width=\linewidth,trim={2 7 0 6},clip]
    {\detokenize{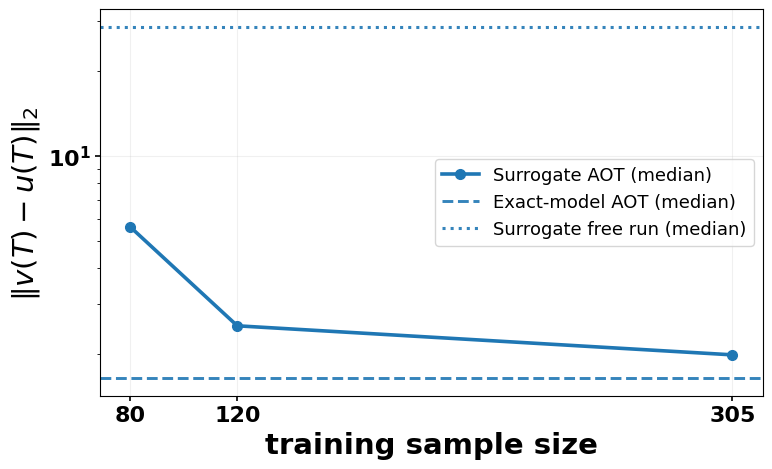}}
    \caption{Absolute tracking error (noisy observations).}
\end{subfigure}
\hspace{0.015\textwidth}
\begin{subfigure}[t]{0.40\textwidth}
    \centering
    \includegraphics[width=\linewidth,trim={0 7 0 6},clip]
    {\detokenize{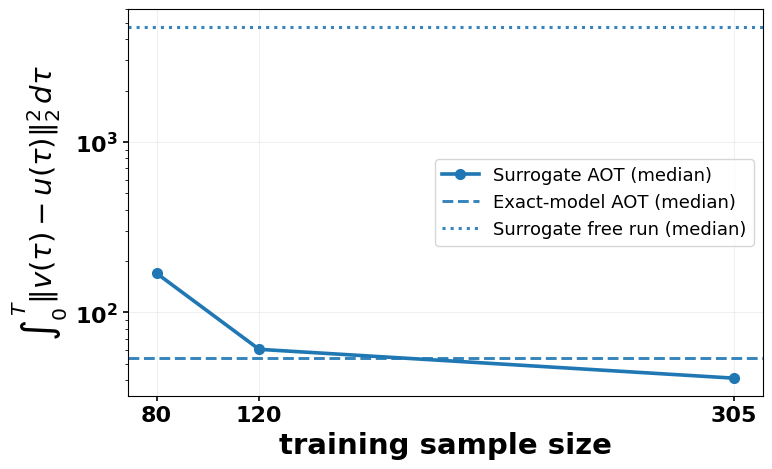}}
    \caption{Integrated squared error (noisy observations).}
\end{subfigure}

\caption{
Sensitivity of surrogate AOT performance to the training sample size in the direct vector-field learning route.
}
\label{fig:numerics-sample-summary-dft-1}
\end{figure}

\subsection{Solution-map learning}
\label{subsec:numerics-route2}
In the second learning route, we approximate the short-time solution map and then convert the learned map into an induced drift, as described in Subsection~\ref{subsec:bridge-route2}. This provides a more flexible and less problem-specific surrogate than the Lorenz--96 dictionary used in the direct vector-field learning route, without requiring prior knowledge of the structure of the vector field.

\paragraph{Training data set}
For each sampled state \(u^{(m)}(t_j)\in \mathcal X_{\rm train}\), we compute the short-time map label \(S_{\Delta t}(u^{(m)}(t_j))\), where \(S_{\Delta t}\) is the lag-\(\Delta t\) solution map defined in \eqref{eq:solution-map}. We also compute the corresponding first-order derivative label \(DS_{\Delta t}(u^{(m)}(t_j))\), which is the derivative information used in the solution-map learning error \(\eta_M^S\) in \eqref{eq:bridge-operator-errors}. Thus, the training data consist of the triples
\[
\Bigl(
 u^{(m)}(t_j),\,
S_{\Delta t}(u^{(m)}(t_j)),\,
DS_{\Delta t}(u^{(m)}(t_j))
\Bigr).
\]
 In the numerical implementation, we fix the solution-map lag at \(\Delta t=0.05\) for all solution-map learning experiments. The map label \(S_{\Delta t}(u^{(m)}(t_j))\) is then obtained by integrating the Lorenz--96 system from the initial state \(u^{(m)}(t_j)\) over the interval
\([0,\Delta t]\).
The Jacobian label \(DS_{\Delta t}(u^{(m)}(t_j))\) is computed by integrating the associated variational equation
\[
\dot J(t)=DF(u(t;u^{(m)}(t_j)))\,J(t),
\qquad J(0)=I,
\]
over \([0,\Delta t]\) with the same time discretization; the terminal value \(J(\Delta t)\) gives \(DS_{\Delta t}(u^{(m)}(t_j))\).

Our main surrogate follows the solution-map learning approach in Subsection~\ref{subsec:complexity-route2}. We train a ReLU network \(G_\theta:\R^d\to\R^d\) to approximate \(S_{\Delta t}\) on the above training data, and define the induced local drift by
\[
\widehat F_M(z)=\frac{G_\theta(z)-z}{\Delta t}.
\]
The global surrogate drift \(F_M\) is then obtained from \(\widehat F_M\) by the cutoff construction described in Subsection~\ref{subsec:bridge-route2}. The complexity analysis for the solution-map learning route in Subsection~\ref{subsec:complexity-route2} is formulated in the DSRN framework, since that architecture is more convenient for establishing Sobolev approximation and generalization bounds. In the numerical experiments, however, we use a standard fully connected ReLU network as a simpler practical proxy; this is easier to implement and already sufficient to test the solution-map learning mechanism predicted by the theory. Concretely, the network has width \(512\) and depth \(4\), and is trained with the loss
\[
\mathcal L(\theta)
=
\frac{1}{M}\sum_{j=1}^M
\bigl\|G_\theta(z_j)-S_{\Delta t}(z_j)\bigr\|_2^2
+
\lambda_{\rm jac}\,
\frac{1}{M}\sum_{j=1}^M
\bigl\|DG_\theta(z_j)-DS_{\Delta t}(z_j)\bigr\|_{\mathrm F}^2,
\]
with \(\lambda_{\rm jac}=0.20\), where \(\{z_j\}_{j=1}^M\) denotes the sampled training states in \(\mathcal X_{\mathrm{train}}\). Here \(DG_\theta(z_j)\) is computed by automatic differentiation in \textbf{PyTorch}. Before training, each input coordinate is normalized using the empirical mean and standard deviation computed from the training set. The final model is chosen as the one achieving the best performance on a held-out validation set.

\subsubsection{Exact-model AOT baseline}

As in the direct vector-field learning route, we begin with the exact-model AOT baseline under a representative configuration of band-limited spectral measurements. Specifically, we fix the resolution parameter at the median value of the tested DFT resolution grid
\[
h \in \{0.0,0.015,0.03,0.048,0.067\},
\]
which corresponds to the DFT truncation orders \(K\in\{40,38,36,34,32\}\). We use the same baseline choice in the subsequent surrogate experiments, so that all comparisons are performed under a common feedback resolution. For the nudging parameter, we take the smallest value in the tested grid
\[
\mu\in\{20,40,60,80\},
\]
to emphasize that the exact-model nudged dynamics already synchronize reliably even under the weakest tested nudging strength.  
The surrogate AOT runs below use separately tuned \(\mu\), allowing the feedback strength to compensate for the additional sensitivity introduced by model error.

Compared with the direct vector-field learning route based on dictionary learning, where the Lorenz--96-specific dictionary is nearly realizable and the surrogate error is already small, we use a slightly stronger nudging-parameter grid for the solution-map learning route. This choice helps compensate for the larger residual error of the solution-map-based surrogate and provides a more aggressive correction of the learned dynamics.

Figure~\ref{fig:numerics-exact-baseline-dft} shows the same qualitative behavior as in the preceding direct vector-field learning experiments. Even under this conservative parameter choice, the exact-model nudged dynamics synchronize rapidly and maintain a small long-time tracking error, consistent with Theorem~\ref{thm:conv}.

\begin{figure}[tbp]
\centering

\begin{subfigure}[t]{0.40\textwidth}
    \centering
    \includegraphics[width=\linewidth,trim={0 7 0 7},clip]
    {\detokenize{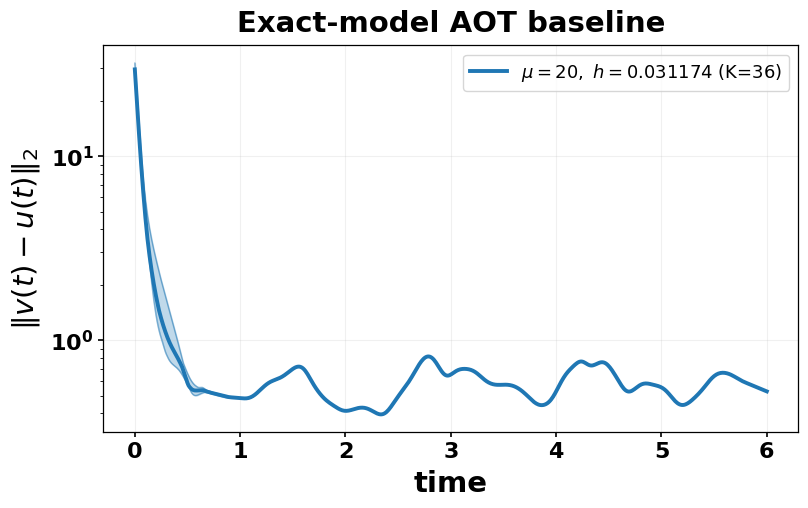}}
    \caption{Absolute tracking error.}
\end{subfigure}
\hspace{0.015\textwidth}
\begin{subfigure}[t]{0.40\textwidth}
    \centering
    \includegraphics[width=1.05\linewidth,trim={0 7 0 7},clip]
    {\detokenize{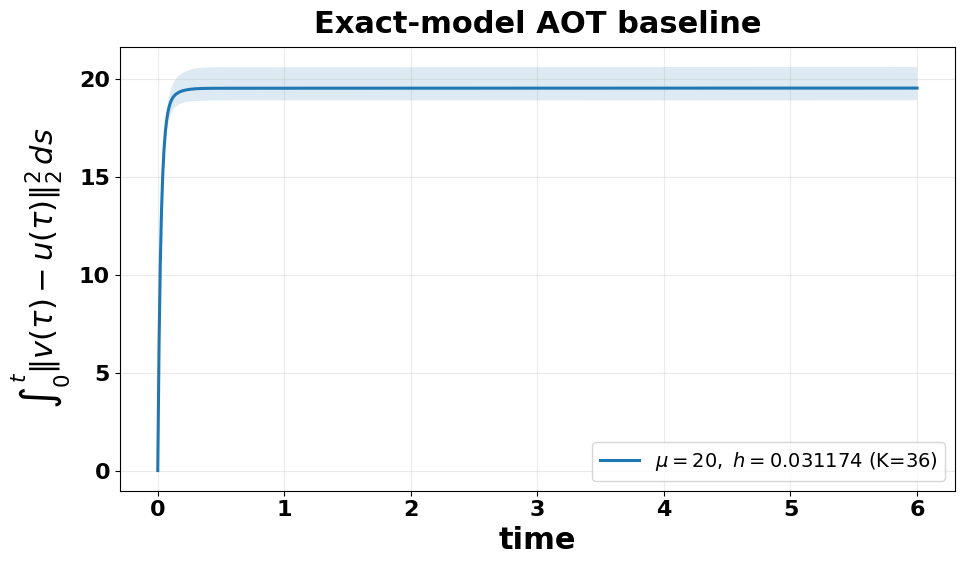}}
    \caption{Integrated squared error.}
\end{subfigure}

\caption{
Tracking performance of the exact-model AOT baseline.
}
\label{fig:numerics-exact-baseline-dft}
\end{figure}

\subsubsection{Surrogate AOT obtained from solution-map learning}

We now turn to the surrogate obtained from solution-map learning. We use the same comparison protocol as in the direct vector-field learning route, replacing the exact drift \(F\) by the drift \(F_M\) induced by the learned short-time solution map. In all cases, we compare exact-model AOT, surrogate AOT, and the corresponding free surrogate run.

As in the direct vector-field learning route, we use the procedure described in Appendix~\ref{app:mu-selection-dft} to choose the nudging strength from the prescribed grid \(\mu\in\{20,40,60,80\}\) before running the main comparisons. This gives \(\mu^\ast=80\) in the noiseless setting and \(\mu^\ast=40\) in the noisy-observation setting, which are used throughout the following subsections.

Figure~\ref{fig:numerics-surrogate-errors-dft} shows that, even when the surrogate is constructed from a learned short-time solution map rather than a directly learned vector field, surrogate AOT can still track the true trajectory in both the noiseless and noisy settings. As in the direct vector-field learning route, the surrogate introduces an additional model-error floor, while the noisy-observation setting further increases the residual level due to stochastic noise in the feedback. Overall, the experiment again demonstrates that the AOT algorithm can stabilize learned surrogate dynamics, even when the surrogate approximation error is more visible.

\begin{figure}[tbp]
\centering

\begin{subfigure}[t]{0.44\textwidth}
    \centering
    \includegraphics[width=0.9\linewidth,trim={0 7 0 7},clip]
    {\detokenize{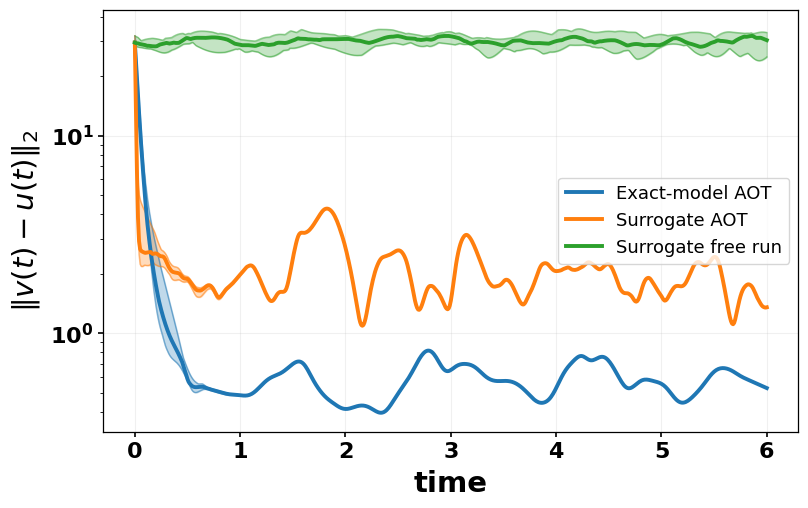}}
    \caption{Tracking error comparison.}
\end{subfigure}
\hspace{-0.015\textwidth}
\begin{subfigure}[t]{0.44\textwidth}
    \centering
    \includegraphics[width=0.9\linewidth,trim={0 7 0 7},clip]
    {\detokenize{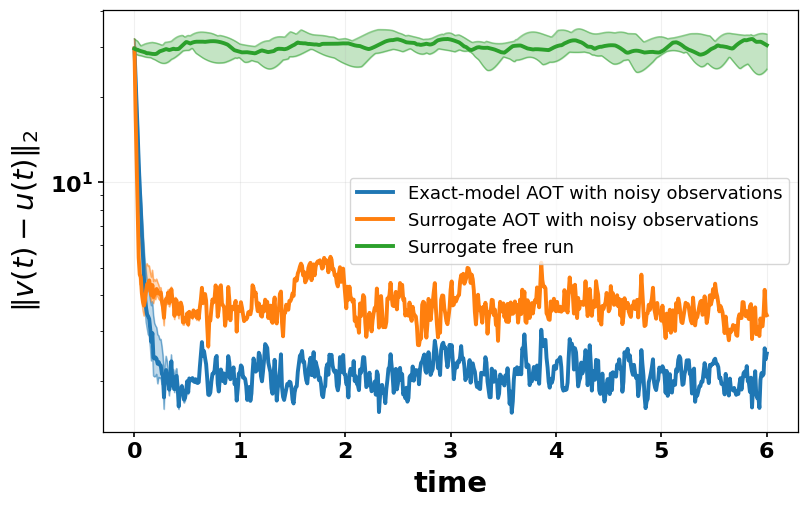}}
    \caption{Tracking error comparison (noisy observations).}
\end{subfigure}

\caption{
Tracking-error comparisons for surrogate AOT under band-limited spectral measurements in the solution-map learning route.
}
\label{fig:numerics-surrogate-errors-dft}
\end{figure}

To further assess the effectiveness of the AOT algorithm with learned surrogate dynamics, we also conduct representative coordinate-level comparisons in the noiseless and noisy settings; see Figures~\ref{fig:numerics-coordinates-dft} and~\ref{fig:numerics-coordinates-dft-noisy}, respectively. In the noiseless case, Figure~\ref{fig:numerics-coordinates-dft} shows the same qualitative behavior as in the direct vector-field learning route: exact-model AOT synchronizes most rapidly, surrogate AOT stays much closer to the truth than the free surrogate, and the remaining discrepancy is more visible because the solution-map surrogate is learned by a generic ReLU network, leading to a larger model error than in the problem-specific dictionary-learning route. Under noisy observations, Figure~\ref{fig:numerics-coordinates-dft-noisy} shows that the same ordering persists, although the noise induces visible fluctuations and a larger residual level. Above all, surrogate AOT still clearly outperforms the free surrogate, confirming that the AOT algorithm continues to stabilize the learned dynamics in the stochastic setting.

\begin{figure}[tbp]
\centering
\includegraphics[width=0.95\linewidth,trim={0 8 0 8},clip]
{\detokenize{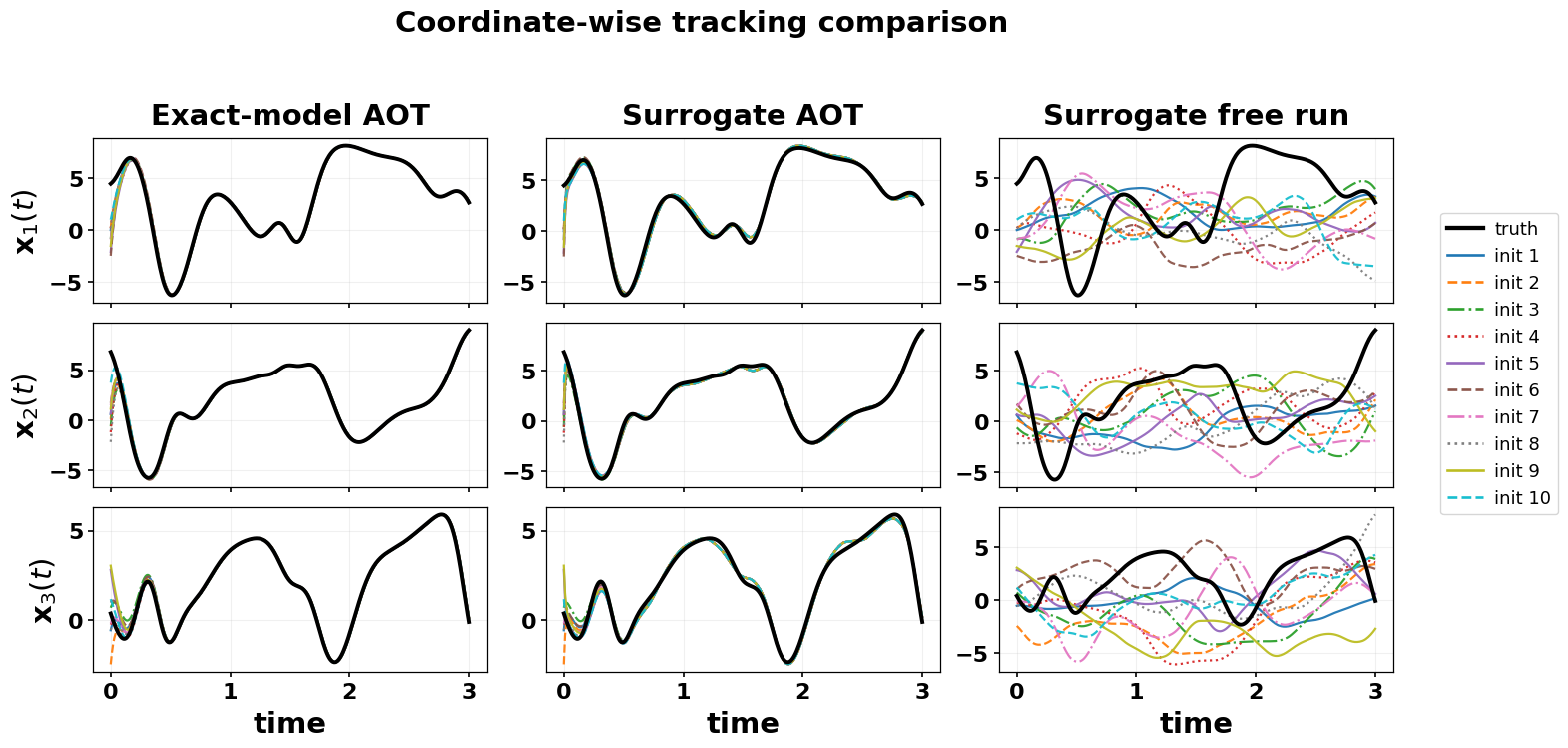}}
\caption{
Coordinate-wise tracking comparison in the solution-map learning route.
}
\label{fig:numerics-coordinates-dft}
\end{figure}

\begin{figure}[tbp]
\centering
\includegraphics[width=0.95\linewidth,trim={0 8 0 8},clip]
{\detokenize{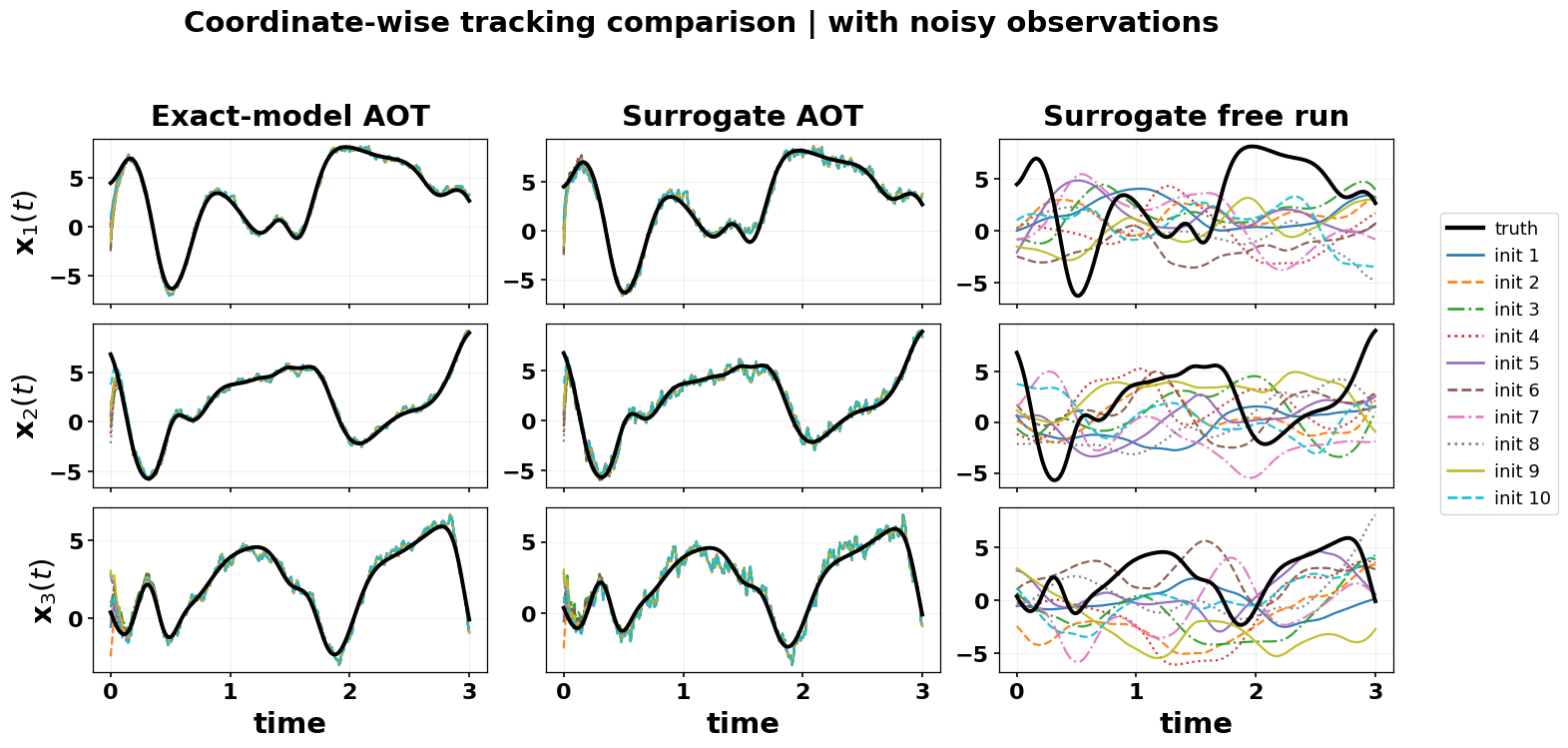}}
\caption{
Coordinate-wise tracking comparison in the solution-map learning route with noisy observations.
}
\label{fig:numerics-coordinates-dft-noisy}
\end{figure}

\subsubsection{Sensitivity with respect to the feedback resolution}

Similarly, we examine the sensitivity of the solution-map-based surrogate nudged dynamics to the band-limited spectral feedback resolution in the noiseless setting. We fix \(\mu=\mu^\ast=80\) and sweep over
\(h\in \{0.0,0.015,0.03,0.048,0.067\}.\)

\begin{figure}[t]
\centering

\begin{subfigure}[t]{0.40\textwidth}
    \centering
    \includegraphics[width=\linewidth,trim={0 7 0 7},clip]
    {\detokenize{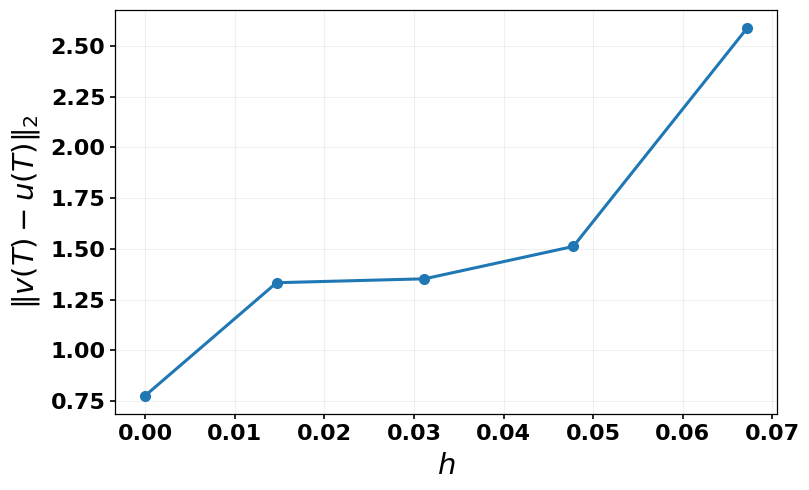}}
    \caption{Absolute tracking error.}
\end{subfigure}
\hspace{0.015\textwidth}
\begin{subfigure}[t]{0.40\textwidth}
    \centering
    \includegraphics[width=\linewidth,trim={0 7 0 7},clip]
    {\detokenize{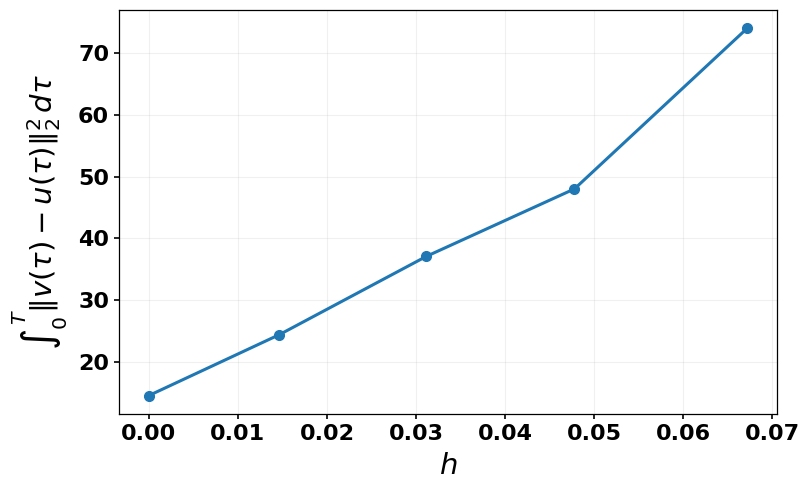}}
    \caption{Integrated squared error.}
\end{subfigure}

\caption{
Sensitivity of surrogate AOT performance to the feedback-resolution parameter \(h\) in the solution-map learning route.
}
\label{fig:numerics-h-summary-dft}
\end{figure}

Figure~\ref{fig:numerics-h-summary-dft} shows the same qualitative trend as in the exact-model and direct vector-field learning cases: finer feedback resolution improves tracking, while coarser feedback degrades it. The dependence on \(h\) is somewhat more visible here, which is consistent with the larger residual error introduced by the ReLU approximation of the short-time solution map.

\subsubsection{Sensitivity with respect to the training sample size}

We next study the effect of the training sample size in the solution-map learning route. As before, we use a nested data construction: starting from a common largest pool of trajectory segments, we form smaller training sets by subsampling from the same pool and retrain the surrogate on each set. This reduces the influence of differences in data quality.

\begin{figure}[tbp]
\centering

\begin{subfigure}[t]{0.40\textwidth}
    \centering
    \includegraphics[width=\linewidth,trim={0 7 0 6},clip]
    {\detokenize{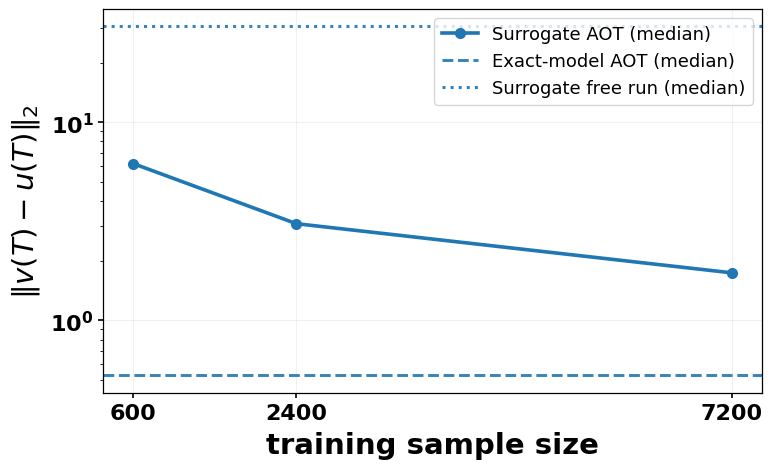}}
    \caption{Absolute tracking error.}
\end{subfigure}
\hspace{0.015\textwidth}
\begin{subfigure}[t]{0.40\textwidth}
    \centering
    \includegraphics[width=\linewidth,trim={0 7 0 6},clip]
    {\detokenize{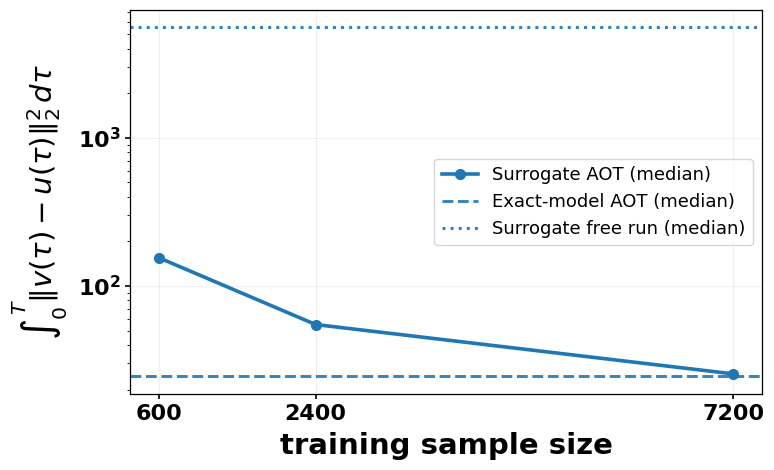}}
    \caption{Integrated squared error.}
\end{subfigure}

\vspace{0.6em}

\begin{subfigure}[t]{0.40\textwidth}
    \centering
    \includegraphics[width=\linewidth,trim={0 7 0 6},clip]
    {\detokenize{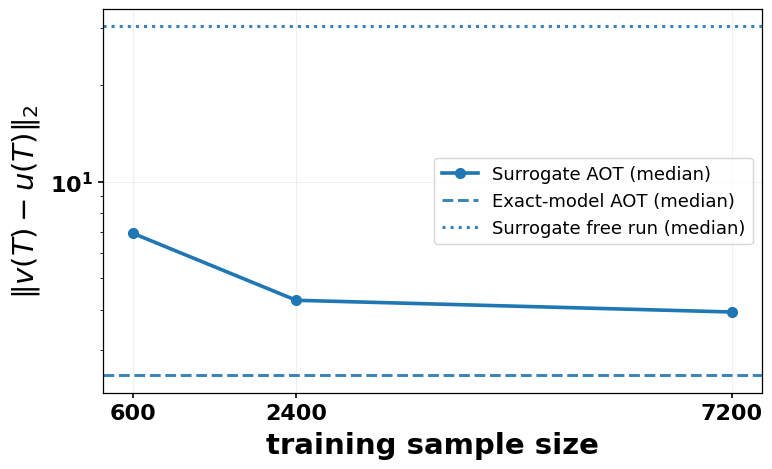}}
    \caption{Absolute tracking error (noisy observations).}
\end{subfigure}
\hspace{0.015\textwidth}
\begin{subfigure}[t]{0.40\textwidth}
    \centering
    \includegraphics[width=\linewidth,trim={0 7 0 6},clip]
    {\detokenize{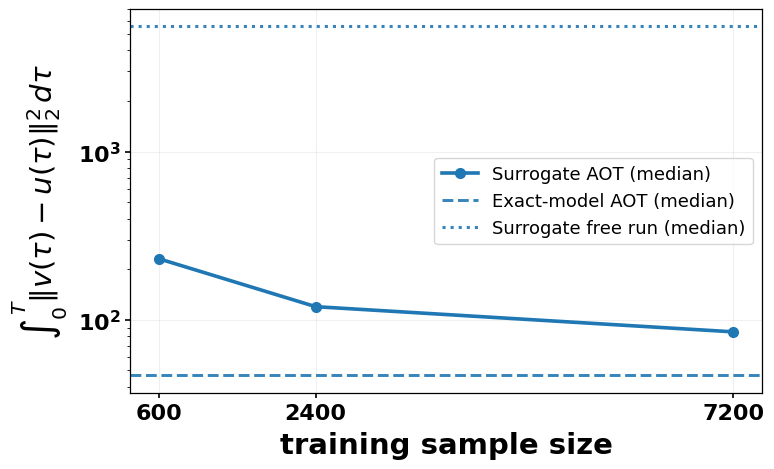}}
    \caption{Integrated squared error (noisy observations).}
\end{subfigure}

\caption{
Sensitivity of surrogate AOT performance to the training sample size in the solution-map learning route.
}
\label{fig:numerics-sample-summary-dft}
\end{figure}

Figure~\ref{fig:numerics-sample-summary-dft} shows that larger training sets generally lead to better tracking. This is consistent with the tracking analysis in Theorem~\ref{thm:sur-track}: improving the approximation of the solution map reduces the surrogate residuals and hence lowers the tracking error floor. In the noisy setting, the same trend remains visible, although it is partially masked by the observation-noise floor.

\subsection{Discussion of numerical findings}

Our numerical experiments support the main qualitative predictions of the tracking analysis. The exact-model nudged dynamics exhibit the expected dependence on the feedback resolution and nudging strength. In both learning routes, surrogate AOT substantially improves over the corresponding free learned dynamics, showing that the AOT algorithm remains effective when the exact drift is replaced by a learned surrogate.

The sensitivity studies are also consistent with the theory: finer feedback resolution and larger training sets improve long-time tracking. Compared with the dictionary-learning route, the solution-map learning route displays a more visible residual error floor, consistent with its use of a generic ReLU approximation rather than a nearly realizable, problem-specific dictionary.

\section{Conclusions}
\label{sec:conclusion}

This paper developed a unified finite-dimensional framework for continuous data assimilation with learned surrogate dynamics. We analyzed an exact-model AOT baseline under structural assumptions on the dynamics and the feedback operator, rather than using model-specific arguments. Within this framework, we established global well-posedness, post-absorption bounds for the true and nudged systems, and exponential tracking guarantees in both the noise-free and noisy-observation settings. We then incorporated learned surrogate dynamics into the AOT framework and quantified how the resulting model error affects synchronization. This gives explicit tracking bounds in which the long-time error can be controlled by local surrogate residuals on the post-absorption region and, in the stochastic setting, by the observation noise.

For constructing AOT-valid surrogates, we provided two learning routes, direct vector-field learning and short-time solution-map learning, and derived checkable criteria linking the corresponding learning errors to the residual quantities that enter the tracking estimates, without necessarily requiring explicit prior structural knowledge of the force field. We further provided learning-theoretic guarantees showing that such surrogates can be learned from finite samples while preserving the long-time synchronization mechanism, thereby tying the statistical analysis directly to the dynamical criterion required by the AOT tracking theory.

The main results are formulated in a finite-dimensional setting, where the standing assumptions in Section~\ref{subsec:baseline-setting} directly imply global well-posedness and the equivalence of the \(H\)- and \(V\)-norms turns \(H\)-level post-absorption bounds into the \(V\)-level bounds needed for the tracking estimates. These properties are no longer automatic in an infinite-dimensional Hilbert space setting. Nevertheless, the underlying well-posedness and tracking arguments are based on energy estimates and are not intrinsically finite-dimensional. Thus, in an infinite-dimensional formulation where the true system is known to admit a global strong solution at the required regularity level, the well-posedness of the corresponding nudged and surrogate-nudged systems can be obtained under similar structural assumptions by standard Galerkin approximation and energy-estimate arguments. If, in addition, the required post-absorption \(V\)-bounds are available, either by assumption or by separate \(V\)-level energy estimates, then the deterministic tracking and surrogate-tracking results of Sections~\ref{sec:baseline} and \ref{sec:surrogate} carry over in the same form, with the same type of conditions on \(\mu\) and \(h\), up to changes in the constants. This perspective is consistent with the original AOT framework \cite{AOT2013}, fluid and geophysical models \cite{biswas2021cda3dns,balakrishna2022determining,JollyMartinezTiti2017}, and abstract semilinear parabolic formulations \cite{delsarto2026cda}. The same interpretation applies to the noisy-observation results when the stochastic forcing is well defined and the quadratic-variation term in the energy estimate
is finite. By contrast, the sample-complexity results in Section~\ref{sec:complexity} remain finite-dimensional. Extending them to infinite-dimensional input spaces and surrogate classes
requires distinct approximation and statistical-complexity theory, which we leave for future work.

\section*{Acknowledgments}
The authors were partly funded by the NSF CAREER award DMS-2237628.

\bibliographystyle{siam}
\bibliography{reference}

\newpage
\appendix

\section{Additional example verifications}
\label{app:examples}

\begin{proof}[\textbf{Verification for the Stuart--Landau oscillator}]
For \eqref{eq:stuart-landau}, we take
\[
H=V=\R^2,\qquad A=\Id,\qquad \nu=1,\qquad f=0,
\]
and
\[
N(u):=-(\lambda+1)u-\omega Ju+|u|^2u.
\]
Assumption~\ref{ass:H1} is immediate.

We first record the basic energy identity. Since \(J\) is skew-symmetric, we have
\[
\ip{Jz}{z}_H=0
\qquad \forall\,z\in\R^2.
\]
Hence, for every \(z\in\R^2\),
\[
\ip{N(z)}{z}_H
=
-(\lambda+1)|z|^2-\omega\ip{Jz}{z}_H+|z|^4
=
|z|^4-(\lambda+1)|z|^2.
\]

To verify Assumption~\ref{ass:H4}, we compute
\[
\nu\,z^\top A z+\ip{N(z)}{z}_H
=
|z|^2+|z|^4-(\lambda+1)|z|^2
=
|z|^4-\lambda |z|^2.
\]
Let \(r:=|z|^2\ge0\). Then, for any fixed \(\alpha>0\),
\[
-\lambda r+r^2
=
\alpha r+\bigl(r^2-(\lambda+\alpha)r\bigr)
=
\alpha r+\left(r-\frac{\lambda+\alpha}{2}\right)^2
-\frac{(\lambda+\alpha)^2}{4}.
\]
Therefore
\[
-\lambda r+r^2
\ge
\alpha r-\frac{(\lambda+\alpha)^2}{4}.
\]
Equivalently,
\[
\nu\,z^\top A z+\ip{N(z)}{z}_H
\ge
\alpha\,\normH{z}^2-\beta,
\qquad
\beta:=\frac{(\lambda+\alpha)^2}{4}.
\]
Thus Assumption~\ref{ass:H4} holds for every choice of \(\alpha>0\), with the
corresponding \(\beta\) above.

We now verify Assumption~\ref{ass:H2}. From the identity for
\(\ip{N(z)}{z}_H\) above,
\[
\ip{N(z)}{z}_H
=
|z|^4-(\lambda+1)|z|^2
\ge
-(|\lambda|+1)|z|^2.
\]
Thus Assumption~\ref{ass:H2}\emph{(i)} holds with
\(C_{\mathrm E}=|\lambda|+1\).

Since \(N\in C^\infty(\R^2;\R^2)\), it is locally Lipschitz on \(\R^2\).
Moreover, Assumption~\ref{ass:H2}\emph{(ii)} follows from
Remark~\ref{rem:examples-local-mono}.
\end{proof}

\begin{proof}[\textbf{Verification for the FitzHugh--Nagumo system}]
Here we use \((u_1,u_2)\) to denote the intrinsic phase variables of the
FitzHugh--Nagumo system. Let \(u=(u_1,u_2)^\top\). We fit \eqref{eq:fhn} into
\eqref{eq:truth} by taking
\[
H=V=\R^2,\qquad
A=
\begin{pmatrix}
1&0\\
0&\varepsilon b
\end{pmatrix},
\qquad
\nu=1,
\qquad
f=
\begin{pmatrix}
I\\
\varepsilon a
\end{pmatrix},
\]
and
\[
N(u)=
\begin{pmatrix}
-2u_1+\frac{u_1^3}{3}+u_2\\
-\varepsilon u_1
\end{pmatrix}.
\]
Assumption~\ref{ass:H1} holds because \(A\) is symmetric positive definite, with \(\epsilon,b>0\).

To verify Assumption~\ref{ass:H4}, let \(z=(z_1,z_2)^\top\in\R^2\). We have
\[
z^\top A z+\ip{N(z)}{z}_H
=
z_1^2+\varepsilon b\,z_2^2
+\left(-2z_1+\frac{z_1^3}{3}+z_2\right)z_1
-\varepsilon z_1z_2,
\]
and hence
\[
z^\top A z+\ip{N(z)}{z}_H
=
\frac{z_1^4}{3}-z_1^2+(1-\varepsilon)z_1z_2+\varepsilon b\,z_2^2.
\]
Fix any \(\alpha\in(0,\varepsilon b)\). Then
\[
z^\top A z+\ip{N(z)}{z}_H-\alpha\normH{z}^2
=
\frac{z_1^4}{3}
-(1+\alpha)z_1^2
+(1-\varepsilon)z_1z_2
+(\varepsilon b-\alpha)z_2^2.
\]
Since \(\varepsilon b-\alpha>0\), Young's inequality gives
\[
(1-\varepsilon)z_1z_2
\ge
-\frac{(1-\varepsilon)^2}{4(\varepsilon b-\alpha)}\,z_1^2
-(\varepsilon b-\alpha)z_2^2.
\]
Therefore
\[
z^\top A z+\ip{N(z)}{z}_H-\alpha\normH{z}^2
\ge
\frac{z_1^4}{3}
-
C_\alpha z_1^2,
\quad
C_\alpha
:=
1+\alpha+\frac{(1-\varepsilon)^2}{4(\varepsilon b-\alpha)}.
\]
Setting \(r:=z_1^2\ge0\), we have
\[
\frac{r^2}{3}-C_\alpha r
\ge
-\frac{3}{4}C_\alpha^2.
\]
Hence
\[
z^\top A z+\ip{N(z)}{z}_H
\ge
\alpha \normH{z}^2-\beta,
\qquad
\beta:=\frac{3}{4}C_\alpha^2.
\]
Thus Assumption~\ref{ass:H4} holds for every \(\alpha\in(0,\varepsilon b)\),
with the corresponding \(\beta\) defined above.

We now verify Assumption~\ref{ass:H2}. First,
\[
\ip{N(z)}{z}_H
=
z_1\left(-2z_1+\frac{z_1^3}{3}+z_2\right)+z_2(-\varepsilon z_1)
=
\frac{z_1^4}{3}-2z_1^2+(1-\varepsilon)z_1z_2.
\]
By Young's inequality,
\[
|(1-\varepsilon)z_1z_2|
\le
\frac{|1-\varepsilon|}{2}(z_1^2+z_2^2).
\]
Therefore
\[
\ip{N(z)}{z}_H
\ge
\frac{z_1^4}{3}
-
\left(2+\frac{|1-\varepsilon|}{2}\right)z_1^2
-
\frac{|1-\varepsilon|}{2}z_2^2
\ge
-C_{\mathrm E}\|z\|_2^2,
\]
with
\(C_{\mathrm E}:=
2+\frac{|1-\varepsilon|}{2}.\)
Thus Assumption~\ref{ass:H2}\emph{(i)} holds.

Since \(N\) is polynomial, it is locally Lipschitz on \(\R^2\). Moreover,
Assumption~\ref{ass:H2}\emph{(ii)} follows from
Remark~\ref{rem:examples-local-mono}.
\end{proof}

\begin{proof}[\textbf{Verification for the Lorenz--96 system}]
For the Lorenz--96 system \eqref{eq:l96}, we take
\[
H=V=\R^d,\qquad A=\Id,\qquad \nu=1,\qquad f=\mathsf{f}\mathbf 1,
\]
and
\[
N(u)_i:=-(u_{i+1}-u_{i-2})u_{i-1}.
\]
Assumption~\ref{ass:H1} is immediate since \(A=\Id\) is symmetric positive
definite.

We next verify Assumption~\ref{ass:H4}. For a
general point \(z\in\R^d\), a direct computation gives,
\[
\ip{N(z)}{z}_H
=
-\sum_{i=1}^d z_i(z_{i+1}-z_{i-2})z_{i-1}
=
-\sum_{i=1}^d z_iz_{i+1}z_{i-1}
+\sum_{i=1}^d z_iz_{i-2}z_{i-1}.
\]
By periodicity, the second sum is a reindexing of the first. Indeed, with
\(j=i-1\),
\[
\sum_{i=1}^d z_iz_{i-2}z_{i-1}
=
\sum_{j=1}^d z_{j+1}z_{j-1}z_j
=
\sum_{j=1}^d z_jz_{j+1}z_{j-1}.
\]
Hence the two terms cancel, and therefore
\(\ip{N(z)}{z}_H=0\), for all \(z\in\R^d.\)
Consequently,
\[
\nu\,z^\top A z+\ip{N(z)}{z}_H
=
\|z\|_2^2.
\]
Thus Assumption~\ref{ass:H4} holds with
\(\alpha=1,
\,
\beta=0.\)

We now verify Assumption~\ref{ass:H2}. Since
\(\ip{N(z)}{z}_H=0\), Assumption~\ref{ass:H2}\emph{(i)} holds with
\(C_{\mathrm E}=0\).

Since \(N\) is a smooth polynomial vector field, it is locally Lipschitz on
\(\R^d\). Moreover, Assumption~\ref{ass:H2}\emph{(ii)} follows from
Remark~\ref{rem:examples-local-mono}.
This completes the verification.
\end{proof}

\begin{proof}[\textbf{Verification for finite-dimensional Galerkin-type models of viscous incompressible flows}]
We verify that a standard class of finite-dimensional Galerkin-type models for
viscous incompressible flows fits into the abstract framework of
Section~\ref{sec:baseline}. Consider reduced models of the form
\begin{equation}\label{eq:galerkin-fluid-app}
\dot u+\nu\Lambda u+\mathcal B(u,u)=g,
\qquad
u(t)\in\R^d,
\end{equation}
where \(\Lambda\) is symmetric positive definite and
\(\mathcal B:\R^d\times\R^d\to\R^d\) is bilinear. The verification below uses
the standard incompressible energy-cancellation identity
\begin{equation}\label{eq:galerkin-energy-preserving-app}
\ip{\mathcal B(z,z)}{z}_H=0,
\qquad
\forall\,z\in\R^d.
\end{equation}

This structure arises, for example, from Fourier--Galerkin truncations of the
two-dimensional Navier--Stokes equations on a periodic domain, and more
generally from Galerkin reductions based on an \(L^2\)-orthonormal
divergence-free reduced velocity space with boundary conditions preserving the
incompressible energy identity, such as periodic, no-penetration, or no-slip
conditions. The \(L^2\)-orthonormality allows the \(L^2\)-energy of the
reduced velocity field to be written as the Euclidean energy of its coefficient
vector in \(\R^d\), while the divergence-free and boundary assumptions remove
the pressure contribution and the boundary term in the convective energy
balance.

To indicate the origin of the bilinear term, let \(X_d=\operatorname{span}\{\varphi_1,\dots,\varphi_d\}\) be an \(L^2\)-orthonormal divergence-free reduced velocity space satisfying the above boundary assumptions, and write
\[
u_d(x,t)=\sum_{j=1}^d u_j(t)\varphi_j(x).
\]
Projecting the convective nonlinearity \((u_d\cdot\nabla)u_d\) onto \(X_d\)
gives a quadratic term with components
\[
\mathcal B_i(u,u)=\sum_{j,k=1}^d b_{ijk}u_ju_k,
\qquad
b_{ijk}:=\ip{(\varphi_j\cdot\nabla)\varphi_k}{\varphi_i}_{L^2}.
\]
Indeed,
\[
\ip{(u_d\cdot\nabla)u_d}{\varphi_i}_{L^2}
=
\sum_{j,k=1}^d u_j u_k
\ip{(\varphi_j\cdot\nabla)\varphi_k}{\varphi_i}_{L^2}.
\]
Moreover, if \(u_z=\sum_{j=1}^d z_j\varphi_j\), then
\[
\ip{\mathcal B(z,z)}{z}_H
=
\ip{(u_z\cdot\nabla)u_z}{u_z}_{L^2}
=
\frac12\int_\Omega u_z\cdot\nabla |u_z|^2\,dx.
\]
Using integration by parts,
\[
\frac12\int_\Omega u_z\cdot\nabla |u_z|^2\,dx
=
\frac12\int_{\partial\Omega} |u_z|^2(u_z\cdot n)\,dS
-
\frac12\int_\Omega (\nabla\cdot u_z)|u_z|^2\,dx.
\]
The second term vanishes because \(u_z\) is divergence-free, and the boundary
term vanishes under the stated periodic, no-penetration, or no-slip boundary
conditions. Hence \eqref{eq:galerkin-energy-preserving-app} holds.

We now verify Assumptions~\ref{ass:H1}, \ref{ass:H4}, and \ref{ass:H2}. Set
\[
H=V=\R^d,
\qquad
A=\Lambda,
\qquad
f=g,
\qquad
N(u)=\mathcal B(u,u).
\]
Since \(\Lambda\) is symmetric positive definite,
Assumption~\ref{ass:H1} holds immediately.

For Assumption~\ref{ass:H4}, using
\eqref{eq:galerkin-energy-preserving-app}, we obtain, for any \(z\in\R^d\),
\[
\nu\,z^\top A z+\ip{N(z)}{z}_H
=
\nu\,z^\top\Lambda z
\ge
\nu\,\lambda_{\min}(\Lambda)\normH{z}^2.
\]
Thus Assumption~\ref{ass:H4} holds with
\(\alpha=\nu\lambda_{\min}(\Lambda),
\,
\beta=0.\)

We now verify Assumption~\ref{ass:H2}. First,
\[
\ip{N(z)}{z}_H
=
\ip{\mathcal B(z,z)}{z}_H
=
0,
\qquad
\forall\,z\in\R^d.
\]
Thus Assumption~\ref{ass:H2}\emph{(i)} holds with \(C_{\mathrm E}=0\).

It remains to verify the local Lipschitz property and
Assumption~\ref{ass:H2}\emph{(ii)}. Since \(\mathcal B\) is bilinear on the
finite-dimensional space \(\R^d\), there is a constant \(C_{\mathcal B}>0\)
such that
\[
\|\mathcal B(z,z')\|_2
\le C_{\mathcal B}\|z\|_2\|z'\|_2,
\qquad
\forall\,z,z'\in\R^d.
\]
Therefore,
\[
N(z')-N(z)
=
\mathcal B(z'-z,z')+\mathcal B(z,z'-z),
\]
and hence
\[
\|N(z')-N(z)\|_2
\le
C_{\mathcal B}\bigl(\|z\|_2+\|z'\|_2\bigr)\|z'-z\|_2.
\]
On the ball \(\|z\|_2,\|z'\|_2\le R\), this gives
\[
\|N(z')-N(z)\|_2
\le
2C_{\mathcal B}R\,\|z'-z\|_2.
\]
In particular, \(N\) is locally Lipschitz on \(\R^d\). Moreover,
\[
-\ip{N(z')-N(z)}{z'-z}_H
\le
\|N(z')-N(z)\|_2\,\|z'-z\|_2
\le
2C_{\mathcal B}R\,\|z'-z\|_2^2.
\]
Thus Assumption~\ref{ass:H2}\emph{(ii)} holds on bounded \(H\)-balls, with
\(C_R=2C_{\mathcal B}R\).

We have therefore verified that any reduced model of the form
\eqref{eq:galerkin-fluid-app} satisfying the structural assumptions above
belongs to the abstract system class of Section~\ref{sec:baseline}. In
particular, this includes Fourier--Galerkin truncations of the two-dimensional
Navier--Stokes equations on periodic domains, as well as divergence-free
Galerkin reductions preserving the incompressible energy identity. Once a
feedback operator \(I_h\) satisfying Assumption~\ref{ass:H3} is specified, the
full nudging framework applies to these finite-dimensional Galerkin-type
models.
\end{proof}

\section{Proofs for the exact-model AOT baseline}
\label{app:baseline}

\begin{lemma}[Lower bound for the interpolant feedback]
\label{lem:ih-lower}
Suppose Assumption~\ref{ass:H3} holds. Then, for every \(z\in\R^d\),
\begin{equation}\label{eq:ih-lower}
\ip{I_h z}{z}_H
\ge
\frac12\,\normH{z}^2
-
\frac{c_0^2 h^2}{2}\,\normV{z}^2.
\end{equation}
\end{lemma}

\begin{proof}
Using the identity
\[
I_h z = z-(z-I_h z),
\]
we write
\[
\ip{I_h z}{z}_H
=
\ip{z}{z}_H-\ip{z-I_h z}{z}_H
=
\normH{z}^2-\ip{z-I_h z}{z}_H.
\]
By Cauchy--Schwarz and the approximation property \eqref{eq:interp},
\[
\bigl|\ip{z-I_h z}{z}_H\bigr|
\le
\normH{z-I_h z}\,\normH{z}
\le
c_0 h\,\normV{z}\,\normH{z}.
\]
Applying Young's inequality in the form
\[
ab\le \frac12 a^2+\frac12 b^2
\]
with \(a=c_0 h\,\normV{z}\) and \(b=\normH{z}\), we obtain
\[
c_0 h\,\normV{z}\,\normH{z}
\le
\frac{c_0^2 h^2}{2}\normV{z}^2+\frac12\normH{z}^2.
\]
Therefore,
\[
-\ip{z-I_h z}{z}_H
\ge
-\frac{c_0^2 h^2}{2}\normV{z}^2-\frac12\normH{z}^2.
\]
Substituting this estimate into the previous identity yields
\[
\ip{I_h z}{z}_H
\ge
\frac12\,\normH{z}^2-\frac{c_0^2 h^2}{2}\normV{z}^2,
\]
which is exactly \eqref{eq:ih-lower}.
\end{proof}

\begin{proof}[Proof of Proposition~\ref{prop:wp}]
We first consider the true system. Define
\[
F(z):=-\nu Az-N(z)+f.
\]
By Assumption~\ref{ass:H2}, the mapping \(N\) is locally Lipschitz on \(\R^d\).
Hence \(F\) is also locally Lipschitz on \(\R^d\). The Picard--Lindel\"of
theorem therefore yields, for every \(u_0\in\R^d\), a unique maximal solution
\[
u\in C^1([0,T_{\max}^u);\R^d)
\]
for some \(T_{\max}^u\in(0,\infty]\).

It remains to show that \(T_{\max}^u=\infty\). Set
\(y(t):=\normH{u(t)}^2.\)
Taking the \(H\)-inner product of \eqref{eq:truth} with \(u(t)\), we obtain
\[
\frac12 y'(t)+\nu\,u(t)^\top A u(t)
=
\ip{f}{u(t)}_H-\ip{N(u(t))}{u(t)}_H.
\]
By Assumption~\ref{ass:H2}\emph{(i)}, we have
\[
-\ip{N(u(t))}{u(t)}_H\le C_{\mathrm E}\,y(t).
\]
Dropping the nonnegative term \(\nu\,u(t)^\top A u(t)\), and estimating the
forcing term by Young's inequality,
\[
\ip{f}{u(t)}_H
\le
\frac12\normH{f}^2+\frac12 y(t),
\]
we obtain
\[
\frac12 y'(t)\le \frac12\normH{f}^2+\Bigl(C_{\mathrm E}+\frac12\Bigr)y(t),
\]
that is,
\[
y'(t)\le \normH{f}^2+(2C_{\mathrm E}+1)y(t),
\qquad t\in(0,T_{\max}^u).
\]
Gronwall's inequality then shows that \(y(t)\) remains bounded on every compact
subinterval of \([0,T_{\max}^u)\). Since, in finite dimensions, a maximal
solution of a locally Lipschitz ODE can fail to extend only through finite-time
blow-up, no such blow-up can occur. Therefore \(T_{\max}^u=\infty\), and the
true solution is global.

We now turn to the exact-model nudged system. Since \(u\in C^1([0,\infty);\R^d)\), the map
\(t\mapsto I_hu(t)\) is continuous on \([0,\infty)\). Define
\[
G(t,z):=-\nu Az-N(z)+f-\mu\bigl(I_hz-I_hu(t)\bigr).
\]
For each fixed \(t\), the map \(z\mapsto G(t,z)\) is locally Lipschitz on
\(\R^d\), uniformly for \(t\) in bounded intervals, and \(G\) is continuous in
\(t\). Standard existence and uniqueness theory for nonautonomous ODEs
therefore yields, for every \(v_0\in\R^d\), a unique maximal solution
\[
v\in C^1([0,T_{\max}^v);\R^d)
\]
for some \(T_{\max}^v\in(0,\infty]\).

To prove global existence, fix \(T>0\). Since \(u\) is continuous on \([0,T]\),
we could define
\[
M_T:=\sup_{t\in[0,T]}\normH{u(t)}<\infty.
\]
Let
\(q(t):=\normH{v(t)}^2.\)
Taking the \(H\)-inner product of \eqref{eq:aot} with \(v(t)\), we obtain
\[
\frac12 q'(t)+\nu\,v(t)^\top A v(t)
=
\ip{f}{v(t)}_H-\ip{N(v(t))}{v(t)}_H
-\mu\,\ip{I_hv(t)}{v(t)}_H
+\mu\,\ip{I_hu(t)}{v(t)}_H.
\]
We estimate the terms on the right-hand side separately. By
Assumption~\ref{ass:H2}\emph{(i)},
\[
-\ip{N(v(t))}{v(t)}_H\le C_{\mathrm E}q(t).
\]
According to Remark~\ref{rem:Ih-bounded}, we have
\[
-\mu\,\ip{I_hv(t)}{v(t)}_H
\le
\mu\,\normH{I_hv(t)}\,\normH{v(t)}
\le
\mu C_I\,q(t).
\]
Moreover,
\[
\ip{f}{v(t)}_H
\le
\frac12\normH{f}^2+\frac12 q(t),
\]
and similarly,
\[
\mu\,\ip{I_hu(t)}{v(t)}_H
\le
\mu\,\normH{I_hu(t)}\,\normH{v(t)}
\le
\frac{\mu^2}{2}\normH{I_hu(t)}^2+\frac12 q(t)
\le
\frac{\mu^2C_I^2}{2}M_T^2+\frac12 q(t).
\]
Dropping again the nonnegative dissipation term \(\nu\,v(t)^\top A v(t)\), we
find
\[
\frac12 q'(t)
\le
\frac12\normH{f}^2+\frac{\mu^2C_I^2}{2}M_T^2
+\bigl(C_{\mathrm E}+\mu C_I+1\bigr)q(t),
\qquad
t\in(0,T\wedge T_{\max}^v).
\]
Thus there exist constants \(a_T,b_T>0\), depending only on \(T\), the system
parameters, and the already constructed true solution \(u\), such that
\[
q'(t)\le a_T+b_T q(t),
\qquad
t\in(0,T\wedge T_{\max}^v).
\]
Gronwall's inequality implies that \(q(t)\) remains bounded on
\([0,T\wedge T_{\max}^v)\). Since \(T>0\) is arbitrary, finite-time blow-up is
impossible, and therefore \(T_{\max}^v=\infty\). Hence the nudged solution is
global as well.
\end{proof}

\begin{proof}[Proof of Proposition~\ref{prop:dissipativity}]
We begin with the true system. Define
\(y(t):=\normH{u(t)}^2.\)
Taking the \(H\)-inner product of \eqref{eq:truth} with \(u(t)\), we obtain
\[
\frac12 y'(t)+\nu\,u(t)^\top A u(t)+\ip{N(u(t))}{u(t)}_H
=
\ip{f}{u(t)}_H.
\]
By Assumption~\ref{ass:H4},
\[
\frac12 y'(t)+\alpha y(t)\le \beta+\ip{f}{u(t)}_H.
\]
We then estimate the forcing term by Young's inequality,
\[
\ip{f}{u(t)}_H
\le
\frac{\alpha}{2}\,y(t)+\frac{1}{2\alpha}\normH{f}^2.
\]
It follows that
\[
\frac12 y'(t)+\frac{\alpha}{2}y(t)
\le
\beta+\frac{1}{2\alpha}\normH{f}^2.
\]
Equivalently,
\[
y'(t)+\alpha y(t)
\le
2\beta+\frac{1}{\alpha}\normH{f}^2.
\]
Applying the integrating-factor method yields
\[
y(t)
\le
e^{-\alpha t}y(0)
+
\frac{2\beta+\normH{f}^2/\alpha}{\alpha}\bigl(1-e^{-\alpha t}\bigr),
\qquad
t\ge0,
\]
which is exactly \eqref{eq:H-abs-bound}.

We now consider the nudged system. By the truth estimate already established,
\[
\normH{u(t)}^2
\le
R_u^2
:=
\max\left\{
\normH{u_0}^2,\,
\frac{2\beta+\normH{f}^2/\alpha}{\alpha}
\right\},
\qquad
t\ge0.
\]
Let
\(y(t):=\normH{v(t)}^2.\)
Taking the \(H\)-inner product of \eqref{eq:aot} with \(v(t)\), we then obtain
\[
\frac12 y'(t)+\nu\,v(t)^\top A v(t)
=
\ip{f}{v(t)}_H-\ip{N(v(t))}{v(t)}_H
-\mu\,\ip{I_hv(t)}{v(t)}_H
+\mu\,\ip{I_hu(t)}{v(t)}_H.
\]
By \eqref{eq:onesided}, we have
\[
-\ip{N(v(t))}{v(t)}_H\le C_{\mathrm E}y(t).
\]
By Lemma~\ref{lem:ih-lower},
\[
-\mu\,\ip{I_hv(t)}{v(t)}_H
\le
-\frac{\mu}{2}\normH{v(t)}^2+\frac{\mu c_0^2h^2}{2}\normV{v(t)}^2.
\]
Substituting these estimates into the energy identity gives
\[
\frac12 y'(t)
+
\Bigl(\nu-\frac{\mu c_0^2h^2}{2}\Bigr)\normV{v(t)}^2
+
\frac{\mu}{2}\normH{v(t)}^2
\le
\ip{f}{v(t)}_H
+
C_{\mathrm E}y(t)
+
\mu\,\ip{I_hu(t)}{v(t)}_H.
\]
Since \(\mu c_0^2h^2<\nu\), we have
\[
\nu-\frac{\mu c_0^2h^2}{2}\ge \frac{\nu}{2}.
\]
Using coercivity, we further obtain
\[
\Bigl(\nu-\frac{\mu c_0^2h^2}{2}\Bigr)\normV{v(t)}^2
\ge
\frac{\nu}{2}\normV{v(t)}^2
\ge
\frac{\nu\lambda_1}{2}\normH{v(t)}^2.
\]
Therefore,
\[
\frac12 y'(t)
+
\Bigl(\frac{\nu\lambda_1}{2}+\frac{\mu}{2}-C_{\mathrm E}\Bigr)y(t)
\le
\ip{f}{v(t)}_H+\mu\,\ip{I_hu(t)}{v(t)}_H.
\]
Recalling the definition of \(\delta\), we rewrite this as
\[
\frac12 y'(t)+\delta y(t)
\le
\ip{f}{v(t)}_H+\mu\,\ip{I_hu(t)}{v(t)}_H.
\]

By \eqref{eq:stab},
\[
\normH{I_hu(t)}\le C_I R_u,
\qquad t\ge0.
\]
We now estimate the two terms on the right-hand side. By Young's inequality,
\[
\ip{f}{v(t)}_H
\le
\frac{1}{\delta}\normH{f}^2+\frac{\delta}{4}\normH{v(t)}^2,
\]
and similarly,
\[
\mu\,\ip{I_hu(t)}{v(t)}_H
\le
\frac{\mu^2}{\delta}\normH{I_hu(t)}^2+\frac{\delta}{4}\normH{v(t)}^2
\le
\frac{\mu^2 C_I^2}{\delta}R_u^2+\frac{\delta}{4}\normH{v(t)}^2.
\]
Substituting these bounds into the previous differential inequality yields
\[
\frac12 y'(t)+\delta y(t)
\le
\frac{1}{\delta}\normH{f}^2
+
\frac{\mu^2 C_I^2}{\delta}R_u^2
+
\frac{\delta}{2}y(t).
\]
Rearranging, we obtain
\[
y'(t)+\delta y(t)
\le
\frac{2}{\delta}\normH{f}^2
+
\frac{2\mu^2 C_I^2}{\delta}R_u^2
=:G.
\]
Applying the integrating-factor method gives
\[
y(t)\le e^{-\delta t}y(0)+\frac{G}{\delta}\bigl(1-e^{-\delta t}\bigr),
\qquad t\ge0.
\]
Hence
\[
\normH{v(t)}^2
\le
\max\Bigl\{\normH{v_0}^2,\frac{G}{\delta}\Bigr\},
\qquad t\ge0.
\]
In particular, the conclusion of the proposition holds with
\[
T_H:=0,
\qquad
R_H:=\max\Bigl\{\normH{v_0},\sqrt{\frac{G}{\delta}}\Bigr\}.
\]
\end{proof}

\begin{proof}[Proof of Corollary~\ref{cor:common-absorbing}]
By Proposition~\ref{prop:dissipativity}, the true solution satisfies
\[
\normH{u(t)}^2
\le
\max\left\{
\normH{u_0}^2,\,
\frac{2\beta+\normH{f}^2/\alpha}{\alpha}
\right\},
\qquad t\ge0.
\]
Hence, defining
\[
R_u:=\max\left\{
\normH{u_0},\,
\sqrt{\frac{2\beta+\normH{f}^2/\alpha}{\alpha}}
\right\},
\]
we have
\[
\sup_{t\ge0}\normH{u(t)}\le R_u.
\]
Again by Proposition~\ref{prop:dissipativity}, there exist constants
\(T_H\ge0\) and \(R_H>0\) such that
\[
\sup_{t\ge T_H}\normH{v(t)}\le R_H.
\]
Set
\[
T_\ast:=T_H,
\qquad
\widetilde R_\ast:=\max\{R_u,R_H\}.
\]
Then, for every \(t\ge T_\ast\),
\[
\normH{u(t)}\le \widetilde R_\ast,
\qquad
\normH{v(t)}\le \widetilde R_\ast.
\]
This proves \eqref{eq:common-absorbing-HV} with \(R_\ast\) replaced by
\(\widetilde R_\ast\).

By norm equivalence \eqref{eq:norm-equivalence},
\[
\normV{z}\le \sqrt{\lambda_{\max}}\,\normH{z},
\qquad z\in\R^d.
\]
Hence, for every \(t\ge T_\ast\),
\[
\normV{u(t)}\le \sqrt{\lambda_{\max}}\,\widetilde R_\ast,
\qquad
\normV{v(t)}\le \sqrt{\lambda_{\max}}\,\widetilde R_\ast.
\]
Therefore, if we define
\[
R_\ast:=\max\{1,\sqrt{\lambda_{\max}}\}\,\widetilde R_\ast,
\]
then both \eqref{eq:common-absorbing-HV} and
\eqref{eq:common-absorbing-V} follow.
\end{proof}

\begin{proof}[Proof of Theorem~\ref{thm:conv}]
Let
\(w:=v-u.\)
Then, subtracting \eqref{eq:truth} from \eqref{eq:aot}, we obtain
\[
\dot w+\nu Aw+\bigl(N(v)-N(u)\bigr)+\mu I_hw=0.
\]
Taking the \(H\)-inner product with \(w(t)\), we obtain
\begin{equation}\label{eq:w-energy}
\frac12\frac{d}{dt}\normH{w(t)}^2
+\nu\normV{w(t)}^2
+\mu\,\ip{I_hw(t)}{w(t)}_H
=
-\ip{N(v(t))-N(u(t))}{w(t)}_H.
\end{equation}
By Lemma~\ref{lem:ih-lower},
\[
\mu\,\ip{I_hw(t)}{w(t)}_H
\ge
\frac{\mu}{2}\normH{w(t)}^2
-
\frac{\mu c_0^2 h^2}{2}\normV{w(t)}^2.
\]
Substituting this bound into \eqref{eq:w-energy} yields
\[
\frac12\frac{d}{dt}\normH{w(t)}^2
+
\Bigl(\nu-\frac{\mu c_0^2 h^2}{2}\Bigr)\normV{w(t)}^2
+
\frac{\mu}{2}\normH{w(t)}^2
\le
-\ip{N(v(t))-N(u(t))}{w(t)}_H.
\]
For \(t\ge T_\ast\), Corollary~\ref{cor:common-absorbing} implies
\[
\normH{u(t)}\le R_\ast,
\qquad
\normH{v(t)}\le R_\ast.
\]
Hence, by Assumption~\ref{ass:H2}\emph{(ii)} with \(R=R_\ast\),
\[
-\ip{N(v(t))-N(u(t))}{w(t)}_H
\le
C_\ast\,\normH{w(t)}^2.
\]
Therefore
\[
\frac12\frac{d}{dt}\normH{w(t)}^2
+
\Bigl(\nu-\frac{\mu c_0^2 h^2}{2}\Bigr)\normV{w(t)}^2
+
\Bigl(\frac{\mu}{2}-C_\ast\Bigr)\normH{w(t)}^2
\le 0.
\]
Under the condition \(\mu c_0^2 h^2<\nu\), the \(V\)-term is nonnegative and thus can
be dropped, giving
\[
\frac12\frac{d}{dt}\normH{w(t)}^2
+
\Bigl(\frac{\mu}{2}-C_\ast\Bigr)\normH{w(t)}^2
\le 0,
\qquad
t\ge T_\ast.
\]
Since \(\mu>2C_\ast\), Gronwall's inequality yields
\[
\normH{w(t)}^2
\le
\exp\bigl(-(\mu-2C_\ast)(t-T_\ast)\bigr)\,
\normH{w(T_\ast)}^2,
\qquad t\ge T_\ast,
\]
which is exactly \eqref{eq:exp-conv}.
\end{proof}

\section{Proofs for Subsection~\ref{subsec:baseline-noisy}}
\label{app:baseline-noisy}

\begin{proof}[Proof of Proposition~\ref{prop:stoch-wp-diss}]
We begin by deriving an a priori \(H\)-moment estimate for the stochastic
nudged system \eqref{eq:aot-noisy}. 


Let
\[
b(t,z):=f-\nu Az-N(z)-\mu\bigl(I_hz-I_hu(t)\bigr).
\]
Since \(u\in C^1([0,\infty);\R^d)\) is deterministic and \(N\) is locally
Lipschitz, the drift \(b(t,\cdot)\) is locally Lipschitz uniformly on compact
time intervals. The diffusion coefficient \(\mu\Gamma_h\) is constant.
Therefore, by the standard finite-dimensional local well-posedness theorem for
SDEs with locally Lipschitz coefficients, \eqref{eq:aot-noisy} admits a unique
maximal strong solution up to its maximal existence time.

Fix a time horizon \(T>0\). For \(R>\normH{v_0}\), define the stopping time
\[
\tau_R:=\inf\{t\ge0:\normH{v(t)}\ge R\},
\]
and set
\[
\theta_R:=\min\{T,\tau_R\}.
\]
We first carry out the energy estimate on the stopped interval
\([0,\theta_R]\), where the stopped solution remains bounded by \(R\).
This localization justifies the estimate before global existence is known.
We derive below a bound with constants independent of \(R\), which allows us
to send \(R\to\infty\) and rule out finite-time explosion.

Applying It\^o's formula to \(\normH{v(t)}^2\) on \([0,\theta_R]\), we obtain
\begin{equation}
    \begin{aligned}
d\normH{v(t)}^2
&=
\Bigl(
-2\nu \normV{v(t)}^2
-2\ip{N(v(t))}{v(t)}_H
-2\mu \ip{I_hv(t)}{v(t)}_H
\notag\\
&\qquad
+2\mu \ip{I_hu(t)}{v(t)}_H
+2\ip{f}{v(t)}_H
+\mu^2 \Tr(\Gamma_h\Gamma_h^\top)
\Bigr)\,dt
+
2\mu \ip{v(t)}{\Gamma_h\,dW_t}_H .
\label{eq:proof-stoch-ito}
\end{aligned}
\end{equation}

By Assumption~\ref{ass:H2}\emph{(i)}, we have
\[
-\ip{N(z)}{z}_H \le C_{\mathrm E}\normH{z}^2.
\]
Moreover, Assumption~\ref{ass:H3} gives
\[
\ip{I_hz}{z}_H
=
\normH{z}^2-\ip{z-I_hz}{z}_H
\ge
\normH{z}^2-c_0h\normV{z}\normH{z}.
\]
Using Young's inequality,
\[
c_0h\normV{z}\normH{z}
\le
\frac12 c_0^2h^2 \normV{z}^2 + \frac12 \normH{z}^2,
\]
and hence
\begin{equation}\label{eq:proof-Ih-lower-v}
\ip{I_hz}{z}_H
\ge
\frac12 \normH{z}^2 - \frac12 c_0^2h^2 \normV{z}^2.
\end{equation}
Substituting these bounds into \eqref{eq:proof-stoch-ito}, we infer that
\begin{align}
d\normH{v(t)}^2
&\le
\Bigl(
-(2\nu-\mu c_0^2h^2)\normV{v(t)}^2
-(\mu-2C_{\mathrm E})\normH{v(t)}^2
\notag\\
&\qquad
+2\mu \ip{I_hu(t)}{v(t)}_H
+2\ip{f}{v(t)}_H
+\mu^2 \Tr(\Gamma_h\Gamma_h^\top)
\Bigr)\,dt
+
2\mu \ip{v(t)}{\Gamma_h\,dW_t}_H .
\label{eq:proof-stoch-pre}
\end{align}

Since \(\mu c_0^2h^2<\nu\), we have
\(2\nu-\mu c_0^2h^2>\nu\). Thus
\[
-(2\nu-\mu c_0^2h^2)\normV{z}^2
\le
-\nu\normV{z}^2
\le
-\nu\lambda_1\normH{z}^2.
\]
Recalling that
\(\delta=\frac{\nu\lambda_1}{2}+\frac{\mu}{2}-C_{\mathrm E}\),
we obtain
\begin{equation}\label{eq:proof-main-damp-v}
-(2\nu-\mu c_0^2h^2)\normV{z}^2
-(\mu-2C_{\mathrm E})\normH{z}^2
\le
-2\delta \normH{z}^2.
\end{equation}

We next estimate the forcing and observation terms. By
Remark~\ref{rem:Ih-bounded},
\(\normH{I_hu(t)}\le C_I \normH{u(t)}.\)
Thus Young's inequality gives
\[
2\mu \ip{I_hu(t)}{v(t)}_H
\le
\frac{\delta}{2}\normH{v(t)}^2
+
\frac{2\mu^2 C_I^2}{\delta}\normH{u(t)}^2,
\]
and similarly,
\[
2\ip{f}{v(t)}_H
\le
\frac{\delta}{2}\normH{v(t)}^2
+
\frac{2}{\delta}\normH{f}^2.
\]
Combining these estimates with \eqref{eq:proof-stoch-pre} and
\eqref{eq:proof-main-damp-v}, we arrive at
\begin{align}
d\normH{v(t)}^2
&\le
\Bigl(
-\delta \normH{v(t)}^2
+
\frac{2}{\delta}\normH{f}^2
+
\frac{2\mu^2 C_I^2}{\delta}\normH{u(t)}^2
+
\mu^2 \Tr(\Gamma_h\Gamma_h^\top)
\Bigr)\,dt
\notag\\
&\qquad
+
2\mu \ip{v(t)}{\Gamma_h\,dW_t}_H .
\label{eq:proof-v-final}
\end{align}

We now use this stopped estimate to rule out explosion. Integrating \eqref{eq:sur-noisy-bdd-pre} over \([0,\theta_R]\) and taking expectations, the stochastic integral drops out by the martingale property. Dropping the nonpositive damping term, we obtain
\[
\mathbb E\normH{v(\theta_R)}^2
\le
\normH{v_0}^2
+
\int_0^T
\left(
\frac{2}{\delta}\normH{f}^2
+
\frac{2\mu^2 C_I^2}{\delta}\normH{u(t)}^2
+
\mu^2 \Tr(\Gamma_h\Gamma_h^\top)
\right)\,dt .
\]
Since \(u\) is continuous on \([0,T]\), the right-hand side is bounded by a
constant \(C_T\) independent of \(R\). By continuity of the sample paths,
\(\normH{v(\tau_R)}=R\) on the event \(\{\tau_R\le T\}\). Hence
\[
R^2\,\mathbb P(\tau_R\le T)
\le
\mathbb E\normH{v(\theta_R)}^2
\le C_T.
\]
Let \(\zeta\) denote the maximal existence time of the local solution. Since
the solution has continuous paths up to \(\zeta\), we have
\(\zeta=\lim_{R\to\infty}\tau_R .\)
Therefore
\[
\mathbb P(\zeta\le T)
\le
\lim_{R\to\infty}\mathbb P(\tau_R\le T)
=0.
\]
Since \(T>0\) was arbitrary, \(\zeta=\infty\) almost surely. This proves global
well-posedness.

Having established global existence, we repeat the above It\^o energy estimate
for the global solution. Taking expectations and using that the stochastic
integral has mean zero, we derive
\[
\frac{d}{dt}\,\mathbb E\normH{v(t)}^2
\le
-\delta\,\mathbb E\normH{v(t)}^2
+
\frac{2}{\delta}\normH{f}^2
+
\frac{2\mu^2 C_I^2}{\delta}\normH{u(t)}^2
+
\mu^2 \Tr(\Gamma_h\Gamma_h^\top).
\]
Applying Gronwall's inequality yields \eqref{eq:stoch-moment-bound}.

It remains to derive the asymptotic bound \eqref{eq:stoch-absorbing}. By
Proposition~\ref{prop:dissipativity},
\(\limsup_{t\to\infty}\normH{u(t)}^2 \le R_u^2.\)
Fix \(\varepsilon>0\). Then there exists \(T_\varepsilon\ge0\) such that
\[
\normH{u(t)}^2 \le R_u^2+\varepsilon,
\qquad
\forall t\ge T_\varepsilon.
\]
Applying the preceding differential inequality on \([T_\varepsilon,\infty)\),
we obtain, for all \(t\ge T_\varepsilon\),
\begin{align*}
\mathbb E\normH{v(t)}^2
&\le
e^{-\delta(t-T_\varepsilon)}
\mathbb E\normH{v(T_\varepsilon)}^2
\\
&\quad
+
\frac{1}{\delta}
\left(
\frac{2}{\delta}\normH{f}^2
+
\frac{2\mu^2 C_I^2}{\delta}(R_u^2+\varepsilon)
+
\mu^2 \Tr(\Gamma_h\Gamma_h^\top)
\right)
\Bigl(1-e^{-\delta(t-T_\varepsilon)}\Bigr).
\end{align*}
Letting \(t\to\infty\) and then \(\varepsilon\downarrow0\) proves
\eqref{eq:stoch-absorbing}. The final assertion follows immediately.
\end{proof}

\begin{proof}[Proof of Theorem~\ref{thm:stoch-tracking}]
We now turn to the tracking estimate. As in the deterministic case, we derive an evolution equation for the error \(w(t):=v(t)-u(t)\) and perform an \(H\)-energy estimate. The difference in the stochastic setting is that the error equation contains an additional noise term, which contributes through the It\^o correction term.

Subtracting \eqref{eq:truth} from \eqref{eq:aot-noisy}, we obtain
\begin{equation}\label{eq:proof-stoch-error}
dw(t)+\nu A w(t)\,dt+\bigl(N(v(t))-N(u(t))\bigr)\,dt
=
-\mu I_h w(t)\,dt+\mu \Gamma_h\,dW_t.
\end{equation}

Applying It\^o's formula to \(\normH{w(t)}^2\), we find
\begin{align}
d\normH{w(t)}^2
&=
\Bigl(
-2\nu \normV{w(t)}^2
-2\ip{N(v(t))-N(u(t))}{w(t)}_H
-2\mu \ip{I_hw(t)}{w(t)}_H
\notag\\
&\qquad\qquad
+
\mu^2 \Tr(\Gamma_h\Gamma_h^\top)
\Bigr)\,dt
+
2\mu \ip{w(t)}{\Gamma_h\,dW_t}_H.
\label{eq:proof-stoch-error-ito}
\end{align}

By Assumption~\ref{ass:H3}, the same estimate as in
\eqref{eq:proof-Ih-lower-v} gives
\[
\ip{I_hw}{w}_H
\ge
\frac12 \normH{w}^2
-
\frac12 c_0^2h^2 \normV{w}^2.
\]
Hence
\[
-2\mu \ip{I_hw}{w}_H
\le
-\mu \normH{w}^2 + \mu c_0^2h^2 \normV{w}^2.
\]
On the other hand, Assumption~\ref{ass:H2-global} yields
\[
-2\ip{N(v)-N(u)}{w}_H
\le
2C_{\rm gl}\normH{w}^2.
\]
Substituting these estimates into \eqref{eq:proof-stoch-error-ito}, we obtain
\begin{align}
d\normH{w(t)}^2
&\le
\Bigl(
-(2\nu-\mu c_0^2h^2)\normV{w(t)}^2
-(\mu-2C_{\rm gl})\normH{w(t)}^2
\notag\\
&\qquad\qquad
+
\mu^2 \Tr(\Gamma_h\Gamma_h^\top)
\Bigr)\,dt
+
2\mu \ip{w(t)}{\Gamma_h\,dW_t}_H.
\label{eq:proof-track-pre}
\end{align}

Since \(\mu c_0^2h^2<\nu\), the \(V\)-term on the right-hand side is
nonpositive and can be discarded. Therefore,
\[
d\normH{w(t)}^2
\le
\Bigl(
-(\mu-2C_{\rm gl})\normH{w(t)}^2
+
\mu^2 \Tr(\Gamma_h\Gamma_h^\top)
\Bigr)\,dt
+
2\mu \ip{w(t)}{\Gamma_h\,dW_t}_H.
\]
Taking expectations and using that the stochastic integral has mean zero, we
obtain
\[
\frac{d}{dt}\,\mathbb E\normH{w(t)}^2
\le
-(\mu-2C_{\rm gl})\,\mathbb E\normH{w(t)}^2
+
\mu^2 \Tr(\Gamma_h\Gamma_h^\top).
\]
Since \(\mu>2C_{\rm gl}\), Gronwall's inequality gives
\eqref{eq:stoch-tracking}. Letting \(t\to\infty\) proves
\eqref{eq:stoch-floor}.
\end{proof}

\section{Proofs for the surrogate extension and post-absorption region}
\label{app:sur-extension}

\begin{proof}[Proof of Proposition~\ref{prop:sur-wp-cutoff}]
Let
\[
 \mathcal K:=\{z\in\R^d:\normV{z}\le R_{\mathrm{ext}}^{+}\}.
\]
By assumption, \(\widehat F_M\) is locally Lipschitz on an open neighborhood
\(U\supset  \mathcal K \nc\). Since \( \mathcal K \) is compact, \(\widehat F_M\) is bounded and
Lipschitz on \( \mathcal K \). Moreover, by construction,
\[
\chi(z)=0
\qquad\text{whenever }\normV{z}\ge R_{\mathrm{ext}}^{+},
\]
so \(\operatorname{supp}\chi\subset  \mathcal K \Subset U\). Define
\[
G(z):=\chi(z)\widehat F_M(z),
\qquad z\in U.
\]
Then we can naturally extend \(G\) to \(\R^d\) by setting \(G(z)=0\) for
\(z\notin U\). Since \(\chi\) vanishes in a neighborhood of
\(\R^d\setminus U\), this extension remains locally Lipschitz on \(\R^d\).
Given that \(G\) vanishes outside \(\operatorname{supp}\chi\), and that
\(\widehat F_M\) is bounded on \( \mathcal K \), there exists \(M_G>0\) such that
\begin{equation}\label{eq:proof-sur-wp-G-bound}
\normH{G(z)}\le M_G,
\qquad \forall z\in\R^d.
\end{equation}

Next, write
\[
F_M(z)=G(z)+\bigl(1-\chi(z)\bigr)\bigl(f-Bz\bigr).
\]
The first term is globally defined and locally Lipschitz. The second term is
the product of a smooth bounded cutoff and an affine map. Hence \(F_M\) is
locally Lipschitz on \(\R^d\). Moreover, using
\eqref{eq:proof-sur-wp-G-bound}, \(0\le 1-\chi\le 1\), and the fact that \(B\)
is fixed, we obtain the linear-growth bound
\begin{equation}\label{eq:proof-sur-wp-FM-growth}
\normH{F_M(z)}
\le
C_M\bigl(1+\normH{z}\bigr),
\qquad \forall z\in\R^d,
\end{equation}
for some constant \(C_M>0\).

Now fix an arbitrary initial pair \((u_0,v_0)\in\R^d\times\R^d\). By
Proposition~\ref{prop:wp}, the true system \eqref{eq:sur-truth} admits a
unique global solution
\(u\in C^1([0,\infty);\R^d)\),
with \(u(0)=u_0\). We fix this deterministic trajectory throughout the rest of
the proof. The surrogate nudged equation \eqref{eq:sur-aot} can then be written
as the nonautonomous ODE
\[
\dot v=b(t,v),
\qquad
b(t,z):=F_M(z)-\mu I_h z+\mu I_h u(t).
\]
Since \(u\in C([0,\infty);\R^d)\) and \(I_h\) is linear, the map
\(t\mapsto I_hu(t)\) is continuous. Together with the local Lipschitz property
of \(F_M\), this implies that \(b\) is continuous in \(t\) and locally
Lipschitz in \(z\), uniformly on compact time intervals. Therefore, the
standard finite-dimensional ODE theorem gives a unique maximal solution
\[
v\in C^1([0,T_{\max});\R^d)
\]
for some \(T_{\max}\in(0,\infty]\).

It remains to show that \(T_{\max}=\infty\). Fix \(T>0\). Since \(u\) is
continuous on \([0,T]\),
\[
M_u(T):=\sup_{0\le t\le T}\normH{u(t)}<\infty .
\]
Using \eqref{eq:proof-sur-wp-FM-growth} and the \(H\)-boundedness of \(I_h\),
we have, for \(0\le t\le T\),
\[
\normH{b(t,z)}
\le
C_M\bigl(1+\normH{z}\bigr)
+\mu C_I\normH{z}
+\mu C_I M_u(T)
\le
a_T+b_T\normH{z}
\]
for suitable constants \(a_T,b_T>0\). Hence, along the maximal solution,
\[
\frac{d}{dt}\normH{v(t)}
\le
a_T+b_T\normH{v(t)}
\qquad\text{for a.e. }t\in[0,\min\{T,T_{\max}\}).
\]
By Gr\"onwall's inequality, \(\normH{v(t)}\) remains bounded on every bounded
time interval. Therefore no finite-time blow-up can occur, and the maximal
solution extends globally. Thus
\[
v\in C^1([0,\infty);\R^d).
\]
Uniqueness follows from the local Lipschitz property of \(b(t,\cdot)\) on bounded sets.
\end{proof}

\begin{proof}[Proof of Proposition~\ref{prop:sur-absorb}]
Fix an arbitrary initial pair \((u_0,v_0)\in\R^d\times\R^d\), and let
\(u\) and \(v\) denote the corresponding global solutions of
\eqref{eq:sur-truth} and \eqref{eq:sur-aot}. Define the cutoff correction
\[
G_M(z):=F_M(z)-\bigl(f-Bz\bigr)
=\chi(z)\bigl(\widehat F_M(z)-f+Bz\bigr).
\]
Since the support of \(\chi\) is contained in
\[
\mathcal{K} \nc :=\{z\in\R^d:\normV{z}\le R_{\mathrm{ext}}^{+}\},
\]
and \(\widehat F_M\) is locally bounded on an open neighborhood of \(  \mathcal{K} \nc \), the
map \(G_M\) is bounded on \(\R^d\): there exists \(M_G>0\) such that
\begin{equation}\label{eq:proof-sur-absorb-GM}
\normH{G_M(z)}\le M_G,
\qquad \forall z\in\R^d.
\end{equation}
Hence the surrogate equation may be written as
\[
\dot v
=
f-Bv+G_M(v)-\mu I_hv+\mu I_hu.
\]

Taking the \(H\)-inner product with \(v\), we obtain
\[
\frac12\frac{d}{dt}\normH{v}^2
=
\ip{f}{v}_H-\ip{Bv}{v}_H+\ip{G_M(v)}{v}_H
-\mu\ip{I_hv}{v}_H+\mu\ip{I_hu}{v}_H.
\]
Using the symmetry and positive definiteness of \(B\), the boundedness of
\(G_M\), and Assumption~\ref{ass:H3}\emph{(ii)}, we infer
\[
\ip{Bv}{v}_H\ge \lambda_B\normH{v}^2,
\qquad
\bigl|\ip{G_M(v)}{v}_H\bigr|\le M_G\normH{v},
\]
and
\[
-\mu\ip{I_hv}{v}_H
\le
\mu\normH{I_hv}\,\normH{v}
\le
\mu C_I\normH{v}^2,
\]
\[
\mu\ip{I_hu}{v}_H
\le
\mu\normH{I_hu}\,\normH{v}
\le
\mu C_I\normH{u}\,\normH{v}.
\]
Therefore
\begin{equation}\label{eq:proof-sur-absorb-energy}
\frac12\frac{d}{dt}\normH{v}^2
\le
-(\lambda_B-\mu C_I)\normH{v}^2
+\bigl(\normH{f}+M_G+\mu C_I\normH{u}\bigr)\normH{v}.
\end{equation}

By Proposition~\ref{prop:dissipativity}, the true solution \(u\) is
eventually bounded in \(H\): there exist constants \(T_u\ge0\) and
\(R_u^H>0\), depending only on the system parameters and the initial truth
state, such that
\[
\sup_{t\ge T_u}\normH{u(t)}\le R_u^H.
\]
Hence, for \(t\ge T_u\), \eqref{eq:proof-sur-absorb-energy} gives
\[
\frac12\frac{d}{dt}\normH{v}^2
\le
-\rho_B\normH{v}^2 + C_u\normH{v},
\]
where
\[
\rho_B:=\lambda_B-\mu C_I>0,
\qquad
C_u:=\normH{f}+M_G+\mu C_I R_u^H.
\]
Applying Young's inequality,
\[
C_u\normH{v}
\le
\frac{\rho_B}{2}\normH{v}^2+\frac{C_u^2}{2\rho_B},
\]
we arrive at
\[
\frac{d}{dt}\normH{v}^2
\le
-\rho_B\normH{v}^2+\frac{C_u^2}{\rho_B},
\qquad t\ge T_u.
\]
A Gr\"onwall argument yields
\[
\normH{v(t)}^2
\le
e^{-\rho_B(t-T_u)}\normH{v(T_u)}^2
+\frac{C_u^2}{\rho_B^2}\bigl(1-e^{-\rho_B(t-T_u)}\bigr),
\qquad t\ge T_u.
\]
Therefore
\[
\sup_{t\ge T_u}\normH{v(t)}
\le
\max\left\{\normH{v(T_u)},\frac{C_u}{\rho_B}\right\}
=:R_M^H.
\]
Thus \eqref{eq:sur-H-absorb} holds with \(T_M:=T_u\).

Finally, by norm equivalence,
\[
\normV{v(t)}\le \sqrt{\lambda_{\max}}\normH{v(t)},
\qquad t\ge T_M.
\]
Hence \eqref{eq:sur-V-absorb} follows with
\(R_M:=\sqrt{\lambda_{\max}}\,R_M^H.\)
\end{proof}

\begin{proof}[Proof of Corollary~\ref{cor:sur-common-ball}]
Fix an arbitrary initial pair \((u_0,v_0)\in\R^d\times\R^d\), and let
\(u\) and \(v\) denote the corresponding solutions of
\eqref{eq:sur-truth} and \eqref{eq:sur-aot}.

By Proposition~\ref{prop:dissipativity}, there exist constants \(T_u\ge0\) and
\(R_u^H>0\) such that
\[
\sup_{t\ge T_u}\normH{u(t)}\le R_u^H.
\]
Using the norm equivalence \eqref{eq:norm-equivalence}, we obtain
\[
\sup_{t\ge T_u}\normV{u(t)}
\le
\sqrt{\lambda_{\max}}\,R_u^H
=:R_u.
\]

On the other hand, Proposition~\ref{prop:sur-absorb} yields constants
\(T_v\ge0\) and \(R_v>0\) such that
\[
\sup_{t\ge T_v}\normV{v(t)}\le R_v.
\]

Now set
\[
T_\ast:=\max\{T_u,T_v\},
\qquad
R_\ast:=\max\{R_u,R_v\}.
\]
Then, for every \(t\ge T_\ast\),
\[
\normV{u(t)}\le R_\ast,
\qquad
\normV{v(t)}\le R_\ast,
\]
which is exactly \eqref{eq:sur-common-ball}.
\end{proof}

\begin{proof}[Proof of Lemma~\ref{lem:ih-lower-finitedim}]
We write
\[
\ip{I_hz}{z}_H
=
\ip{z}{z}_H-\ip{z-I_hz}{z}_H
=
\normH{z}^2-\ip{z-I_hz}{z}_H.
\]
By Cauchy--Schwarz and Assumption~\ref{ass:H3}\emph{(i)},
\[
-\ip{z-I_hz}{z}_H
\ge
-\normH{z-I_hz}\,\normH{z}
\ge
-c_0h\,\normV{z}\,\normH{z}.
\]
Hence
\[
\ip{I_hz}{z}_H
\ge
\normH{z}^2-c_0h\,\normV{z}\,\normH{z}.
\]
Applying Young's inequality in the form
\[
ab\le \frac12 a^2+\frac12 b^2
\]
with \(a=\normH{z}\) and \(b=c_0h\,\normV{z}\), we obtain
\[
c_0h\,\normV{z}\,\normH{z}
\le
\frac12\normH{z}^2+\frac{c_0^2h^2}{2}\normV{z}^2.
\]
Substituting this into the previous inequality yields
\[
\ip{I_hz}{z}_H
\ge
\frac12\,\normH{z}^2-\frac{c_0^2h^2}{2}\normV{z}^2,
\]
which proves \eqref{eq:ih-lower-finitedim}.
\end{proof}

\begin{proof}[Proof of Lemma~\ref{lem:sur-squeeze-F}]
Let \(z,z'\in\cB_\ast\). Since \(\normV{z}\le R_\ast\) and \(\normV{z'}\le R_\ast\),
the norm equivalence \eqref{eq:norm-equivalence} implies
\[
\normH{z}\le \frac{R_\ast}{\sqrt{\lambda_1}}=R_H^\ast,
\qquad
\normH{z'}\le \frac{R_\ast}{\sqrt{\lambda_1}}=R_H^\ast.
\]
Hence Assumption~\ref{ass:H2}\emph{(ii)} applies with radius \(R_H^\ast\), so
\[
-\ip{N(z')-N(z)}{z'-z}_H
\le
C_{R_H^\ast}\normH{z'-z}^2.
\]

Now recall that
\[
F(z)=f-\nu Az-N(z).
\]
Therefore
\[
F(z')-F(z)
=
-\nu A(z'-z)-\bigl(N(z')-N(z)\bigr),
\]
and thus
\[
\ip{F(z')-F(z)}{z'-z}_H
=
-\nu\ip{A(z'-z)}{z'-z}_H
-\ip{N(z')-N(z)}{z'-z}_H.
\]
Since \(\ip{A(z'-z)}{z'-z}_H=\normV{z'-z}^2\), we obtain
\[
\ip{F(z')-F(z)}{z'-z}_H
=
-\nu\normV{z'-z}^2
-\ip{N(z')-N(z)}{z'-z}_H
\le
-\nu\normV{z'-z}^2
+
C_{R_H^\ast}\normH{z'-z}^2.
\]
Setting
\(C_{\mathrm{sq}}:=C_{R_H^\ast}\)
gives \eqref{eq:sur-squeeze-F}.
\end{proof}

\begin{proof}[Proof of Lemma~\ref{lem:sur-squeeze-FM}]
Write \(F_M=F+r_M\). By Lemma~\ref{lem:sur-squeeze-F}, we have
\[
\ip{F(z')-F(z)}{z'-z}_H
\le
-\nu\,\normV{z'-z}^2
+
C_{\mathrm{sq}}\,\normH{z'-z}^2.
\]
Moreover, by Cauchy--Schwarz and the definition of \(\ell_M\) in
\eqref{eq:sur-residual-def},
\[
\ip{r_M(z')-r_M(z)}{z'-z}_H
\le
\normH{r_M(z')-r_M(z)}\,\normH{z'-z}
\le
\ell_M\,\normH{z'-z}^2.
\]
Adding the two estimates gives \eqref{eq:sur-squeeze-FM}.
\end{proof}

\section{Proofs for Subsection~\ref{subsec:sur-noisy}}
\label{app:sur-noisy}

\begin{proof}[Proof of Proposition~\ref{prop:sur-wp-noisy-bdd}]
We use the same stopping argument as in
Proposition~\ref{prop:stoch-wp-diss}, but the dissipative estimate now comes
from the cutoff-extended surrogate drift.

Set
\[
b(t,z):=F_M(z)-\mu\bigl(I_hz-I_hu(t)\bigr).
\]
By Proposition~\ref{prop:sur-wp-cutoff}, \(F_M\) is locally Lipschitz on
\(\R^d\). Since \(u\in C^1([0,\infty);\R^d)\) is deterministic and \(I_h\) is
linear, \(b(t,\cdot)\) is locally Lipschitz uniformly on compact time
intervals. Moreover, the diffusion coefficient is the constant matrix \(\mu\Gamma_h\).
Thus the standard finite-dimensional SDE theorem gives a unique maximal strong
solution.

Fix \(T>0\). For \(R>\normH{v_0}\), we define
\[
\tau_R:=\inf\{t\ge0:\normH{v(t)}\ge R\},
\qquad
\theta_R:=\min\{T,\tau_R\}.
\]
All estimates below are first understood on the stopped interval
\([0,\theta_R]\). Applying It\^o's formula to \(\normH{v(t)}^2\) gives
\begin{align}
d\normH{v(t)}^2
&=
\Bigl(
2\ip{F_M(v(t))}{v(t)}_H
-2\mu \ip{I_hv(t)}{v(t)}_H
+2\mu \ip{I_hu(t)}{v(t)}_H
\notag\\
&\qquad\qquad
+\mu^2\Tr(\Gamma_h\Gamma_h^\top)
\Bigr)\,dt
+
2\mu\ip{v(t)}{\Gamma_h\,dW_t}_H .
\label{eq:sur-noisy-bdd-ito}
\end{align}

By Assumption~\ref{ass:sur-global-diss},
\[
2\ip{F_M(z)}{z}_H
\le
2\beta_M-2\alpha_M\normH{z}^2.
\]
The boundedness of \(I_h\) gives
\[
-2\mu \ip{I_hz}{z}_H
\le
2\mu\normH{I_hz}\normH{z}
\le
2\mu C_I\normH{z}^2,
\]
and
\[
2\mu \ip{I_hu(t)}{v(t)}_H
\le
2\mu C_I\normH{u(t)}\normH{v(t)}.
\]
Let
\(\vartheta_M:=\alpha_M-\mu C_I>0.\)
Then Young's inequality implies
\[
2\mu C_I\normH{u(t)}\normH{v(t)}
\le
\vartheta_M\normH{v(t)}^2
+
\frac{\mu^2C_I^2}{\vartheta_M}\normH{u(t)}^2.
\]
Substituting these estimates into \eqref{eq:sur-noisy-bdd-ito}, we obtain
\begin{align}
d\normH{v(t)}^2
&\le
\Bigl(
-\vartheta_M\normH{v(t)}^2
+
2\beta_M
+
\frac{\mu^2C_I^2}{\vartheta_M}\normH{u(t)}^2
+
\mu^2\Tr(\Gamma_h\Gamma_h^\top)
\Bigr)\,dt
\notag\\
&\qquad
+
2\mu\ip{v(t)}{\Gamma_h\,dW_t}_H .
\label{eq:sur-noisy-bdd-pre}
\end{align}

After integrating up to \(\theta_R\) and taking expectations, the stochastic
integral vanishes by the martingale property. Dropping the nonpositive damping
term yields
\[
\mathbb E\normH{v(\theta_R)}^2
\le
\normH{v_0}^2
+
\int_0^T
\left(
2\beta_M
+
\frac{\mu^2C_I^2}{\vartheta_M}\normH{u(t)}^2
+
\mu^2\Tr(\Gamma_h\Gamma_h^\top)
\right)\,dt .
\]
Since \(u\) is continuous on \([0,T]\), the right-hand side is bounded by a
constant \(C_T\) independent of \(R\). On \(\{\tau_R\le T\}\), path continuity
gives \(\normH{v(\tau_R)}=R\), and hence
\[
R^2\,\mathbb P(\tau_R\le T)
\le
\mathbb E\normH{v(\theta_R)}^2
\le C_T.
\]
Letting \(R\to\infty\) rules out finite-time explosion on \([0,T]\). Since
\(T>0\) was arbitrary, the maximal strong solution is global almost surely.

For the global solution, the differential inequality
\eqref{eq:sur-noisy-bdd-pre} can now be averaged directly. Thus
\[
\frac{d}{dt}\,\mathbb E\normH{v(t)}^2
\le
-\vartheta_M\,\mathbb E\normH{v(t)}^2
+
2\beta_M
+
\frac{\mu^2C_I^2}{\vartheta_M}\normH{u(t)}^2
+
\mu^2\Tr(\Gamma_h\Gamma_h^\top).
\]
Gronwall's inequality gives \eqref{eq:sur-noisy-moment-bound}.

It remains to pass to the long-time bound. By
Proposition~\ref{prop:dissipativity},
\[
\limsup_{t\to\infty}\normH{u(t)}^2 \le R_u^2.
\]
Fix \(\varepsilon>0\). Then there exists \(T_\varepsilon\ge0\) such that
\[
\normH{u(t)}^2 \le R_u^2+\varepsilon,
\qquad
\forall t\ge T_\varepsilon.
\]
Using the preceding moment inequality on \([T_\varepsilon,\infty)\), we obtain,
for all \(t\ge T_\varepsilon\),
\begin{align*}
\mathbb E\normH{v(t)}^2
&\le
e^{-\vartheta_M(t-T_\varepsilon)}
\mathbb E\normH{v(T_\varepsilon)}^2
\\
&\quad
+
\frac{1}{\vartheta_M}
\left(
2\beta_M
+
\frac{\mu^2 C_I^2}{\vartheta_M}(R_u^2+\varepsilon)
+
\mu^2\Tr(\Gamma_h\Gamma_h^\top)
\right)
\Bigl(1-e^{-\vartheta_M(t-T_\varepsilon)}\Bigr).
\end{align*}
Taking \(t\to\infty\) and then \(\varepsilon\downarrow0\) proves
\eqref{eq:sur-noisy-moment-absorb}. The final assertion follows immediately.
\end{proof}

\begin{proof}[Proof of Theorem~\ref{thm:sur-track-noisy}]
We now derive the tracking estimate. The argument parallels that of
Theorem~\ref{thm:sur-track}, with the only new contribution arising from the
quadratic variation of the stochastic forcing.

Set
\(w:=v-u.\)
Then subtracting \eqref{eq:sur-truth} from \eqref{eq:sur-aot-noisy}, we obtain
\begin{equation}\label{eq:sur-track-noisy-error}
dw
=
\Bigl(
F_M(v)-F(u)-\mu I_hw
\Bigr)\,dt
+
\mu\Gamma_h\,dW_t.
\end{equation}
Applying It\^o's formula to \(\normH{w}^2\), we find
\begin{align}
d\normH{w}^2
&=
2\ip{F_M(v)-F(u)}{w}_H\,dt
-2\mu\,\ip{I_hw}{w}_H\,dt
+\mu^2\Tr(\Gamma_h\Gamma_h^\top)\,dt
\notag\\
&\qquad
+
2\mu \ip{w}{\Gamma_h\,dW_t}_H.
\label{eq:sur-track-noisy-ito}
\end{align}

We decompose
\[
\ip{F_M(v)-F(u)}{w}_H
=
\ip{F_M(v)-F_M(u)}{w}_H+\ip{r_M(u)}{w}_H.
\]
By Assumption~\ref{ass:sur-global-sq},
\begin{equation}\label{eq:sur-FM-sq-use-noisy}
\ip{F_M(v)-F_M(u)}{w}_H
\le
-\nu\,\normV{w}^2
+
C_{M,\mathrm{gl}}\normH{w}^2.
\end{equation}

For \(t\ge T_\ast\), Corollary~\ref{cor:sur-common-ball} implies that
\(u(t)\in\cB_\ast\). Hence, by Definition~\ref{def:sur-residual}, we have
\(\normH{r_M(u(t))}\le \delta_M.\)
Therefore, by Cauchy--Schwarz and the norm equivalence
\eqref{eq:norm-equivalence},
\begin{equation}\label{eq:sur-resid-term-noisy}
\ip{r_M(u)}{w}_H
\le
\normH{r_M(u)}\,\normH{w}
\le
\delta_M\,\normH{w}
\le
\lambda_1^{-1/2}\delta_M\,\normV{w}.
\end{equation}
Finally, Lemma~\ref{lem:ih-lower-finitedim} yields
\begin{equation}\label{eq:sur-ih-use-noisy}
-\mu\,\ip{I_hw}{w}_H
\le
-\frac{\mu}{2}\normH{w}^2
+
\frac{\mu c_0^2h^2}{2}\normV{w}^2.
\end{equation}

Substituting \eqref{eq:sur-FM-sq-use-noisy}--\eqref{eq:sur-ih-use-noisy} into
\eqref{eq:sur-track-noisy-ito}, we obtain for \(t\ge T_\ast\),
\begin{align}
d\normH{w}^2
&\le
\Bigl(
-2\Bigl(\nu-\frac{\mu c_0^2h^2}{2}\Bigr)\normV{w}^2
-\Bigl(\mu-2C_{M,\mathrm{gl}}\Bigr)\normH{w}^2
\notag\\
&\qquad\qquad
+
2\lambda_1^{-1/2}\delta_M\,\normV{w}
+
\mu^2\Tr(\Gamma_h\Gamma_h^\top)
\Bigr)\,dt
+
2\mu \ip{w}{\Gamma_h\,dW_t}_H.
\label{eq:sur-track-noisy-pre}
\end{align}
Recall that \(\nu_{\mathrm{eff}}:=\nu-\frac{\mu c_0^2h^2}{2}>0\) and \(\gamma_{M,\mathrm{gl}}:=\mu-2C_{M,\mathrm{gl}}>0\).
Then
\[
d\normH{w}^2
\le
\Bigl(
-2\nu_{\mathrm{eff}}\normV{w}^2
-\gamma_{M,\mathrm{gl}}\normH{w}^2
+
2\lambda_1^{-1/2}\delta_M\,\normV{w}
+
\mu^2\Tr(\Gamma_h\Gamma_h^\top)
\Bigr)\,dt
+
2\mu \ip{w}{\Gamma_h\,dW_t}_H.
\]
Applying Young's inequality to the residual term with parameter
\(\nu_{\mathrm{eff}}\), we obtain
\[
2\lambda_1^{-1/2}\delta_M\,\normV{w}
\le
\nu_{\mathrm{eff}}\normV{w}^2
+
\frac{\delta_M^2}{\lambda_1\nu_{\mathrm{eff}}}.
\]
Substituting this estimate and discarding the remaining nonpositive term
\(-\nu_{\mathrm{eff}}\normV{w}^2\), we arrive at
\[
d\normH{w}^2
\le
\left(
-\gamma_{M,\mathrm{gl}}\normH{w}^2
+
\frac{\delta_M^2}{\lambda_1\nu_{\mathrm{eff}}}
+
\mu^2\Tr(\Gamma_h\Gamma_h^\top)
\right)\,dt
+
2\mu \ip{w}{\Gamma_h\,dW_t}_H.
\]
Taking expectations and using that the stochastic integral has mean zero, we
obtain
\[
\frac{d}{dt}\,\mathbb E\normH{w(t)}^2
\le
-\gamma_{M,\mathrm{gl}}\,\mathbb E\normH{w(t)}^2
+
\frac{\delta_M^2}{\lambda_1\nu_{\mathrm{eff}}}
+
\mu^2\Tr(\Gamma_h\Gamma_h^\top),
\qquad t\ge T_\ast.
\]
An application of Gronwall's inequality on \([T_\ast,t]\) yields
\eqref{eq:sur-track-bound-noisy}. Letting \(t\to\infty\) proves
\eqref{eq:sur-track-floor-noisy}.
\end{proof}

\section{Proofs for the bridge from learning errors to residual bounds}
\label{app:bridge}

\begin{proof}[Proof of Proposition~\ref{prop:bridge-route1}]
Since \(F_M\equiv \widehat F_M\) on \(\cB_\ast\), the residual
\(r_M:=F_M-F\)
satisfies
\[
r_M(z)=\widehat F_M(z)-F(z),
\qquad z\in\cB_\ast.
\]
Therefore, by definition,
\[
\delta_M
=
\sup_{z\in\cB_\ast}\normH{r_M(z)}
=
\sup_{z\in\cB_\ast}\normH{\widehat F_M(z)-F(z)}
=
\varepsilon_M^F.
\]
In particular, we have
\(\delta_M\le \varepsilon_M^F.\)

To estimate \(\ell_M\), note that \(\cB_\ast\) is convex, since it is a closed
ball in the norm induced by the positive definite matrix \(A\). Hence for any
\(z,z'\in\cB_\ast\), the segment
\[
z_\theta:=z+\theta(z'-z),
\qquad \theta\in[0,1],
\]
remains in \(\cB_\ast\). Applying the mean value theorem to
\(r_M=\widehat F_M-F\), we obtain
\[
r_M(z')-r_M(z)
=
\int_0^1 Dr_M(z_\theta)(z'-z)\,d\theta.
\]
Taking norms gives
\[
\normH{r_M(z')-r_M(z)}
\le
\left(\sup_{z\in\cB_\ast}\|Dr_M(z)\|_{\mathcal L(H,H)}\right)\normH{z'-z}.
\]
Since
\[
Dr_M(z)=D\widehat F_M(z)-DF(z),
\]
we deduce
\[
\normH{r_M(z')-r_M(z)}
\le
\eta_M^F\,\normH{z'-z}.
\]
Taking the supremum over all \(z\neq z'\) in \(\cB_\ast\) yields
\(\ell_M\le \eta_M^F.\)
Thus \eqref{eq:bridge-route1-residual} holds.

The final claim follows immediately by substituting these bounds into
Theorem~\ref{thm:sur-track}.
\end{proof}

\begin{proof}[Proof of Lemma~\ref{lem:bridge-flow-expansion}]
We first prove the flow expansion \eqref{eq:bridge-flow-expansion}. Since
\[
\dot u = F(u),
\qquad
u(0)=z,
\]
then the variation-of-constants formula gives
\[
S_{\Delta t}(z)
=
z+\int_0^{\Delta t}F(S_\tau(z))\,d\tau.
\]
Hence, we have
\[
S_{\Delta t}(z)
=
z+\Delta t\,F(z)+R_{\Delta t}(z),
\]
where
\[
R_{\Delta t}(z)
:=
\int_0^{\Delta t}\bigl(F(S_\tau(z))-F(z)\bigr)\,d\tau.
\]

Let
\[
M_0:=\sup_{z'\in\overline{\cB_\ast^{+}}}\normH{F(z')},
\qquad
M_1:=\sup_{z'\in\overline{\cB_\ast^{+}}}\|DF(z')\|_{\mathcal L(H,H)},
\]
and let \(L_1\) denote a Lipschitz constant for \(DF\) on
\(\overline{\cB_\ast^{+}}\). By assumption, these quantities are finite. Since
\(S_\tau(z)\in\cB_\ast^{+}\) for all \(z\in\cB_\ast\) and \(\tau\in[0,\Delta t]\), we
have
\[
\normH{S_\tau(z)-z}
=
\left\|\int_0^\tau F(S_\theta(z))\,d\theta\right\|_H
\le
M_0 \tau.
\]
Therefore, by the Lipschitz continuity of \(F\) on \(\overline{\cB_\ast^{+}}\),
\[
\normH{F(S_\tau(z))-F(z)}
\le
M_1 \normH{S_\tau(z)-z}
\le
M_0 M_1 \tau.
\]
Integrating in \(\tau\) yields
\[
\normH{R_{\Delta t}(z)}
\le
\int_0^{\Delta t} M_0M_1 \tau\,d\tau
=
\frac12 M_0M_1 \Delta t^2.
\]
Thus \eqref{eq:bridge-flow-expansion} holds, with any constant
\(C_{\mathrm{flow}}\) such that
\(C_{\mathrm{flow}}\ge \frac12 M_0M_1.\)

We now prove the derivative expansion \eqref{eq:bridge-Dflow-expansion}. Let
\[
J_\tau(z):=DS_\tau(z).
\]
Then \(J_\tau(z)\) satisfies the variational equation
\[
\frac{d}{d\tau}J_\tau(z)=DF(S_\tau(z))\,J_\tau(z),
\qquad
J_0(z)=I.
\]
By Gronwall's inequality,
\[
\|J_\tau(z)\|_{\mathcal L(H,H)}\le e^{M_1 \tau},
\qquad
0\le \tau\le \Delta t.
\]
Integrating the variational equation from \(0\) to \(\Delta t\), we obtain
\[
DS_{\Delta t}(z)
=
I+\int_0^{\Delta t}DF(S_\tau(z))J_\tau(z)\,d\tau.
\]
Thus
\[
DS_{\Delta t}(z)
=
I+\Delta t\,DF(z)+\widetilde R_{\Delta t}(z),
\]
where
\[
\widetilde R_{\Delta t}(z)
:=
\int_0^{\Delta t}\bigl(DF(S_\tau(z))J_\tau(z)-DF(z)\bigr)\,d\tau.
\]
We split the integrand as
\[
DF(S_\tau(z))J_\tau(z)-DF(z)
=
\bigl(DF(S_\tau(z))-DF(z)\bigr)J_\tau(z)
+
DF(z)\bigl(J_\tau(z)-I\bigr).
\]
For the first term, using the Lipschitz continuity of \(DF\),
\[
\|(DF(S_\tau(z))-DF(z))J_\tau(z)\|
\le
L_1 \normH{S_\tau(z)-z}\,e^{M_1 \tau}
\le
L_1 M_0 \tau\,e^{M_1\Delta t}.
\]
For the second term, we note that
\[
J_\tau(z)-I
=
\int_0^\tau DF(S_\theta(z))J_\theta(z)\,d\theta,
\]
so that
\[
\|J_\tau(z)-I\|
\le
\int_0^\tau M_1 e^{M_1\theta}\,d\theta
\le
M_1 e^{M_1\Delta t}\tau.
\]
Hence
\[
\|DF(z)(J_\tau(z)-I)\|
\le
M_1^2 e^{M_1\Delta t}\tau.
\]
Combining the two bounds and integrating in \(\tau\), we find
\[
\|\widetilde R_{\Delta t}(z)\|
\le
\left(L_1M_0 e^{M_1\Delta t}+M_1^2 e^{M_1\Delta t}\right)\int_0^{\Delta t}\tau\,d\tau
\le
C_{\mathrm{flow}}\Delta t^2
\]
for a suitable enlarged constant \(C_{\mathrm{flow}}\). This proves
\eqref{eq:bridge-Dflow-expansion}.
\end{proof}

\begin{proof}[Proof of Proposition~\ref{prop:bridge-route2}]
By definition,
\[
\widehat F_M(z)
=
\frac{\widehat S_{\Delta t}^{(M)}(z)-z}{\Delta t}.
\]
Using the flow expansion from Lemma~\ref{lem:bridge-flow-expansion},
\[
S_{\Delta t}(z)=z+\Delta t F(z)+R_{\Delta t}(z),
\]
we obtain
\[
F(z)
=
\frac{S_{\Delta t}(z)-z}{\Delta t}
-
\frac{R_{\Delta t}(z)}{\Delta t}.
\]
Therefore
\[
\widehat F_M(z)-F(z)
=
\frac{\widehat S_{\Delta t}^{(M)}(z)-S_{\Delta t}(z)}{\Delta t}
+
\frac{R_{\Delta t}(z)}{\Delta t}.
\]
Taking norms and using \eqref{eq:bridge-flow-expansion}, we get
\[
\normH{\widehat F_M(z)-F(z)}
\le
\frac{\varepsilon_M^S}{\Delta t}
+
C_{\mathrm{flow}}\Delta t,
\qquad z\in\cB_\ast.
\]
Since \(F_M\equiv \widehat F_M\) on \(\cB_\ast\), it follows that
\[
\delta_M
=
\sup_{z\in\cB_\ast}\normH{F_M(z)-F(z)}
\le
\frac{\varepsilon_M^S}{\Delta t}
+
C_{\mathrm{flow}}\Delta t.
\]
This proves \eqref{eq:bridge-delta-bound}.

For the Lipschitz residual, differentiate \eqref{eq:bridge-def-FhatM} to obtain
\[
D\widehat F_M(z)
=
\frac{D\widehat S_{\Delta t}^{(M)}(z)-I}{\Delta t}.
\]
On the other hand, by \eqref{eq:bridge-Dflow-expansion},
\[
DS_{\Delta t}(z)
=
I+\Delta t DF(z)+\widetilde R_{\Delta t}(z),
\]
so
\[
DF(z)
=
\frac{DS_{\Delta t}(z)-I}{\Delta t}
-
\frac{\widetilde R_{\Delta t}(z)}{\Delta t}.
\]
Hence
\[
D\widehat F_M(z)-DF(z)
=
\frac{D\widehat S_{\Delta t}^{(M)}(z)-DS_{\Delta t}(z)}{\Delta t}
+
\frac{\widetilde R_{\Delta t}(z)}{\Delta t}.
\]
Taking operator norms and using \eqref{eq:bridge-Dflow-expansion}, we obtain
\[
\|D\widehat F_M(z)-DF(z)\|_{\mathcal L(H,H)}
\le
\frac{\eta_M^S}{\Delta t}
+
C_{\mathrm{flow}}\Delta t,
\qquad z\in\cB_\ast.
\]

Now set
\(r_M:=F_M-F.\)
Since \(F_M\equiv \widehat F_M\) on \(\cB_\ast\), we have on \(\cB_\ast\),
\[
Dr_M(z)=D\widehat F_M(z)-DF(z).
\]
Arguing exactly as in the proof of Proposition~\ref{prop:bridge-route1}, using
the convexity of \(\cB_\ast\) and the mean value theorem, we conclude that
\[
\ell_M
\le
\sup_{z\in\cB_\ast}\|Dr_M(z)\|_{\mathcal L(H,H)}
\le
\frac{\eta_M^S}{\Delta t}
+
C_{\mathrm{flow}}\Delta t.
\]
This proves \eqref{eq:bridge-ell-bound}.

The final statement follows by substituting
\eqref{eq:bridge-delta-bound}--\eqref{eq:bridge-ell-bound} into
Theorem~\ref{thm:sur-track}.
\end{proof}

\section{Direct vector-field learning via dictionary learning}
\label{app:route1}

\begin{proof}[Proof of Lemma~\ref{lem:dict-bounds}]
Since \(\cB_\ast\) is compact and each \(\varphi_k\) is \(C^2\) on \(\mathcal U_{\mathrm{ext}}\), the feature map \(\phi\) and its Jacobian \(D\phi\) are continuous on a neighborhood of \(\cB_\ast\). Hence both are bounded on \(\cB_\ast\), which gives the claimed constants \(B_0\) and \(B_1\).
\end{proof}

\begin{proof}[Proof of Lemma~\ref{lem:dict-vJ-vs-theta}]
By \eqref{eq:dict-candidate} and Assumption~\ref{ass:dict-realizable},
\[
F_\Theta(z)-F(z)
=
(\Theta-\Theta^\star)^\top\phi(z).
\]
Therefore, we have
\[
\normH{F_\Theta(z)-F(z)}
\le
\|\Theta-\Theta^\star\|_{\mathrm{op}}\,\|\phi(z)\|_2.
\]
Taking the supremum over \(z\in\cB_\ast\) and using Lemma~\ref{lem:dict-bounds} gives
\[
\sup_{z\in\cB_\ast}\normH{F_\Theta(z)-F(z)}
\le
B_0\,\|\Theta-\Theta^\star\|_{\mathrm{op}}.
\]

Similarly, following from
\[
DF_\Theta(z)-DF(z)
=
(\Theta-\Theta^\star)^\top D\phi(z),
\]
we obtain
\[
\|DF_\Theta(z)-DF(z)\|_{\mathcal L(H,H)}
\le
\|\Theta-\Theta^\star\|_{\mathrm{op}}\,
\|D\phi(z)\|_{\mathcal L(H,\R^p)}.
\]
Taking the supremum over \(z\in\cB_\ast\) and applying Lemma~\ref{lem:dict-bounds} yields
\[
\sup_{z\in\cB_\ast}
\|DF_\Theta(z)-DF(z)\|_{\mathcal L(H,H)}
\le
B_1\,\|\Theta-\Theta^\star\|_{\mathrm{op}}.
\]
\end{proof}

\begin{proof}[Proof of Proposition~\ref{prop:dict-param}]
Fix \(\ell\in\{1,\dots,d\}\), and set
\[
\mathbf x^{(m)}:=\phi(z^{(m)})\in\R^p,
\qquad
y^{(m)}:=Y_\ell^{(m)}\in\R.
\]
By \eqref{eq:dict-dataset} and Assumption~\ref{ass:dict-realizable},
\[
y^{(m)}
=
\langle \beta^\star_{(\ell)},\mathbf x^{(m)}\rangle
+
\xi_\ell^{(m)},
\qquad
\beta^\star_{(\ell)}:=\Theta^\star_{\cdot,\ell}.
\]
Since the multivariate least-squares problem \eqref{eq:dict-ols} separates
over the output coordinates, the corresponding OLS estimator is
\[
\hat\beta_{(\ell)}:=\hat\Theta_{M,\cdot,\ell}.
\]
Thus each coordinate is a random-design linear regression problem.

We apply \cite[Theorem~1 and Remark~9]{HsuKakadeZhang2014} with
\(\lambda=0\), where the covariate dimension in the cited theorem is \(p\).
We first check the assumptions in that result. Regarding the bounded leverage
condition \cite[Condition~1]{HsuKakadeZhang2014}, Lemma~\ref{lem:dict-bounds}
and Assumption~\ref{ass:dict-sigma} imply
\[
\|\Sigma^{-1/2}\mathbf x^{(m)}\|_2
\le
\|\Sigma^{-1/2}\|_{\mathrm{op}}\|\mathbf x^{(m)}\|_2
\le
\frac{B_0}{\sqrt{\lambda_{\min}}}
\qquad \text{a.s.}
\]
Hence Condition~1 holds with a leverage parameter satisfying
\(\rho_0^2 p\le B_0^2/\lambda_{\min}\). As for the noise condition
\cite[Condition~2]{HsuKakadeZhang2014}, Assumption~\ref{ass:dict-subg} applied
with \(a=e_\ell\) gives that \(\xi_\ell^{(m)}\) is conditionally mean-zero and
\(\sigma\)-sub-Gaussian. Finally, regarding the approximation condition
\cite[Condition~3]{HsuKakadeZhang2014}, the support condition
\(z^{(m)}\in\cB_\ast\) together with Assumption~\ref{ass:dict-realizable}
implies
\[
\E\bigl[y^{(m)}\mid \mathbf x^{(m)}\bigr]
=
\langle \beta^\star_{(\ell)},\mathbf x^{(m)}\rangle.
\]
Thus the deterministic approximation term vanishes, so Condition~3 holds with
\(b_0=0\).

Consequently, for absolute constants \(C_{\mathrm{cov}},C_{\mathrm{ols}}>0\),
if
\[
M
\ge
C_{\mathrm{cov}}\,
\frac{B_0^2}{\lambda_{\min}}(p+t),
\]
then, after increasing \(C_{\mathrm{cov}}\) if necessary to meet the harmless
lower restriction on \(t\) in \cite[Theorem~1]{HsuKakadeZhang2014}, we have
with probability at least \(1-3e^{-t}\),
\[
\|\hat\beta_{(\ell)}-\beta^\star_{(\ell)}\|_2
\le
\frac{C_{\mathrm{ols}}\sigma}{\sqrt{M\lambda_{\min}}}
\sqrt{p+t}.
\]
Here we used \(\Sigma\succeq \lambda_{\min}I_p\) to convert the corresponding
covariance-norm estimate into the Euclidean parameter bound.

Taking \(t=\log(3d/\delta)\) and applying a union bound over
\(\ell=1,\dots,d\), we obtain, with probability at least \(1-\delta\),
\[
\max_{1\le \ell\le d}
\|\hat\beta_{(\ell)}-\beta^\star_{(\ell)}\|_2
\le
\frac{C_{\mathrm{ols}}\sigma}{\sqrt{M\lambda_{\min}}}
\sqrt{p+\log(3d/\delta)}.
\]
Since the columns of \(\hat\Theta_M-\Theta^\star\) are
\(\hat\beta_{(\ell)}-\beta^\star_{(\ell)}\), we have
\[
\|\hat\Theta_M-\Theta^\star\|_{\mathrm{op}}
\le
\|\hat\Theta_M-\Theta^\star\|_{\mathrm F}
\le
\sqrt d\,
\max_{1\le \ell\le d}
\|\hat\beta_{(\ell)}-\beta^\star_{(\ell)}\|_2.
\]
Combining the last two displays gives
\[
\|\hat\Theta_M-\Theta^\star\|_{\mathrm{op}}
\le
\frac{C_{\mathrm{ols}}\sqrt d}{\sqrt{M\lambda_{\min}}}
\,\sigma\sqrt{p+\log(3d/\delta)}.
\]
This proves \eqref{eq:dict-param-bound}.
\end{proof}

\begin{proof}[Proof of Theorem~\ref{thm:dict-complexity}]
By Proposition~\ref{prop:dict-param}, whenever
\(M
\ge
C_{\mathrm{cov}}\,
\frac{B_0^2}{\lambda_{\min}}
\bigl(p+\log(3d/\delta)\bigr),\)
we have, with probability at least \(1-\delta\),
\[
\|\hat\Theta_M-\Theta^\star\|_{\mathrm{op}}
\le
\frac{C_{\mathrm{ols}}\sqrt d}{\sqrt{M\lambda_{\min}}}
\,\sigma\sqrt{p+\log(3d/\delta)}.
\]
Moreover, since \(\widehat F_M=F_{\hat\Theta_M}\), for every realization of \(\hat\Theta_M\), Lemma~\ref{lem:dict-vJ-vs-theta} gives,
\[
\varepsilon_M^F
\le
B_0\|\hat\Theta_M-\Theta^\star\|_{\mathrm{op}},
\qquad
\eta_M^F
\le
B_1\|\hat\Theta_M-\Theta^\star\|_{\mathrm{op}}.
\]
It follows that
\[
\max\{\varepsilon_M^F,\eta_M^F\}
\le
\max\{B_0,B_1\}\,
\|\hat\Theta_M-\Theta^\star\|_{\mathrm{op}}.
\]
Combining this with \eqref{eq:dict-param-bound}, we obtain
\[
\max\{\varepsilon_M^F,\eta_M^F\}
\le
\frac{C_{\mathrm{dict}}}{\sqrt M}
\,\sigma\sqrt{p+\log(3d/\delta)}
\]
with probability at least \(1-\delta\). Therefore, if \(M\) also satisfies the
second lower bound in \eqref{eq:dict-sample-complexity}, then
\[
\max\{\varepsilon_M^F,\eta_M^F\}
\le
\min\{\bar\varepsilon,\bar\eta\},
\]
and hence
\[
\varepsilon_M^F\le \bar\varepsilon,
\qquad
\eta_M^F\le \bar\eta.
\]
Since
\[
C_{\mathrm{dict}}^2
=
\frac{C_{\mathrm{ols}}^2 d\,\max\{B_0,B_1\}^2}{\lambda_{\min}},
\]
the stated order of the sample-size requirement follows after absorbing
constants depending only on
\(B_0,B_1,\lambda_{\min},C_{\mathrm{cov}},C_{\mathrm{ols}}\). This proves the
claim.
\end{proof}

\begin{proof}[Proof of Corollary~\ref{cor:dict-aot-valid}]
For any \(M\) satisfying \eqref{eq:dict-sample-complexity},
Theorem~\ref{thm:dict-complexity} ensures that, with probability at least \(1-\delta\), 
\[
\varepsilon_M^F\le \bar\varepsilon,
\qquad
\eta_M^F\le \bar\eta.
\]
We work on this high-probability event. Since the cutoff-extended surrogate
agrees with the learned local drift on \(\cB_\ast\),
Proposition~\ref{prop:bridge-route1} gives
\[
\delta_M\le \varepsilon_M^F\le \bar\varepsilon,
\qquad
\ell_M\le \eta_M^F\le \bar\eta.
\]
This proves item \emph{(i)}.

The assumptions
\(\mu>2\bigl(C_{\mathrm{sq}}+\bar\eta\bigr),
\,
\mu c_0^2h^2<\nu\)
then imply
\[
\gamma_M
:=
\mu-2\bigl(C_{\mathrm{sq}}+\ell_M\bigr)
\ge
\mu-2\bigl(C_{\mathrm{sq}}+\bar\eta\bigr)
=
\gamma_{\mathrm F}>0,
\]
and
\[
\nu_{\mathrm{eff}}
=
\nu-\frac{\mu c_0^2h^2}{2}>0.
\]
Thus the residual-quantity hypotheses of Theorem~\ref{thm:sur-track} are satisfied. Let
\(w=v-u\). By \eqref{eq:sur-track-error}, for every \(t\ge T_\ast\),
\[
\frac{d}{dt}\normH{w}^2
\le
-\gamma_M\normH{w}^2
+
\frac{\delta_M^2}{\lambda_1\nu_{\mathrm{eff}}}.
\]
Using the learning-error bounds
\(\delta_M\le \varepsilon_M^F\le\bar\varepsilon,
\,
\ell_M\le \eta_M^F\le\bar\eta,\)
we obtain the weaker estimate
\[
\frac{d}{dt}\normH{w}^2
\le
-\gamma_{\mathrm F}\normH{w}^2
+
\frac{\bar\varepsilon^2}{\lambda_1\nu_{\mathrm{eff}}},
\qquad t\ge T_\ast.
\]
Solving this scalar differential inequality gives
\[
\|v(t)-u(t)\|_H^2
\le
e^{-\gamma_{\mathrm F}(t-T_\ast)}
\|v(T_\ast)-u(T_\ast)\|_H^2
+
\frac{\bar\varepsilon^2}{\lambda_1\nu_{\mathrm{eff}}\gamma_{\mathrm F}}
\Bigl(1-e^{-\gamma_{\mathrm F}(t-T_\ast)}\Bigr),
\]
which is \eqref{eq:dict-aot-valid-bound}. Letting \(t\to\infty\) gives
\eqref{eq:dict-aot-valid-floor}.
\end{proof}

\section{Proofs for the DSRN-based complexity analysis}
\label{app:dsrn}

\begin{proof}[Proof of Proposition~\ref{prop:r2-dsrn-approx}]
We first transfer the approximation problem from \(\cB_\ast\) to the unit cube,
where the DSRN Sobolev approximation theorem applies, and then pull the
resulting approximant back to \(\cB_\ast\). Throughout the proof, we use the
convention that Sobolev spaces on \(\cB_\ast\) are understood through its
interior.

Since this interior is a bounded Lipschitz domain, there exists a bounded
Sobolev extension operator
\[
E:W^{m,\infty}(\cB_\ast;H)\to W^{m,\infty}(Q;H),
\]
where \(Q\subset\R^d\) is a cube containing \(\cB_\ast\), such that
\[
\|E\mathsf g\|_{W^{m,\infty}(Q;H)}
\le
C_{\mathrm{ext}}\|\mathsf g\|_{W^{m,\infty}(\cB_\ast;H)}
\qquad
\forall\,\mathsf g\in W^{m,\infty}(\cB_\ast;H).
\]
Let \(T:(0,1)^d\to Q\) be an affine bijection and define
\[
G:=(ES_{\Delta t})\circ T\in W^{m,\infty}((0,1)^d;H).
\]
By the affine change-of-variables formula for Sobolev norms, there exists a
constant \(C_T>0\), depending only on the geometry of \(Q\), such that
\[
\|G\|_{W^{m,\infty}((0,1)^d;H)}
\le
C_T\|ES_{\Delta t}\|_{W^{m,\infty}(Q;H)}
\le
C_T C_{\mathrm{ext}}
\|S_{\Delta t}\|_{W^{m,\infty}(\cB_\ast;H)}.
\]

We now apply \cite[Corollary~2]{yang2023nearlyoptimalvcdimensionpseudodimension} on the unit cube, with
\(p=\infty\), source regularity \(m\), and target Sobolev order \(s\), to the
\(H\simeq\R^d\)-valued map \(G\). Here the \(H\)-valued Sobolev norm is
understood through the equivalent coordinate Sobolev norms. Thus, there exist
integers \(N_0,L_0\in\N\), depending only on \(d,m,s\), and
\(\|G\|_{W^{m,\infty}((0,1)^d;H)}\), such that for every \(N\ge N_0\) and
\(L\ge L_0\), there exists a DSRN \(\widetilde{\mathsf h}_{N,L}\) on
\((0,1)^d\) satisfying
\[
\|\widetilde{\mathsf h}_{N,L}-G\|_{W^{s,\infty}((0,1)^d;H)}
\le
C
\|G\|_{W^{m,\infty}((0,1)^d;H)}
N^{-2(m-s)/d}L^{-2(m-s)/d},
\]
and
\[
\mathrm{width}(\widetilde{\mathsf h}_{N,L})
\le
C_{\mathrm{wd}}\,N\log(2+N),
\qquad
\mathrm{depth}(\widetilde{\mathsf h}_{N,L})
\le
C_{\mathrm{dp}}\,L\log(2+L).
\]
Here \(C,C_{\mathrm{wd}},C_{\mathrm{dp}}\) may depend on the input and output
dimensions, the Sobolev orders, and the architectural constants in the DSRN
approximation theorem, but not on \(N\) or \(L\).

Define
\[
\mathsf h^\star_{N,L}
:=
\widetilde{\mathsf h}_{N,L}\circ T^{-1}\big|_{\cB_\ast}.
\]
The fixed affine map \(T^{-1}\) can be absorbed into the first affine layer,
and restriction to \(\cB_\ast\) does not change the order of width or depth.
Hence the same width-depth bounds hold for \(\mathsf h^\star_{N,L}\), up to
constants depending only on the geometry of \(Q\). Moreover, by restriction to
\(\cB_\ast\) and another affine change of variables,
\[
\|\mathsf h^\star_{N,L}-S_{\Delta t}\|_{W^{s,\infty}(\cB_\ast;H)}
\le
C
\|\widetilde{\mathsf h}_{N,L}-G\|_{W^{s,\infty}((0,1)^d;H)}.
\]
Combining the preceding estimates yields
\[
\|\mathsf h^\star_{N,L}-S_{\Delta t}\|_{W^{s,\infty}(\cB_\ast;H)}
\le
A_0\,N^{-2(m-s)/d}L^{-2(m-s)/d},
\]
where
\[
A_0
:=
C_{\mathrm{app}}
\bigl(1+\|S_{\Delta t}\|_{W^{m,\infty}(\cB_\ast;H)}\bigr),
\]
and \(C_{\mathrm{app}}>0\) depends only on \(d,m,s\), and the geometry of
\(\cB_\ast\). Since \(\cB_\ast\) has finite Lebesgue measure, the same estimate
also gives
\[
\|\mathsf h^\star_{N,L}-S_{\Delta t}\|_{W^{s,2}(\cB_\ast;H)}
\le
C A_0\,N^{-2(m-s)/d}L^{-2(m-s)/d}.
\]
Absorbing this additional constant into \(A_0\), we keep the notation
\[
\|\mathsf h^\star_{N,L}-S_{\Delta t}\|_{W^{s,2}(\cB_\ast;H)}
\le
A_0\,N^{-2(m-s)/d}L^{-2(m-s)/d}.
\]

Set
\[
\varepsilon_0
:=
\min\Bigl\{
1,\,
A_0\,N_0^{-4(m-s)/d},\,
A_0\,L_0^{-4(m-s)/d}
\Bigr\}.
\]
Fix \(0<\varepsilon_{\rm app}\le\varepsilon_0\), and choose
\[
N_{\rm app}
=
L_{\rm app}
:=
\left\lceil
\left(
\frac{A_0}{\varepsilon_{\rm app}}
\right)^{\frac{d}{4(m-s)}}
\right\rceil.
\]
Then \(N_{\rm app}\ge N_0\) and \(L_{\rm app}\ge L_0\), so the preceding
approximation estimate applies. Moreover,
\[
A_0
N_{\rm app}^{-2(m-s)/d}
L_{\rm app}^{-2(m-s)/d}
\le
\varepsilon_{\rm app}.
\]
Therefore
\[
\|\mathsf h^\star_{N_{\rm app},L_{\rm app}}
-S_{\Delta t}\|_{W^{s,2}(\cB_\ast;H)}
\le
\varepsilon_{\rm app}.
\]
This proves \eqref{eq:r2-approx}.

Finally, define
\[
W_{\rm app}
:=
\left\lceil
C_{\mathrm{wd}}\,N_{\rm app}\log(2+N_{\rm app})
\right\rceil,
\qquad
D_{\rm app}
:=
\left\lceil
C_{\mathrm{dp}}\,L_{\rm app}\log(2+L_{\rm app})
\right\rceil.
\]
Since the preceding \(W^{s,\infty}\)-estimate and
\(\varepsilon_{\rm app}\le1\) give a uniform \(W^{s,\infty}\)-bound on
\(\mathsf h^\star_{N_{\rm app},L_{\rm app}}\), this approximant satisfies the
envelope restriction after taking the fixed envelope level \(C_{\mathrm{env}}\)
large enough in Assumption~\ref{ass:r2-dsrn-envelope}. Hence
\[
\mathsf h^\star_{\varepsilon_{\rm app}}
:=
\mathsf h^\star_{N_{\rm app},L_{\rm app}}
\in
\mathcal H^{\mathrm{DSRN}}_{W_{\rm app},D_{\rm app}}.
\]
Since
\[
N_{\rm app}=L_{\rm app}
\lesssim
\varepsilon_{\rm app}^{-\frac{d}{4(m-s)}},
\]
we obtain
\[
W_{\rm app}
\le
C_W\,\varepsilon_{\rm app}^{-\frac{d}{4(m-s)}}
\log \bigl(2+\varepsilon_{\rm app}^{-1}\bigr),
\qquad
D_{\rm app}
\le
C_D\,\varepsilon_{\rm app}^{-\frac{d}{4(m-s)}}
\log \bigl(2+\varepsilon_{\rm app}^{-1}\bigr),
\]
after absorbing constants into \(C_W\) and \(C_D\). This proves
\eqref{eq:r2-WD-from-epsapp}.
\end{proof}

\begin{proof}[Proof of Lemma~\ref{lem:r2-loss-compare}]
By definition,
\[
\mathcal L_s^\varrho(\mathsf h)
=
\sum_{|\alpha|\le s}
\int_{\cB_\ast}
\|D^\alpha(\mathsf h-S_{\Delta t})(z)\|_2^2\,\omega(z)\,dz.
\]
Using the lower and upper bounds in Assumption~\ref{ass:r2-mu-density}, we
obtain
\[
\omega_{\min}
\sum_{|\alpha|\le s}
\int_{\cB_\ast}
\|D^\alpha(\mathsf h-S_{\Delta t})(z)\|_2^2\,dz
\le
\mathcal L_s^\varrho(\mathsf h)
\le
\omega_{\max}
\sum_{|\alpha|\le s}
\int_{\cB_\ast}
\|D^\alpha(\mathsf h-S_{\Delta t})(z)\|_2^2\,dz.
\]
The two sums on the left and right are precisely
\(\mathcal L_s^{\mathrm{Leb}}(\mathsf h)\), which gives
\eqref{eq:r2-loss-compare}.
\end{proof}

\begin{proof}[Proof of Proposition~\ref{prop:r2-deriv-complexity}]
Fix \(|\alpha|\le s\) and \(1\le \ell\le d\). The class
\(\mathcal G_{\alpha,\ell;W,D}\) consists of scalar derivatives of a fixed
output coordinate of DSRNs with width at most \(W\) and depth at most \(D\).
We use the derivative pseudo-dimension bound for DSRNs in
\cite[Theorem~5]{yang2023dsrn}, together with the higher-order derivative
extension discussed there.

For fixed derivative order \(|\alpha|\) and fixed output coordinate \(\ell\),
this bound gives
\[
\Pdim(\mathcal G_{\alpha,\ell;W,D})
\le
C_{\alpha,\ell}\,
W^2D^2\log(2+W)\,\log(2+D),
\]
where \(C_{\alpha,\ell}\) is independent of \(W\) and \(D\). Since
\(|\alpha|\le s\), the output coordinate \(\ell\) ranges over
\(\{1,\dots,d\}\), and \(s\) is fixed throughout this subsection, all
dependence on \(\alpha\) and \(\ell\) may be absorbed into a constant depending
only on \(d\) and \(s\). Hence there exists
\(C_{\mathrm{pdim}}=C_{\mathrm{pdim}}(d,s)>0\) such that
\[
\Pdim(\mathcal G_{\alpha,\ell;W,D})
\le
C_{\mathrm{pdim}}\,
W^2D^2\log(2+W)\,\log(2+D).
\]
This proves \eqref{eq:r2-pdim}.
\end{proof}

\begin{proof}[Proof of Lemma~\ref{lem:r2-envelope}]
By Assumption~\ref{ass:r2-dsrn-envelope},
\[
\sup_{z\in\cB_\ast}
\|\mathcal Z_s\mathsf h(z)\|_2^2
=
\sup_{z\in\cB_\ast}
\sum_{|\alpha|\le s}\|D^\alpha \mathsf h(z)\|_2^2
\le
C_{\mathrm{env}}^2
\qquad
\forall\,\mathsf h\in\mathcal H_{W,D}^{\mathrm{DSRN}}.
\]
It remains to control \(\mathcal Z_s S_{\Delta t}\). Since
\(S_{\Delta t}\in W^{m,\infty}(\cB_\ast;H)\) with \(m>s\), there exists a
constant \(C>0\), depending only on \(d,m,s\), and the geometry of
\(\cB_\ast\), such that
\[
\sup_{z\in\cB_\ast}
\sum_{|\alpha|\le s}\|D^\alpha S_{\Delta t}(z)\|_2^2
\le
C\|S_{\Delta t}\|_{W^{m,\infty}(\cB_\ast;H)}^2.
\]
Consequently, after enlarging constants if necessary, there exists
\(C_{\mathrm{feat}}\ge1\) such that
\[
\sup_{z\in\cB_\ast}\|\mathcal Z_s\mathsf h(z)\|_2\le C_{\mathrm{feat}},
\qquad
\sup_{z\in\cB_\ast}\|\mathcal Z_sS_{\Delta t}(z)\|_2\le C_{\mathrm{feat}}
\qquad
\forall\,\mathsf h\in\mathcal H_{W,D}^{\mathrm{DSRN}}.
\]
This proves \eqref{eq:r2-feature-envelope}.

For the loss envelope, let \(g\in\mathcal F_{s,W,D}\). Then, for some
\(\mathsf h\in\mathcal H_{W,D}^{\mathrm{DSRN}}\),
\[
g(z)
=
\|\mathcal Z_s\mathsf h(z)-\mathcal Z_sS_{\Delta t}(z)\|_2^2.
\]
By the triangle inequality and \eqref{eq:r2-feature-envelope},
\[
\|\mathcal Z_s\mathsf h(z)-\mathcal Z_sS_{\Delta t}(z)\|_2
\le
\|\mathcal Z_s\mathsf h(z)\|_2+\|\mathcal Z_sS_{\Delta t}(z)\|_2
\le
2C_{\mathrm{feat}}.
\]
Squaring gives
\[
0\le g(z)\le 4C_{\mathrm{feat}}^2,
\]
which is \eqref{eq:r2-loss-envelope}.
\end{proof}

\begin{proof}[Proof of Lemma~\ref{lem:r2-covering}]
Fix a sample
\(\mathbf z:=(z^{(1)},\dots,z^{(M)})\in \cB_\ast^M\)
and an accuracy level \(\varepsilon\in(0,1]\). Throughout the proof, we use
external empirical covers: the centers of the cover need not belong to the
function class being covered. For scalar-valued functions \(g,g'\) on
\(\{z^{(1)},\dots,z^{(M)}\}\), write
\[
\|g-g'\|_{\infty,\mathbf z}
:=
\max_{1\le j\le M}|g(z^{(j)})-g'(z^{(j)})|.
\]
For vector-valued functions \(U,V\) on \(\{z^{(1)},\dots,z^{(M)}\}\), write
\[
\|U-V\|_{\infty,\mathbf z;2}
:=
\max_{1\le j\le M}\|U(z^{(j)})-V(z^{(j)})\|_2.
\]

Let
\[
\Pi_s:=\{(\alpha,\ell): |\alpha|\le s,\ 1\le \ell\le d\},
\qquad
|\Pi_s|=r_s=d\binom{d+s}{s}.
\]
For each \((\alpha,\ell)\in \Pi_s\), recall that
\[
\mathcal G_{\alpha,\ell;W,D}
=
\bigl\{
z\mapsto D^\alpha \mathsf h_\ell(z):\
\mathsf h\in\mathcal H_{W,D}^{\mathrm{DSRN}}
\bigr\}.
\]

We first bound the empirical covering number of each scalar derivative class.
By Lemma~\ref{lem:r2-envelope}, every function in
\(\mathcal G_{\alpha,\ell;W,D}\) is uniformly bounded in absolute value by
\(C_{\mathrm{feat}}\). Let
\[
p_{\alpha,\ell}:=\Pdim(\mathcal G_{\alpha,\ell;W,D}).
\]
By Proposition~\ref{prop:r2-deriv-complexity},
\[
p_{\alpha,\ell}
\le
C(d,s)\,W^2D^2\log(2+W)\log(2+D).
\]
Applying the uniform empirical \(L^\infty\)-covering estimate in
\cite[Lemma~8]{yang2023dsrn} with \(B=C_{\mathrm{feat}}\), \(n=M\), and
\(\varepsilon=\varepsilon_{\mathrm{cov}}\), we get, when
\(M\ge p_{\alpha,\ell}\),
\[
\log \mathcal N_\infty \bigl(
\varepsilon_{\mathrm{cov}},\mathcal G_{\alpha,\ell;W,D}|_{\mathbf z}
\bigr)
\le
p_{\alpha,\ell}
\log \Bigl(
\frac{2eM C_{\mathrm{feat}}}
{\varepsilon_{\mathrm{cov}}p_{\alpha,\ell}}
\Bigr).
\]
When \(M<p_{\alpha,\ell}\), we use the boundedness directly. Since
\(\mathcal G_{\alpha,\ell;W,D}|_{\mathbf z}\subset
[-C_{\mathrm{feat}},C_{\mathrm{feat}}]^M\), a standard grid cover gives
\[
\log \mathcal N_\infty\bigl(
\varepsilon_{\mathrm{cov}},
\mathcal G_{\alpha,\ell;W,D}|_{\mathbf z}
\bigr)
\le
M\log\Bigl(\frac{3C_{\mathrm{feat}}}{\varepsilon_{\mathrm{cov}}}\Bigr).
\]
Since \(M<p_{\alpha,\ell}\), this is bounded by the same expression as above,
after enlarging the constant. Since \(p_{\alpha,\ell}\ge1\), in both cases
there exists a constant \(C_1=C_1(d,s)>0\) such that, for every
\(\varepsilon_{\mathrm{cov}}\in(0,1]\),
\begin{equation}\label{eq:r2-scalar-cover-proof}
\log \mathcal N_\infty \bigl(
\varepsilon_{\mathrm{cov}},\mathcal G_{\alpha,\ell;W,D}|_{\mathbf z}
\bigr)
\le
C_1\,W^2D^2\log(2+W)\,\log(2+D)\,
\log \Bigl(\frac{C_1\,C_{\mathrm{feat}}\,M}{\varepsilon_{\mathrm{cov}}}\Bigr).
\end{equation}
Here the cited result controls the uniform empirical covering number, so it
also applies to the fixed sample \(\mathbf z\).

We now choose
\begin{equation}\label{eq:r2-eps-cov-choice-proof}
\varepsilon_{\mathrm{cov}}
:=
\frac{\varepsilon}{8C_{\mathrm{feat}}\sqrt{r_s}}.
\end{equation}
For each \((\alpha,\ell)\in\Pi_s\), let
\(\mathcal C_{\alpha,\ell}
\subset \R^{M}\)
be an empirical \(L^\infty\)-cover of
\(\mathcal G_{\alpha,\ell;W,D}|_{\mathbf z}\) at radius
\(\varepsilon_{\mathrm{cov}}\), whose cardinality realizes
\(\mathcal N_\infty(\varepsilon_{\mathrm{cov}},
\mathcal G_{\alpha,\ell;W,D}|_{\mathbf z})\) up to the usual ceiling
convention. By \eqref{eq:r2-scalar-cover-proof}, each such cover satisfies
\begin{equation}\label{eq:r2-card-bound-proof}
\log |\mathcal C_{\alpha,\ell}|
\le
C_1\,W^2D^2\log(2+W)\,\log(2+D)\,
\log \Bigl(\frac{C_1\,C_{\mathrm{feat}}\,M}{\varepsilon_{\mathrm{cov}}}\Bigr).
\end{equation}

Taking the Cartesian product of these scalar covers over all
\((\alpha,\ell)\in\Pi_s\), we obtain a finite family
\[
\mathcal C_{\mathrm{vec}}
:=
\prod_{(\alpha,\ell)\in\Pi_s}\mathcal C_{\alpha,\ell}.
\]
Each element of \(\mathcal C_{\mathrm{vec}}\) defines a vector-valued map on
the sample \(\{z^{(1)},\dots,z^{(M)}\}\), with one coordinate for each pair
\((\alpha,\ell)\in\Pi_s\). We claim that \(\mathcal C_{\mathrm{vec}}\) gives an
empirical \(L^\infty\)-cover of the restricted feature class
\[
\mathcal Z_s\mathcal H_{W,D}^{\mathrm{DSRN}}
:=
\bigl\{
z\mapsto \mathcal Z_s\mathsf h(z):\
\mathsf h\in\mathcal H_{W,D}^{\mathrm{DSRN}}
\bigr\}
\]
at radius \(\varepsilon/(8C_{\mathrm{feat}})\) in the norm
\(\|\cdot\|_{\infty,\mathbf z;2}\).

Indeed, let \(\mathsf h\in\mathcal H_{W,D}^{\mathrm{DSRN}}\) be arbitrary. For
each \((\alpha,\ell)\in\Pi_s\), since \(\mathcal C_{\alpha,\ell}\) covers the
corresponding scalar restriction, choose
\(c_{\alpha,\ell}\in\mathcal C_{\alpha,\ell}\) such that
\[
\max_{1\le j\le M}
\bigl|
D^\alpha \mathsf h_\ell(z^{(j)})-c_{\alpha,\ell}(z^{(j)})
\bigr|
\le
\varepsilon_{\mathrm{cov}}.
\]
Collecting these coordinates defines an element
\(\widetilde{\mathcal Z}_s^{\,\mathsf h}\in \mathcal C_{\mathrm{vec}}\).
Then, for every \(1\le j\le M\),
\[
\bigl\|
\mathcal Z_s\mathsf h(z^{(j)})
-\widetilde{\mathcal Z}_s^{\,\mathsf h}(z^{(j)})
\bigr\|_2^2
=
\sum_{(\alpha,\ell)\in\Pi_s}
\bigl|
D^\alpha \mathsf h_\ell(z^{(j)})-c_{\alpha,\ell}(z^{(j)})
\bigr|^2
\le
r_s\varepsilon_{\mathrm{cov}}^2.
\]
Taking square roots and then the maximum over \(j\) gives
\begin{equation}\label{eq:r2-feature-cover-proof}
\bigl\|
\mathcal Z_s\mathsf h-\widetilde{\mathcal Z}_s^{\,\mathsf h}
\bigr\|_{\infty,\mathbf z;2}
\le
\sqrt{r_s}\,\varepsilon_{\mathrm{cov}}
=
\frac{\varepsilon}{8C_{\mathrm{feat}}}.
\end{equation}

We next pass from the feature class to the loss class. For each \(z\in\cB_\ast\),
define
\[
\Psi_z(\eta):=\|\eta-\mathcal Z_s S_{\Delta t}(z)\|_2^2,
\qquad \eta\in\R^{r_s}.
\]
By Lemma~\ref{lem:r2-envelope},
\[
\|\mathcal Z_s\mathsf h(z)\|_2\le C_{\mathrm{feat}},
\qquad
\|\mathcal Z_s S_{\Delta t}(z)\|_2\le C_{\mathrm{feat}}
\qquad
\forall\,z\in\cB_\ast,\ \forall\,\mathsf h\in\mathcal H_{W,D}^{\mathrm{DSRN}}.
\]
Moreover, if \(\eta'\) lies within distance \(\varepsilon/(8C_{\mathrm{feat}})\)
of some vector of norm at most \(C_{\mathrm{feat}}\), then, since
\(\varepsilon\le 1\) and \(C_{\mathrm{feat}}\ge 1\),
\[
\|\eta'\|_2
\le
C_{\mathrm{feat}}+\frac{\varepsilon}{8C_{\mathrm{feat}}}
\le
C_{\mathrm{feat}}+\frac18
\le
2C_{\mathrm{feat}}.
\]
Accordingly, whenever
\[
\|\eta\|_2\le C_{\mathrm{feat}},
\qquad
\|\eta'\|_2\le 2C_{\mathrm{feat}},
\qquad
\|\mathcal Z_sS_{\Delta t}(z)\|_2\le C_{\mathrm{feat}},
\]
we have
\[
|\Psi_z(\eta)-\Psi_z(\eta')|
=
\bigl|(\eta-\eta')\cdot (\eta+\eta'-2\mathcal Z_sS_{\Delta t}(z))\bigr|
\le
5C_{\mathrm{feat}}\|\eta-\eta'\|_2.
\]
Applying this with
\[
\eta=\mathcal Z_s\mathsf h(z^{(j)}),
\qquad
\eta'=\widetilde{\mathcal Z}_s^{\,\mathsf h}(z^{(j)}),
\]
and using \eqref{eq:r2-feature-cover-proof}, we obtain
\[
\max_{1\le j\le M}
\Bigl|
\Psi_{z^{(j)}}\bigl(\mathcal Z_s\mathsf h(z^{(j)})\bigr)
-
\Psi_{z^{(j)}}\bigl(\widetilde{\mathcal Z}_s^{\,\mathsf h}(z^{(j)})\bigr)
\Bigr|
\le
5C_{\mathrm{feat}}\cdot \frac{\varepsilon}{8C_{\mathrm{feat}}}
<
\varepsilon.
\]
By the definition of \(\Psi_z\), the family obtained by applying the loss map
pointwise to \(\mathcal C_{\mathrm{vec}}\) gives an empirical
\(L^\infty\)-cover of the restricted loss class
\(\mathcal F_{s,W,D}|_{\mathbf z}\) at radius \(\varepsilon\).

It remains to count the size of this cover. Since
\(|\Pi_s|=r_s\), \eqref{eq:r2-card-bound-proof} gives
\[
\log \mathcal N_\infty \bigl(
\varepsilon,\mathcal F_{s,W,D}|_{\mathbf z}
\bigr)
\le
\sum_{(\alpha,\ell)\in\Pi_s}
\log |\mathcal C_{\alpha,\ell}|
\le
r_s\,C_1\,W^2D^2\log(2+W)\,\log(2+D)\,
\log \Bigl(\frac{C_1\,C_{\mathrm{feat}}\,M}{\varepsilon_{\mathrm{cov}}}\Bigr).
\]
Substituting the choice \eqref{eq:r2-eps-cov-choice-proof} of
\(\varepsilon_{\mathrm{cov}}\), we get
\[
\frac{C_1\,C_{\mathrm{feat}}\,M}{\varepsilon_{\mathrm{cov}}}
=
\frac{8C_1\,C_{\mathrm{feat}}^2\,\sqrt{r_s}\,M}{\varepsilon}.
\]
Since \(r_s=r_s(d,s)\) depends only on \(d\) and \(s\), we absorb it into the
constant. Therefore, after redefining the constant as
\(C_{\mathrm{cov}}=C_{\mathrm{cov}}(d,s)\), we arrive at
\[
\log \mathcal N_\infty \bigl(
\varepsilon,\mathcal F_{s,W,D}|_{z^{(1)},\dots,z^{(M)}}
\bigr)
\le
C_{\mathrm{cov}}\,
W^2D^2\log(2+W)\,\log(2+D)\,
\log \Bigl(
\frac{C_{\mathrm{cov}}\,C_{\mathrm{feat}}^2\,M}{\varepsilon}
\Bigr),
\]
which proves \eqref{eq:r2-covering}.
\end{proof}

\begin{proof}[Proof of Proposition~\ref{prop:r2-uniform-generalization}]
Set
\(\Delta_M
:=
\sup_{\mathsf h\in\mathcal H_{W,D}^{\mathrm{DSRN}}}
\bigl|
\widehat{\mathcal L}_{s,M}(\mathsf h)-\mathcal L_s^\varrho(\mathsf h)
\bigr|.\)
We first reduce the problem to a uniform deviation bound for a
\([0,1]\)-valued loss class. By Lemma~\ref{lem:r2-envelope}, every function
\(g\in\mathcal F_{s,W,D}\) satisfies
\[
0\le g(z)\le 4C_{\mathrm{feat}}^2
\qquad
\forall\,z\in\cB_\ast.
\]
Let
\[
B_{\mathcal F}:=4C_{\mathrm{feat}}^2,
\qquad
\widetilde{\mathcal F}_{s,W,D}
:=
\left\{\frac{g}{B_{\mathcal F}}: g\in\mathcal F_{s,W,D}\right\}.
\]
Then \(\widetilde{\mathcal F}_{s,W,D}\) is a class of \([0,1]\)-valued
functions.

By Theorem~17.1 of \cite{AnthonyBartlett1999}, there exist universal constants
\(c_0,c_1>0\) such that, for every \(t\in(0,1]\),
\[
\mathbb P \left(
\sup_{g\in \widetilde{\mathcal F}_{s,W,D}}
\left|
\mathbb E_{\varrho}[g]-\frac1M\sum_{j=1}^M g(Z_j)
\right|
>t
\right)
\le
4\,
\mathcal N_1 \left(
c_1 t,\widetilde{\mathcal F}_{s,W,D},2M
\right)
\exp \bigl(-c_0Mt^2\bigr).
\]
Since the empirical \(L^1\) distance is bounded by the empirical
\(L^\infty\) distance on every fixed sample,
\[
\mathcal N_1 \left(
c_1 t,\widetilde{\mathcal F}_{s,W,D},2M
\right)
\le
\mathcal N_\infty \left(
c_1 t,\widetilde{\mathcal F}_{s,W,D},2M
\right).
\]
Thus
\begin{equation}\label{eq:r2-unif-dev-normalized-proof}
\mathbb P \left(
\sup_{g\in \widetilde{\mathcal F}_{s,W,D}}
\left|
\mathbb E_{\varrho}[g]-\frac1M\sum_{j=1}^M g(Z_j)
\right|
>t
\right)
\le
4\,
\mathcal N_\infty \left(
c_1 t,\widetilde{\mathcal F}_{s,W,D},2M
\right)
\exp \bigl(-c_0Mt^2\bigr).
\end{equation}

Rescaling from \(\widetilde{\mathcal F}_{s,W,D}\) back to
\(\mathcal F_{s,W,D}\), and taking \(t=a/B_{\mathcal F}\), gives, for every
\(a\in(0,1]\),
\begin{equation}\label{eq:r2-unif-dev-rescaled-proof}
\mathbb P \left(\Delta_M>a\right)
\le
4\,
\mathcal N_\infty \left(
c_1a,\mathcal F_{s,W,D},2M
\right)
\exp \left(-c_0M\frac{a^2}{B_{\mathcal F}^2}\right).
\end{equation}
Indeed,
\[
\mathcal N_\infty \left(
\frac{c_1a}{B_{\mathcal F}},
\widetilde{\mathcal F}_{s,W,D},2M
\right)
=
\mathcal N_\infty \left(
c_1a,
\mathcal F_{s,W,D},2M
\right).
\]

Here \(\mathcal N_\infty(\cdot,\mathcal F,2M)\) denotes the uniform empirical
\(L^\infty\)-covering number over samples of size \(2M\). Since
Lemma~\ref{lem:r2-covering} holds for every sample, it gives the same bound for
this uniform covering number. Hence, applying Lemma~\ref{lem:r2-covering} with
radius \(c_1a\) and sample size \(2M\), and absorbing fixed numerical constants,
we get
\[
\log \mathcal N_\infty \left(
c_1a,\mathcal F_{s,W,D},2M
\right)
\le
C_{\mathrm{cov}}\,
W^2D^2\log(2+W)\,\log(2+D)\,
\log \left(
\frac{C_{\mathrm{cov}}\,C_{\mathrm{feat}}^2\,M}{a}
\right).
\]
Substituting this into \eqref{eq:r2-unif-dev-rescaled-proof}, we obtain
\begin{equation}\label{eq:r2-unif-dev-with-cover-proof}
\mathbb P \left(\Delta_M>a\right)
\le
4\exp \Bigg(
C_{\mathrm{cov}}\,
W^2D^2\log(2+W)\,\log(2+D)\,
\log \left(
\frac{C_{\mathrm{cov}}\,C_{\mathrm{feat}}^2\,M}{a}
\right)
-
c_0M\frac{a^2}{B_{\mathcal F}^2}
\Bigg).
\end{equation}

We now turn this tail bound into an explicit sample-size condition. Specialize
to \(a=\varepsilon_{\rm gen}\) and introduce
\[
K_{W,D}:=W^2D^2\log(2+W)\,\log(2+D),
\qquad
\Gamma_{W,D,\delta}
:=
K_{W,D}+\log\frac{2}{\delta}.
\]
Let
\[
M_0
:=
\widetilde C_{\mathrm{gen}}
\frac{\Gamma_{W,D,\delta}}{\varepsilon_{\rm gen}^2}
\log \Bigl(
\frac{\widetilde C_{\mathrm{gen}}\Gamma_{W,D,\delta}}
{\varepsilon_{\rm gen}^3}
\Bigr),
\]
where \(\widetilde C_{\mathrm{gen}}\ge1\) will be chosen sufficiently large.
Assume \(M\ge M_0\).

Define the exponent in \eqref{eq:r2-unif-dev-with-cover-proof} by
\[
\Psi(M)
:=
C_{\mathrm{cov}}\,K_{W,D}\,
\log \left(
\frac{C_{\mathrm{cov}}\,C_{\mathrm{feat}}^2\,M}{\varepsilon_{\rm gen}}
\right)
-
c_0M\frac{\varepsilon_{\rm gen}^2}{B_{\mathcal F}^2}.
\]
We next show that, once \(M\ge M_0\), this exponent decreases as \(M\)
increases. Indeed,
\[
\Psi'(M)
=
\frac{C_{\mathrm{cov}}\,K_{W,D}}{M}
-
c_0\frac{\varepsilon_{\rm gen}^2}{B_{\mathcal F}^2}.
\]
Since \(\Gamma_{W,D,\delta}\ge K_{W,D}\), by enlarging
\(\widetilde C_{\mathrm{gen}}\) if necessary we may ensure that
\[
M_0
\ge
\frac{2C_{\mathrm{cov}}B_{\mathcal F}^2}{c_0}
\frac{K_{W,D}}{\varepsilon_{\rm gen}^2}.
\]
Hence, for every \(M\ge M_0\),
\[
\Psi'(M)
\le
\frac{C_{\mathrm{cov}}\,K_{W,D}}{M_0}
-
c_0\frac{\varepsilon_{\rm gen}^2}{B_{\mathcal F}^2}
\le
-\frac{c_0}{2}\frac{\varepsilon_{\rm gen}^2}{B_{\mathcal F}^2}
<0.
\]
Therefore \(\Psi\) is decreasing on \([M_0,\infty)\), so it remains only to
bound \(\Psi(M_0)\).

Set
\[
A
:=
\frac{\widetilde C_{\mathrm{gen}}\Gamma_{W,D,\delta}}
{\varepsilon_{\rm gen}^3}.
\]
Then
\[
M_0=\varepsilon_{\rm gen}^{-2}\,\widetilde C_{\mathrm{gen}}
\Gamma_{W,D,\delta}\log A.
\]
Consequently,
\[
\frac{C_{\mathrm{cov}}\,C_{\mathrm{feat}}^2\,M_0}{\varepsilon_{\rm gen}}
=
C_{\mathrm{cov}}\,C_{\mathrm{feat}}^2\,
\widetilde C_{\mathrm{gen}}\Gamma_{W,D,\delta}\,
\varepsilon_{\rm gen}^{-3}\log A
=
C_{\mathrm{cov}}\,C_{\mathrm{feat}}^2\,A\log A.
\]
After enlarging \(\widetilde C_{\mathrm{gen}}\) if necessary, we may assume
\(A\ge e\). Using the elementary bound
\[
\log(cA\log A)\le C_1\log A
\qquad\text{for all }A\ge e,
\]
with constants \(c,C_1>0\) depending only on
\(C_{\mathrm{cov}}\) and \(C_{\mathrm{feat}}\), we obtain
\[
\log \left(
\frac{C_{\mathrm{cov}}\,C_{\mathrm{feat}}^2\,M_0}{\varepsilon_{\rm gen}}
\right)
\le
C_1\log A.
\]
Substituting this bound into \(\Psi(M_0)\) yields
\[
\Psi(M_0)
\le
C_1C_{\mathrm{cov}}K_{W,D}\log A
-
c_0\widetilde C_{\mathrm{gen}}
\frac{\Gamma_{W,D,\delta}}{B_{\mathcal F}^2}\log A.
\]
Since \(K_{W,D}\le \Gamma_{W,D,\delta}\), we get
\[
\Psi(M_0)
\le
\left(
C_1C_{\mathrm{cov}}
-
c_0\widetilde C_{\mathrm{gen}}B_{\mathcal F}^{-2}
\right)
\Gamma_{W,D,\delta}\log A.
\]
Choose \(\widetilde C_{\mathrm{gen}}\) sufficiently large so that
\[
c_0\widetilde C_{\mathrm{gen}}B_{\mathcal F}^{-2}
\ge
C_1C_{\mathrm{cov}}+2.
\]
Then
\[
\Psi(M_0)\le -2\,\Gamma_{W,D,\delta}\log A.
\]
Since \(A\ge e\), we have \(\log A\ge1\), and therefore
\[
\Psi(M_0)\le -2\,\Gamma_{W,D,\delta}
\le
-2\log\frac{2}{\delta}
\le
-\log\frac{4}{\delta}.
\]
Because \(\Psi\) is decreasing on \([M_0,\infty)\), the same bound holds for
every \(M\ge M_0\). Returning to
\eqref{eq:r2-unif-dev-with-cover-proof} with \(a=\varepsilon_{\rm gen}\), we
conclude that
\[
\mathbb P \left(\Delta_M>\varepsilon_{\rm gen}\right)
\le
4e^{-\log(4/\delta)}
=
\delta.
\]
Equivalently, with probability at least \(1-\delta\), we have
\(\Delta_M\le \varepsilon_{\rm gen}\). That is,
\[
\sup_{\mathsf h\in\mathcal H_{W,D}^{\mathrm{DSRN}}}
\bigl|
\widehat{\mathcal L}_{s,M}(\mathsf h)-\mathcal L_s^\varrho(\mathsf h)
\bigr|
\le
\varepsilon_{\rm gen}.
\]
This proves \eqref{eq:r2-generalization}.
\end{proof}

\begin{proof}[Proof of Proposition~\ref{prop:r2-aerm-generalization}]
On the event in Proposition~\ref{prop:r2-uniform-generalization}, we have
\[
\mathcal L_s^\varrho(\widehat S_{\Delta t}^{(M)})
\le
\widehat{\mathcal L}_{s,M}(\widehat S_{\Delta t}^{(M)})
+
\varepsilon_{\rm gen}.
\]
By Assumption~\ref{ass:r2-aerm},
\[
\widehat{\mathcal L}_{s,M}(\widehat S_{\Delta t}^{(M)})
\le
\inf_{\mathsf h\in\mathcal H_{W,D}^{\mathrm{DSRN}}}
\widehat{\mathcal L}_{s,M}(\mathsf h)
+
\varepsilon_{\mathrm{opt}}.
\]
Applying the same uniform generalization bound again,
\[
\inf_{\mathsf h\in\mathcal H_{W,D}^{\mathrm{DSRN}}}
\widehat{\mathcal L}_{s,M}(\mathsf h)
\le
\inf_{\mathsf h\in\mathcal H_{W,D}^{\mathrm{DSRN}}}
\mathcal L_s^\varrho(\mathsf h)
+
\varepsilon_{\rm gen}.
\]
Combining the three inequalities yields
\[
\mathcal L_s^\varrho(\widehat S_{\Delta t}^{(M)})
\le
\inf_{\mathsf h\in\mathcal H_{W,D}^{\mathrm{DSRN}}}\mathcal L_s^\varrho(\mathsf h)
+
2\varepsilon_{\rm gen}
+
\varepsilon_{\mathrm{opt}},
\]
which proves \eqref{eq:r2-aerm-generalization}.
\end{proof}

\begin{proof}[Proof of Lemma~\ref{lem:r2-loss-to-W1inf}]
Let
\(e:=\mathsf h-S_{\Delta t}\in W^{s,2}(\cB_\ast;H).\)
Since \(s>\frac d2+1\) by \eqref{eq:r2-s-def}, the Sobolev embedding theorem on
the bounded Lipschitz domain \(\cB_\ast\) yields a constant
\(C_{\mathrm{SE}}>0\), depending only on \(d\), \(s\), and \(\cB_\ast\), such
that
\[
\|e\|_{W^{1,\infty}(\cB_\ast;H)}
\le
C_{\mathrm{SE}}\|e\|_{W^{s,2}(\cB_\ast;H)}.
\]
Moreover, by Lemma~\ref{lem:r2-loss-compare},
\[
\omega_{\min}\,\|e\|_{W^{s,2}(\cB_\ast;H)}^2
=
\omega_{\min}\,\mathcal L_s^{\mathrm{Leb}}(\mathsf h)
\le
\mathcal L_s^\varrho(\mathsf h).
\]
Hence
\[
\|e\|_{W^{s,2}(\cB_\ast;H)}
\le
\omega_{\min}^{-1/2}\,
\bigl(\mathcal L_s^\varrho(\mathsf h)\bigr)^{1/2}.
\]
Combining the last two displays, we obtain
\[
\|\mathsf h-S_{\Delta t}\|_{W^{1,\infty}(\cB_\ast;H)}
\le
C_{\mathrm{SE}}\,\omega_{\min}^{-1/2}\,
\bigl(\mathcal L_s^\varrho(\mathsf h)\bigr)^{1/2}.
\]
This proves \eqref{eq:r2-loss-to-W1inf} after setting
\[
C_{\mathrm{emb}}:=C_{\mathrm{SE}}\,\omega_{\min}^{-1/2}.
\]
The bounds \eqref{eq:r2-loss-to-operator} and \eqref{eq:r2-loss-to-jac} are
immediate consequences of \eqref{eq:r2-loss-to-W1inf}.
\end{proof}

\begin{proof}[Proof of Theorem~\ref{thm:r2-complexity}]
Set
\(\varepsilon_{\rm app}:=\varepsilon_\ast.\)
Since \(\varepsilon_\ast\le \varepsilon_0\) by
\eqref{eq:r2-eps-star}, Proposition~\ref{prop:r2-dsrn-approx} gives
architecture bounds \(W,D\in\N\) and a DSRN
\[
\mathsf h^\star_{\varepsilon_{\rm app}}
\in
\mathcal H^{\mathrm{DSRN}}_{W,D}
\]
such that
\[
\|\mathsf h^\star_{\varepsilon_{\rm app}}-S_{\Delta t}\|_{W^{s,2}(\cB_\ast;H)}
\le
\varepsilon_{\rm app}.
\]
By Lemma~\ref{lem:r2-loss-compare},
\[
\mathcal L_s^\varrho(\mathsf h^\star_{\varepsilon_{\rm app}})
\le
\omega_{\max}\,\mathcal L_s^{\mathrm{Leb}}(\mathsf h^\star_{\varepsilon_{\rm app}})
\le
\omega_{\max}\,\varepsilon_{\rm app}^2
=
\omega_{\max}\,\varepsilon_\ast^2.
\]
Hence
\[
\inf_{\mathsf h\in\mathcal H_{W,D}^{\mathrm{DSRN}}}
\mathcal L_s^\varrho(\mathsf h)
\le
\omega_{\max}\,\varepsilon_\ast^2.
\]

Choose
\(\varepsilon_{\rm gen}:=\varepsilon_\ast^2.\)
Since \(W,D\) satisfy \eqref{eq:r2-WD-choice}, we have
\[
W^2D^2\log(2+W)\,\log(2+D)
\le
C\,\varepsilon_\ast^{-\frac{d}{m-s}}
\,\mathrm{polylog} \Bigl(\frac1{\varepsilon_\ast}\Bigr)
\]
for a constant \(C>0\) depending only on
\(d,m,s,\cB_\ast,\|S_{\Delta t}\|_{W^{m,\infty}(\cB_\ast;H)}\). After enlarging
\(C_{\mathrm{samp}}\) if necessary, the lower bound
\eqref{eq:r2-sample-complexity} guarantees the sample-size condition
\eqref{eq:r2-gen-sample} required in
Proposition~\ref{prop:r2-uniform-generalization}, with
\(\varepsilon_{\rm gen}=\varepsilon_\ast^2\).

By Proposition~\ref{prop:r2-aerm-generalization}, with probability at least
\(1-\delta\),
\[
\mathcal L_s^\varrho\bigl(\widehat S_{\Delta t}^{(M)}\bigr)
\le
\inf_{\mathsf h\in\mathcal H_{W,D}^{\mathrm{DSRN}}}\mathcal L_s^\varrho(\mathsf h)
+
2\varepsilon_{\rm gen}
+
\varepsilon_{\mathrm{opt}}.
\]
Using the bound on the infimum and \eqref{eq:r2-opt-small}, we obtain
\[
\mathcal L_s^\varrho\bigl(\widehat S_{\Delta t}^{(M)}\bigr)
\le
\omega_{\max}\,\varepsilon_\ast^2
+
2\varepsilon_\ast^2
+
\varepsilon_\ast^2
=
(\omega_{\max}+3)\,\varepsilon_\ast^2.
\]

Applying Lemma~\ref{lem:r2-loss-to-W1inf}, we obtain
\[
\sup_{z\in\cB_\ast}
\normH{\widehat S_{\Delta t}^{(M)}(z)-S_{\Delta t}(z)}
\le
C_{\mathrm{emb}}
\bigl(\mathcal L_s^\varrho(\widehat S_{\Delta t}^{(M)})\bigr)^{1/2}
\le
C_{\mathrm{emb}}\sqrt{\omega_{\max}+3}\,\varepsilon_\ast,
\]
and
\[
\sup_{z\in\cB_\ast}
\|D\widehat S_{\Delta t}^{(M)}(z)-DS_{\Delta t}(z)\|_{\mathcal L(H,H)}
\le
C_{\mathrm{emb}}
\bigl(\mathcal L_s^\varrho(\widehat S_{\Delta t}^{(M)})\bigr)^{1/2}
\le
C_{\mathrm{emb}}\sqrt{\omega_{\max}+3}\,\varepsilon_\ast.
\]
By the definition \eqref{eq:r2-eps-star} of \(\varepsilon_\ast\), this yields
\[
\sup_{z\in\cB_\ast}
\normH{\widehat S_{\Delta t}^{(M)}(z)-S_{\Delta t}(z)}
\le
\frac12\min\{\bar\varepsilon_S,\bar\eta_S,1\}
\le
\bar\varepsilon_S,
\]
and
\[
\sup_{z\in\cB_\ast}
\|D\widehat S_{\Delta t}^{(M)}(z)-DS_{\Delta t}(z)\|_{\mathcal L(H,H)}
\le
\frac12\min\{\bar\varepsilon_S,\bar\eta_S,1\}
\le
\bar\eta_S.
\]
These are precisely the bounds
\[
\varepsilon_M^S\le \bar\varepsilon_S,
\qquad
\eta_M^S\le \bar\eta_S.
\]
This proves the claim.
\end{proof}

\begin{proof}[Proof of Corollary~\ref{cor:r2-aot-valid}]
Since \(M\) satisfies \eqref{eq:r2-sample-complexity},
Theorem~\ref{thm:r2-complexity} gives, with probability at least \(1-\delta\), we have
\[
\varepsilon_M^S\le \bar\varepsilon_S,
\qquad
\eta_M^S\le \bar\eta_S.
\]
We work on this high-probability event. Applying
Proposition~\ref{prop:bridge-route2}, we obtain
\[
\delta_M
\le
\frac{\varepsilon_M^S}{\Delta t}+C_{\mathrm{flow}}\Delta t
\le
\bar\delta_{\mathrm S},
\]
and
\[
\ell_M
\le
\frac{\eta_M^S}{\Delta t}+C_{\mathrm{flow}}\Delta t
\le
\bar\ell_{\mathrm S}.
\]
This proves item \emph{(i)}.

The assumptions
\[
\mu>2\bigl(C_{\mathrm{sq}}+\bar\ell_{\mathrm S}\bigr),
\qquad
\mu c_0^2 h^2<\nu
\]
then imply
\[
\gamma_M
:=
\mu-2\bigl(C_{\mathrm{sq}}+\ell_M\bigr)
\ge
\mu-2\bigl(C_{\mathrm{sq}}+\bar\ell_{\mathrm S}\bigr)
=
\gamma_{\mathrm S}>0,
\]
and
\[
\nu_{\mathrm{eff}}
=
\nu-\frac{\mu c_0^2h^2}{2}>0.
\]
Thus the hypotheses of Theorem~\ref{thm:sur-track} are satisfied. Let
\(w=v-u\). By \eqref{eq:sur-track-error}, for every \(t\ge T_\ast\),
\[
\frac{d}{dt}\normH{w}^2
\le
-\gamma_M\normH{w}^2
+
\frac{\delta_M^2}{\lambda_1\nu_{\mathrm{eff}}}.
\]
Using the learning-error bounds through
\(\delta_M\le\bar\delta_{\mathrm S}\) and
\(\gamma_M\ge\gamma_{\mathrm S}\), we obtain the weaker estimate
\[
\frac{d}{dt}\normH{w}^2
\le
-\gamma_{\mathrm S}\normH{w}^2
+
\frac{\bar\delta_{\mathrm S}^2}{\lambda_1\nu_{\mathrm{eff}}},
\qquad t\ge T_\ast.
\]
Solving this scalar differential inequality gives
\[
\|v(t)-u(t)\|_H^2
\le
e^{-\gamma_{\mathrm S}(t-T_\ast)}
\|v(T_\ast)-u(T_\ast)\|_H^2
+
\frac{\bar\delta_{\mathrm S}^2}
{\lambda_1\,\nu_{\mathrm{eff}}\,\gamma_{\mathrm S}}
\Bigl(1-e^{-\gamma_{\mathrm S}(t-T_\ast)}\Bigr),
\]
which is \eqref{eq:r2-aot-valid-bound}. Letting \(t\to\infty\) gives
\eqref{eq:r2-aot-valid-floor}.
\end{proof}

\section{Choice of the nudging parameter}
\label{app:mu-selection-dft}

The tracking theorems identify qualitative conditions under which nudging is effective, but the constants entering those conditions are not explicitly available in practice. Accordingly, we do not interpret the numerical choice of \(\mu\) as approximating a theorem-level optimum. Instead, for each learning route and each observation regime, we choose \(\mu\) through a simple calibration step based on one true trajectory generated from an arbitrary initial state, one random initialization of the nudged state, and one representative feedback resolution from the corresponding experiment.

More precisely, for each candidate \(\mu\) in the prescribed grid, we evolve the surrogate nudged dynamics for this single calibration run and compute the integrated squared tracking error
\[
J(\mu)=\int_0^T \|v_\mu(t)-u(t)\|_2^2\,dt.
\]
We then choose the value of \(\mu\) that minimizes \(J(\mu)\) in this calibration run and keep it fixed in all subsequent multi-initialization summaries. In this way, the choice of \(\mu\) should be viewed as a validation-style selection of a reasonable operating point, rather than as an oracle choice based on the full collection of reported experiments. For the experiments reported here, the resulting values are \(\mu^\ast=25\) and \(\mu^\ast=20\) in the direct vector-field learning route, and \(\mu^\ast=80\) and \(\mu^\ast=40\) in the solution-map learning route, in the noiseless and noisy-observation settings, respectively.

\begin{figure}[tbp]
\begin{subfigure}[t]{0.50\textwidth}
    \centering
    \includegraphics[width=0.75\linewidth]
    {\detokenize{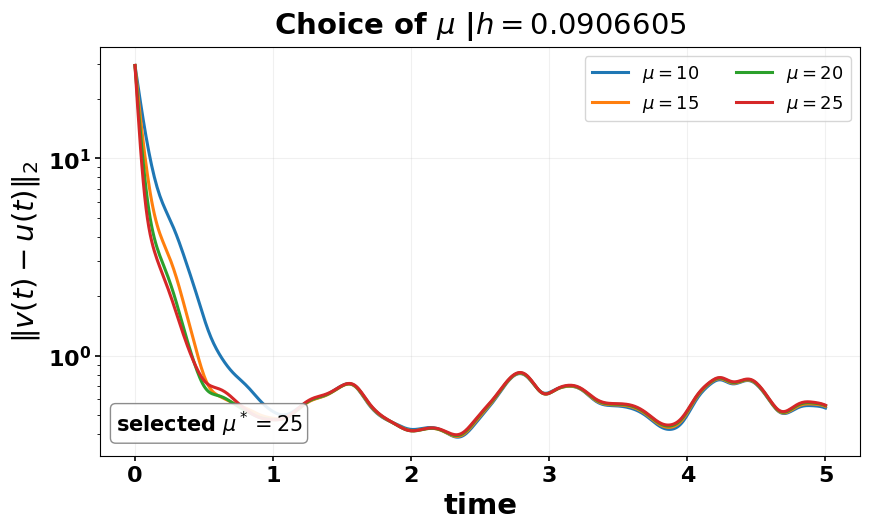}}
    \caption{Choice of \(\mu\) for surrogate AOT.}
\end{subfigure}\hfill
\begin{subfigure}[t]{0.50\textwidth}
    \centering
    \includegraphics[width=0.75\linewidth]
    {\detokenize{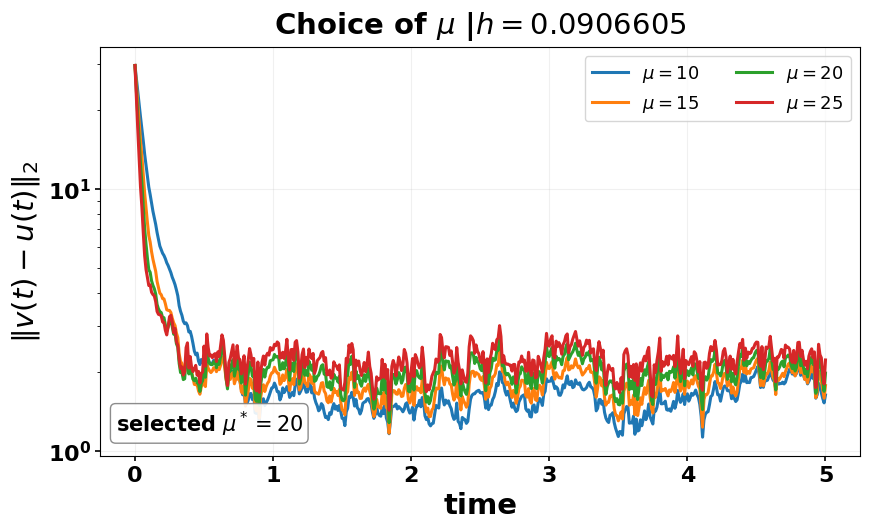}}
    \caption{Choice of \(\mu\) for surrogate AOT with noisy observations.}
\end{subfigure}
\caption{
Choice of the nudging parameter \(\mu\) for surrogate AOT under
band-limited spectral measurements in the direct
vector-field learning route. 
}
\label{fig:numerics-mu-selection-dft-1}
\end{figure}
\begin{figure}[H]
\centering
\begin{subfigure}[t]{0.50\textwidth}
    \centering
    \includegraphics[width=0.75\linewidth]
    {\detokenize{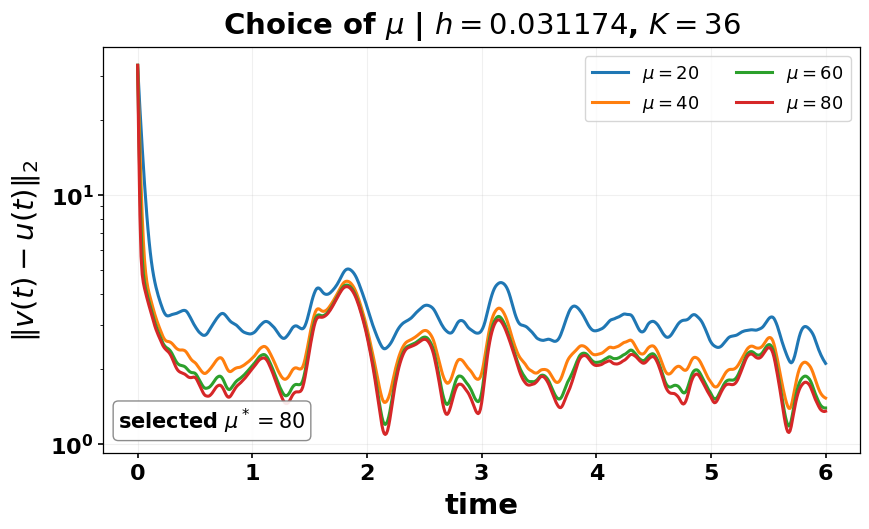}}
    \caption{Choice of \(\mu\) for surrogate AOT.}
\end{subfigure}\hfill
\begin{subfigure}[t]{0.50\textwidth}
    \centering
    \includegraphics[width=0.75\linewidth]
    {\detokenize{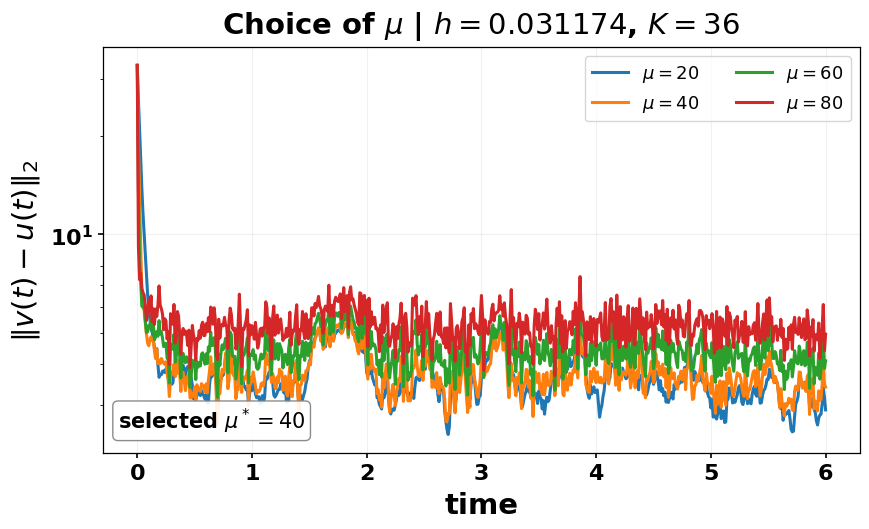}}
    \caption{Choice of \(\mu\) for surrogate AOT with noisy observations.}
\end{subfigure}
\caption{
Choice of the nudging parameter \(\mu\) for surrogate AOT under
band-limited spectral measurements in the solution-map
learning route. 
}
\label{fig:numerics-mu-selection-dft}
\end{figure}
\FloatBarrier

\end{document}